\documentclass[11pt,a4]{article}

\usepackage[utf8]{inputenc}
\usepackage[english,main=russian]{babel}
\usepackage{csquotes}

\usepackage{amsfonts, amsmath, amsthm}
\usepackage{amssymb}
\usepackage{url}
\usepackage{booktabs}

\newcommand{\al}[1]{{\color{black}#1}}

\usepackage[breaklinks=true,colorlinks=true,linkcolor=red, citecolor=blue, urlcolor=blue]{hyperref}%

\usepackage[backend=biber,bibencoding=utf8,sorting=none,style=gost-numeric,language=autobib,autolang=other,clearlang=true,sortcites=true,doi=false,isbn=false,date=year]{biblatex}

\usepackage{mathtools}
\usepackage{makecell}
\usepackage{tikz}
\usepackage{dsfont}
\usepackage{secdot}
\usepackage{longtable}

\usepackage{algorithm}
\usepackage{algorithmic}
\usepackage{float}

\sectiondot{section}
\sectiondot{subsection}
\usepackage{titlesec}
\titleformat{\subsection}
  {\centering\normalfont\normalsize\bfseries\itshape}{\thesubsection.}{1em}{}
  
\addto\captionsrussian{}

\newtheorem{theorem}{Теорема}
\newtheorem{lemma}{Лемма}
\newtheorem{corollary}{Следствие}

\newtheorem{remark}{Замечание}

\newtheorem{assumption}{Предположение}
\allowdisplaybreaks
\usepackage[symbol]{footmisc}

\newcommand{\dotprod}[2]{\langle #1,#2 \rangle}

\definecolor{orange}{HTML}{FF7F00}

\begin{document}

\textbf{УДК} 519.85

\begin{center}
\textbf{Безградиентные методы федеративного обучения с $l_1$ и $l_2$-рандомизацией для задач негладкой выпуклой стохастической оптимизации}\footnote{
Исследование выполнено при поддержке Министерства науки и высшего образования Российской Федерации (госзадание) №075-00337-20-03 (проект 0714-2020-0005).
}









{\bf
\copyright\,2023 г.\,\,
Б.\,А.~Альашкар$^{1}$,
А.\,В.~Гасников$^{1, 2, 3}$,
Д.\,М.~Двинских$^{4}$,
А.\,В.~Лобанов$^{1,5,6}$\footnote{Основной вклад в статью принадлежит Александру~Лобанову~\textless{}lobbsasha@mail.ru\textgreater{}. Согласно правилам журнала, авторы статьи расставлены в \textbf{алфавитном} порядке.}
}

$^{1}$141700 Долгопрудный, М.о., Институтский пер., 9, НИУ МФТИ;

$^{2}$127051 Москва, Б. Каретный пер., 19, стр. 1, ИППИ РАН;

$^{3}$385000 Республика Адыгея, Майкоп, ул. Первомайская, 208, КМЦ АГУ;

$^{4}$109028 Москва, Покровский б-р, 11, НИУ ВШЭ;

$^{5}$109004 Москва, ул. А. Солженицына, 25, ИСП РАН;

$^{6}$125993 Москва, Волоколамское шоссе, 4, МАИ (НИУ).


Поступила в редакцию: 18.11.2022 г.\\
Переработанный вариант  20.05.2023 г.\\
Принята к публикации ??.??.2023 г.
\end{center}

\renewcommand{\abstractname}{\vspace{-\baselineskip}}
\begin{abstract}
  В статье исследуются негладкие задачи выпуклой стохастической оптимизации. С помощью техники сглаживания, базирующейся на замене значения функции в рассматриваемой точке усредненным значением функции по шару (в $l_1$-норме или $l_2$-норме) малого радиуса с центром в этой точке, исходная задача сводится к гладкой задаче (константа Липшица градиента которой обратно пропорционально радиусу шара). Важным свойством используемого сглаживания является возможность вычислять несмещенную оценку градиента сглаженной функции на основе только реализаций исходной функции. Полученную гладкую задачу стохастической оптимизации предлагается далее решать в распределенной архитектуре федеративного обучения (задача решается параллельно: узлы делают локальные шаги, например, стохастического градиентного спуска, потом коммуницируют -- все со всеми, затем все это повторяется).  Цель статьи -- построить на базе современных достижений в области безградиентной негладкой оптимизации и в области федеративного обучения безградиентные методы решения задач негладкой стохастической оптимизации в архитектуре федеративного обучения.
\end{abstract}

\textbf{Ключевые слова: }  безградиентные методы, методы с неточным оракулом,   федеративное обучение.

\section{Введение}\label{section:introduction}
В статье рассматриваются задачи стохастической оптимизации
\begin{equation}
    \label{Convex_opt_problem0}
    \min_{x \in Q \subseteq \mathbb{R}^{d}}  f(x) := \mathbb{E}_\xi \left[ f(x,\xi) \right]
\end{equation}
и их седловые обобщения. При этом считается, что функция $f(x)$ -- выпуклая, негладкая и имеет ограниченную константу Липшица. Также предполагается, что для наблюдения доступна реализация функции $f(x,\xi)$, но не ее градиент (по $x$) $\nabla f(x,\xi)$. Более того, в работе также рассматривается случай, когда эта реализация доступна с небольшим шумом неслучайной природы. Данная статья основывается на работе \cite{Gasnikov_2022_Smooth_Scheme}, в которой предложен оптимальный алгоритм (с точностью до логарифмических по размерности пространства множителей в оценках оракульных вызовов) решения задачи \eqref{Convex_opt_problem0} сразу по трем критериям: 1) число оракульных вызовов (вычислений $f(x,\xi)$), 2) число последовательных итераций метода (на одной итерации можно вызывать оракул многократно), 3) максимально допустимый уровень (потенциально враждебного) шума, при котором все еще возможно достичь желаемой точности. В основе подхода \cite{Gasnikov_2022_Smooth_Scheme} лежит довольно старая идея (см., например, \cite{Nemirovski_1979}) замены исходной негладкой функции $f(x)$ ее сглаженной версией $f_{\gamma}(x) = \mathbb{E}_{\Tilde{e}} \left[ f(x + \gamma \Tilde{e})\right]$, где $\Tilde{e}$ равномерно распределенный случайный вектор на $B_2^d(1)$ -- единичном евклидовом шаре  с центром в нуле. Для сглаженной функции конечная аппроксимация (с шагом $\gamma$) производной по случайному направлению (выбранному равновероятно на единичной евклидовой сфере) будет несмещенной (рандомизированной) оценкой градиента, обладающей при использовании симметричной разностной аппроксимации хорошей дисперсией \cite{Shamir_2017} -- такой же, как если бы $f(x)$ была гладкой.  

 В другой работе \cite{Tsybakov_2022} показано, что если использовать сглаживание не по евклидовому шару, а по шару в $l_1$-норме, то в определенных ситуациях можно улучшить оценки~\cite{Gasnikov_2022_Smooth_Scheme} на число оракульных вызовов на логарифмический множитель (по размерности пространства). Однако при этом в \cite{Tsybakov_2022} не обсуждался вопрос о числе последовательных итераций и использовалась другая (более узкая) концепция враждебного шума, для которой вопрос о точных нижних оценках, насколько нам известно, по-прежнему остается~открытым.
 
В настоящей статье показано, как можно в теории улучшить результаты работы \cite{Gasnikov_2022_Smooth_Scheme} за счет схемы сглаживания из \cite{Tsybakov_2022} (см. также работу \cite{Gasnikov_2016}, в которой данная схема сглаживания также предлагалась, но анализ был проведен менее точно). Как уже отмечалось~выше, улучшения носят логарифмический масштаб. Отметим при этом, что проведенные в работе численные эксперименты не смогли явно уловить эффект такого улучшения~оценок.

Другим важным направлением развития работы \cite{Gasnikov_2022_Smooth_Scheme} стало обобщение всех результатов~этой работы (и ее модификации со сглаживанием на шаре в $l_1$-норме) на распределенные алгоритмы в архитектуре федеративного обучения \cite{kairouz2021advances}. Отметим, что в федеративном обучении рассматриваются задачи типа \eqref{Convex_opt_problem0}, во-первых, гладкие, во-вторых, с~полноградиенным (стохастическим) оракулом. Задача \eqref{Convex_opt_problem0} решается независимо в каждом узле, потом узлы коммуницируют и вычисляют среднее (по узлам), затем снова начинают независимо решать задачу, стартуя с этого среднего. Через некоторое время, узлы снова коммуницируют, считают среднее и процесс повторяется... В данной работе за счет использования схемы сглаживания (с введением дополнительной рандомизации) получается свести негладкую задачу стохастической оптимизации с безградиентным оракулом, к гладкой задаче стохастической оптимизации с оракулом, выдающим стохастический градиент, в котором стохастика формируется из изначальной стохастики, присущей исходной постановке задачи, и стохастике, возникшей при сглаживании (рандомизированной). Оказалось, что в существующей сейчас литературе практически отсутствуют методы федеративного обучения для негладких задач. По-видимому, это было связано с тем, что для негладких задач в общем случае невозможно осуществлять батч-параллелизацию (параллелизацию по $\xi$). Однако для безградиентного оракула батч-параллелизация возможна за счет появления дополнительной (рандомизированной) случайности в виде случайного направления \cite{Gasnikov_2022_Smooth_Scheme}! Собственно за счет этого и удается перенести результаты работы \cite{Gasnikov_2022_Smooth_Scheme} на архитектуру федеративного обучения. Насколько нам известно результатам, полученным в данном направлении, сейчас нет конкурентов, поэтому здесь не приводится литературный обзор конкурирующих работ.  
 

\section{Основной вклад и структура}
 Основной вклад статьи состоит в следующем.
\begin{itemize}
    
    \item Приводим подробное описание двух схем сглаживания (параллельно): с $l_1$ и $l_2$-рандомизациями. Находим константу Липшица градиента $L_{f_\gamma}$ для $l_1$-рандомизации. Для явного описания оценки второго момента в $l_2$-рандомизации с двухточечной обратной связью находим константу $c$, которая в исходной статье \cite{Shamir_2017}, из которой эта константа была взята, так и не была посчитана. Получаем оценку дисперсии (второго момента) для одноточечного случая $l_1$-рандомизации. Показываем, что $l_1$-рандомизация с одноточеным оракулом также превосходит $l_2$-рандомизацию с точностью до $\ln d$, как и с двухточечным оракулом.
    
    \item Получаем оптимальные верхние оценки скорости сходимости пробатченных алгоритмов первого порядка Minibatch SMP и Single-Machine SMP для решения задач оптимизации с седловой точкой в архитектуре федеративного обучения.
    
    \item 
    Описываем технику генерирования безградиентных алгоритмов (решения
    седловых задач и задач выпуклой оптимизации), оптимальных по количеству коммуникационных раундов $N$, максимальному значению допустимого шума $\Delta$ и оракульной сложности $T$. Показываем, что одноточечные и двухточечные алгоритмы в архитектуре федеративного обучения, использующие $l_1$-рандомизацию, работают лучше алгоритмов, использующих $l_2$-рандомизацию, с точностью до логарифма в $l_1$-норме. Сравниваем одноточечные алгоритмы с двухточеными и показываем, что для решения с $\varepsilon$ точностью (по функции) одноточечные алгоритмы требуют в $O \left( d/\varepsilon^2 \right)$ больше обращений к оракулу, чем двухточечные. Анализируем работу алгоритма Minibatch Accelerated SGD, используя разные схемы сглаживания, на практическом эксперименте.
\end{itemize}
Структура статьи следующая: в разд.~\ref{section:Smooth_schemes}~и~\ref{section:Federated_Learning} представлено краткое введение в техники сглаживания и федеративное обучение соответственно. В п.~\ref{subsection:Gradient_Free_Algorithms} приводится главный результат работы. В разд.~\ref{section:Elements_for_proof} приведены основные идеи  доказательств теорем~\ref{theorem_1},\ref{theorem_2}. Численные эксперименты представлены в разд.~\ref{section:Numerical_result}.


 \section{Схемы сглаживания}\label{section:Smooth_schemes}
 
Схема сглаживания позволяет создавать безградиентные методы решения негладких задач, модифицируя одноименные алгоритмы первого порядка, предназначенные для решения гладких задач. Впервые схема сглаживания была описана в книге \cite{Nemirovski_1979}, где представлена идея решения задач методами первого порядка, используя вместо стохастического оракула первого порядка безградиентный стохастический оракул, который получен через теорему Стокса. С тех пор уже придумали различные техники сглаживания, однако основная идея восходит от \cite{Nemirovski_1979}. В данном разделе представим параллельно две схемы сглаживания: с $l_1$-рандомизацией \cite{Tsybakov_2022,Gasnikov_2016} и с $l_2$-рандомизацией \cite{Gasnikov_2022_Smooth_Scheme}, включающую стохастическую оптимизацию и смещенную оценку безградиентного оракула. Для этого рассматривается стохастическая негладкая выпуклая задача оптимизации 
\begin{equation}
    \label{Convex_opt_problem}
    \min_{x \in Q \subseteq \mathbb{R}^{d}} \left\{ f(x) := \mathbb{E}_\xi \left[ f(x,\xi) \right] \right\},
\end{equation}
где $Q$~--- выпуклое и компактное множество и $f(x, \xi)$~--- выпуклая функция на множестве $Q_\gamma := Q + B_p^d(\gamma)$. Здесь предполагается, что безградиентный оракул возвращает значение функции $f(x)$, возможно, с некоторым враждебным шумом $\delta(x)$: $\;f_\delta(x) := f(x) + \delta(x)$.

\subsection{Обозначения и предположения}\label{subsection:SS_1}
Обозначим через
$\dotprod{x}{y}:= \sum_{i=1}^d x_i y_i$ стандартное скалярное произведение $x,y \in \mathbb{R}^d$. Через $\| x \|_p := \left( \sum_{i=1}^d |x_i|^p \right)^{1/p}$ обозначаем $l_p$-норму ($p\ge 1$) в $\mathbb{R}^d$. $B_p^d(r):=\left\{ x \in \mathbb{R^d} : \| x \|_p \leq r \right\}$ и $S_p^d(r):=\left\{ x \in \mathbb{R^d} : \| x \|_p = r \right\}$ определяют $l_p$-шар и $l_p$-сферу соответственно. Объем $l_p$-шара определяется через $V(B_p^d (r)):= c_d r^d = \frac{\left( 2 \Gamma \left( \frac{1}{p} + 1 \right) \right)^d}{\Gamma \left( \frac{d}{p} + 1 \right)} r^d$, где $\Gamma(\cdot)$ обозначает гамма-функцию. Для обозначения <<расстояния>> между начальной точкой $x^0$ и решением исходной задачи $x_*$ вводим $R := \Tilde{O} \left( \| x^0 - x_* \|_p \right)$, где $\Tilde{O}(\cdot)$ является $O(\cdot)$ с точностью до логарифма $\sqrt{\ln d}$.

\begin{assumption}\label{assumption:SS_1}
    (липшицева непрерывная функция). Функция $f(x, \xi)$ является $M$-липшицевой непрерывной функцией в $l_p$-норме, то есть для всех $x,y \in Q$ имеем
    \begin{equation*}
        | f(y,\xi) - f(x, \xi) | \leq M(\xi) \| y - x \|_p.
    \end{equation*}
    Более того, существует положительная константа $M$, которая определяется следующим образом: $\mathbb{E} \left[ M^2(\xi) \right] \leq M^2$. В частности, для $p=2$ используем обозначение $M_2$ для константы Липшица.
\end{assumption}
\begin{assumption}\label{assumption:SS_2}
    (Ограниченность шума). Для всех $x \in Q$ выполняется $| \delta(x) | \leq \Delta $, где $\Delta$ -- уровень неточности (шума).
\end{assumption}
\begin{assumption}\label{assumption:SS_3}
    Для всех $x \in Q$ выполняется $\mathbb{E}_{\xi} \left[ | f(x,\xi) |^2 \right] \leq G^2$.
\end{assumption}


\subsection{Гладкая аппроксимация}\label{subsection:SS_2}

Поскольку задача \eqref{Convex_opt_problem} является негладкой, то вводим следующую гладкую аппроксимацию негладкой функции
\begin{equation}
    \label{Smooth_Convex_opt_problem}
    f_\gamma(x) := \mathbb{E}_{\Tilde{e}} \left[ f(x + \gamma \Tilde{e})\right],
\end{equation}
где $\gamma > 0$, $ \Tilde{e}$~--- случайный вектор, равномерно распределенный на $B_p^d(1)$ (далее ограничимся рассмотрением случаев $p=1$ и $p=2$). Здесь $f(x):= \mathbb{E} f(x, \xi)$. 

Следующая лемма представляет совокупность свойств аппроксимации функции $f$ в зависимости от распределения вектора $ \Tilde{e}$. Основываясь на \cite{Gasnikov_2022_Smooth_Scheme, Tsybakov_2022, Duchi_2015} и дополняя найденной константой Липшица градиента $L_{f_\gamma}$ в случае $p = 1$, выпишем свойства функции $f_\gamma$. 

\setcounter{table}{100}

\begin{lemma}
\label{properties_f_gamma}
    Для всех $x,y \in Q$ c предположением \ref{assumption:SS_1} справедливо 
    
        \begin{longtable}{ | c | c | }
            \hline
            $\Tilde{e} \in RB_1^d(1)$ & $\Tilde{e} \in RB_2^d(1)$  \\ [1ex] 
            \hline
            \hline
            \endhead
            & \\
            $
                f(x) \leq f_\gamma(x) \leq f(x) + \frac{2}{\sqrt{d}} \gamma M_2
            $ &
            $
                 f(x) \leq f_\gamma(x) \leq f(x) + \gamma M_2
            $ \\ [2ex]
            \hline
            \hline
            \pagebreak
            \multicolumn{2}{| l |}{$f_\gamma$~--- $M$-липшицева:}\\ [1ex]
            \hline
            & \\
            $
                | f_\gamma(y) - f_\gamma(x) | \leq M \| y - x \|_{\al{p}}
            $ & 
            $
                | f_\gamma(y) - f_\gamma(x) | \leq M \| y - x \|_p
            $ \\ [2ex]
            \hline
            \hline
            \multicolumn{2}{| l |}{$f_\gamma$ имеет $L_{f_{\gamma}}$-липшицевый градиент:}\\ [1ex]
            \hline
            & \\
            $
                \| \nabla f_\gamma(y) - \nabla f_\gamma(x) \|_{\al{q}} \leq \underbrace{\frac{d M}{\al{2} \gamma}}_{L_{f_{\gamma}}} \| y - x \|_{\al{p}}
            $ 
             & 
            $
                \| \nabla f_\gamma(y) - \nabla f_\gamma(x) \|_q \leq \underbrace{\frac{\sqrt{d} M}{\gamma}}_{L_{f_{\gamma}}} \| y - x \|_p 
            $ \\ [2ex]
            \hline
            \hline
        \end{longtable}
где $q$ такой, что $1/p + 1/q = 1$.
\end{lemma}

Доказательство приведено в п.~\ref{subsection:5_1}.


\subsection{Рандомизация с двухточечной обратной связью}\label{subsection:SS_3}
Аппроксимация градиента зашумленной функции $f_\gamma(x, \xi)$ из  \eqref{Smooth_Convex_opt_problem} может быть получена через две точки, близких к $x$. Для этого определим случайный вектор $e$, равномерно распределенный на $S_p^d(1)$. Тогда градиент может быть оценен следующей аппроксимацией.
\begin{itemize}
    \item Оценка градиента для $l_1$-рандомизации ($e \in RS_1^d(1)$) \cite{Gasnikov_2016} (см. также \cite{Tsybakov_2022}):
    \begin{equation}
    \label{grad_l1_randomization}
        \nabla f_\gamma (x,\xi,e) = \frac{d}{2\gamma}(f_\delta(x+\gamma e,\xi) - f_\delta(x-\gamma e,\xi)) \text{sign}(e).
    \end{equation}

    \item Оценка градиента для $l_2$-рандомизации  ($e \in RS_2^d(1)$) \cite{Shamir_2017}:
    \begin{equation}
    \label{grad_l2_randomization}
        \nabla f_\gamma (x,\xi,e) = \frac{d}{2\gamma}(f_\delta(x+\gamma e,\xi) - f_\delta(x-\gamma e,\xi)) e.
    \end{equation}
\end{itemize}
Для оценки градиента в \eqref{grad_l1_randomization} и \eqref{grad_l2_randomization} выбрана центральная конечно-разностная  схема рандомизации, так как в работе \cite{Scheinberg_2022} поясняется, что в гладком случае выгоднее оценивать градиент именно центральной конечно-разностной схемой (CFD), а не прямой конечно-разностной схемой (FFD). Отметим, что при $\Delta = 0$ оценки будут несмещенными, т.е. $\mathds{E}_{e,\xi} \left[ \nabla f_\gamma (x, \xi, e)  \right] = \nabla f_\gamma (x)$. 

Далее приведем свойства $\nabla f_\gamma (x,\xi,e)$ для двух рандомизаций, используя известные результаты из \cite{Gasnikov_2022_Smooth_Scheme, Shamir_2017, Tsybakov_2022, Gasnikov_2016, Beznosikov_2020, Ledoux_2005}. Во многих работах для $l_2$-рандомизации оценка второго момента записывается с точностью до константы $c$, поэтому в лемме \ref{properties_grad_f_gamma} приводятся оценки второго момента для $l_1$ и $l_2$-рандомизаций с уточнением константы $c$.
\begin{lemma}
    \label{properties_grad_f_gamma}
    Для всех $x \in Q$ с предположениями \ref{assumption:SS_1} и \ref{assumption:SS_2} имеем
    \begin{itemize}
        \item[i)] $\nabla f_\gamma (x,\xi,e)$ с $l_1$-рандомизации (\ref{grad_l1_randomization}) имеет оценку дисперсии (второй момент):
        \begin{equation*}
            \mathds{E}_e \left[ \| \nabla f_\gamma (x, \xi, e) \|^2_q \right] \leq \kappa(p,d) \left( M_2^2 + \frac{d^2 \Delta^2}{12(1+\sqrt{2})^2 \gamma^2} \right),
        \end{equation*}
        где $1/p + 1/q = 1$ и 
        \begin{equation*}
            \kappa(p,d)  = 48(1+\sqrt{2})^2 d^{2 - \frac{2}{p}}.
        \end{equation*}
        Если $\Delta$ достаточно мала, то 
        \begin{equation}
        \label{sigma_for_l1_randomization}
            \mathds{E}_e \left[ \| \nabla f_\gamma (x, \xi, e) \|^2_q \right] \leq 2 \kappa(p,d) M_2^2.
        \end{equation}
        
        \item[ii)] $\nabla f_\gamma (x,\xi,e)$ с $l_2$-рандомизации (\ref{grad_l2_randomization}) имеет оценку дисперсии (второй момент):
        \begin{equation*}
            \mathds{E}_e \left[ \| \nabla f_\gamma (x, \xi, e) \|^2_q \right] \leq \kappa(p,d) \left(  M_2^2 + \frac{d^2 \Delta^2}{\sqrt{2} \gamma^2}  \right), 
        \end{equation*}
        где $1/p + 1/q = 1$ и 
        \begin{equation*}
            \kappa(p,d) = \sqrt{2} \min \left\{ q, \ln d \right\} d^{2-\frac{2}{p}}.
        \end{equation*}
        Если $\Delta$ достаточно мала, то 
        \begin{equation}
        \label{sigma_for_l2_randomization}
            \mathds{E}_e \left[ \| \nabla f_\gamma (x, \xi, e) \|^2_q \right] \leq 2 \kappa(p,d) M_2^2. 
        \end{equation}
        
    \end{itemize}
\end{lemma}
Доказательство приведено в п.~\ref{proof_properties_grad_f_gamma}.


\subsection{Рандомизация с одноточечной обратной связью}\label{subsection:SS_4}

Для случая когда двухточечная обратная связь недоступна, схемы сглаживания могут использовать одноточечную обратную связь через следующую несмещенную оценку.
\begin{itemize}
    \item Оценка градиента для $l_1$-рандомизации ($e \in RS_1^d(1)$):
    \begin{equation}
    \label{grad_l1_randomization_one_point}
        \nabla f_\gamma (x,\xi,e) = \frac{d}{\gamma}(f_\delta(x+\gamma e,\xi)) \text{sign}(e).
    \end{equation}

    \item Оценка градиента для $l_2$-рандомизации  ($e \in RS_2^d(1)$) \cite{Gasnikov_2022_Smooth_Scheme}:
    \begin{equation}
    \label{grad_l2_randomization_one_point}
        \nabla f_\gamma (x,\xi,e) = \frac{d}{\gamma}(f_\delta(x+\gamma e,\xi)) e.
    \end{equation}
\end{itemize}

Тогда свойства $\nabla f_\gamma (x,\xi,e)$ для двух рандомизаций будут иметь следующий вид.
\begin{lemma}
    \label{properties_grad_f_gamma_one_point}
    Для всех $x \in Q$ с предположениями \ref{assumption:SS_2} и \ref{assumption:SS_3} имеем
    \begin{itemize}
        \item[i)] $\nabla f_\gamma (x,\xi,e)$ с $l_1$-рандомизации (\ref{grad_l1_randomization_one_point}) имеет оценку дисперсии (второй момент):
        \begin{equation*}
            \mathds{E}_e \left[ \| \nabla f_\gamma (x, \xi, e) \|^2_q \right] \leq \kappa(p,d) \left( \frac{G^2}{\gamma^2} + \frac{\Delta^2}{\gamma^2} \right),
        \end{equation*}
        где $1/p + 1/q = 1$ и 
        \begin{equation*}
            \kappa(p,d) = d^{4-\frac{2}{p}}.
        \end{equation*}
        Если $\Delta$ достаточно мала, то 
        \begin{equation}
        \label{sigma_for_l1_randomization_one_point}
            \mathds{E}_e \left[ \| \nabla f_\gamma (x, \xi, e) \|^2_q \right] \leq 2 \kappa(p,d) \frac{G^2}{\gamma^2}.
        \end{equation}
        
        \item[ii)] $\nabla f_\gamma (x,\xi,e)$ с $l_2$-рандомизации (\ref{grad_l2_randomization_one_point}) имеет оценку дисперсии (второй момент):
        \begin{equation*}
            \mathds{E}_e \left[ \| \nabla f_\gamma (x, \xi, e) \|^2_q \right] \leq \kappa(p,d) \left(  \frac{G^2}{\gamma^2} + \frac{\Delta^2}{\gamma^2}  \right), 
        \end{equation*}
        где $1/p + 1/q = 1$ и 
        \begin{equation*}
            \kappa(p,d) = \min \left\{ q, \ln d \right\} d^{3-\frac{2}{p}}.
        \end{equation*}
        Если $\Delta$ достаточно мала, то 
        \begin{equation}
        \label{sigma_for_l2_randomization_one_point}
            \mathds{E}_e \left[ \| \nabla f_\gamma (x, \xi, e) \|^2_q \right] \leq 2 \kappa(p,d)  \frac{G^2}{\gamma^2}. 
        \end{equation}
        
    \end{itemize}
\end{lemma}
Доказательство приведено в п.~\ref{proof_properties_grad_f_gamma_one_point}.


\subsection{Алгоритм сглаживания}\label{subsection:SS_5}

Основываясь на выше представленных элементах, опишем общий подход, именуемый, как схема сглаживания. Предположим, что у нас есть некоторый ускоренный пакетный алгоритм \textbf{A}$(L, \sigma^2)$ архитектуры федеративного обучения, который решает задачу \eqref{Convex_opt_problem} с предположением, что $f$ является гладкой и удовлетворяет
\begin{equation*}
    \| \nabla f(y) - \nabla f(x) \|_q \leq L \| y - x \|_p, \; \forall x,y \in Q_\gamma
\end{equation*}
и с помощью стохастического оракула первого порядка, который зависит от случайной величины $\eta$ и возвращает в точке $x$ смещенный стохастический градиент $\nabla_x f_\gamma (x, \eta)$
\begin{equation*}
    \mathbb{E}_{\eta} \left[ \| \nabla_x f_\gamma (x, \eta) - \nabla f(x) \|_q^2 \right] \leq \sigma^2.
\end{equation*}
Тогда общий подход схемы сглаживания состоит в применении \textbf{A}$(L, \sigma^2)$ к сглаженной задаче
\begin{equation}
    \label{Convex_opt_problem_FL}
    \min_{x \in Q \subseteq \mathbb{R}^{d}} f_\gamma(x)
\end{equation}
для решения с $\varepsilon/2$ точностью c известными параметрами $\eta = e$, $\nabla_x f_\gamma (x, \eta) = \nabla f_\gamma (x, \xi, e)$, $L = L_{f_\gamma}$ и $\gamma$, представленная в следствии \ref{gamma_parameter_clear}.

\begin{corollary}\label{gamma_parameter_clear}
    Согласно лемме \ref{properties_f_gamma}, чтобы получить $\varepsilon$-точность решения задачи \eqref{Convex_opt_problem}, необходимо решить задачу \eqref{Convex_opt_problem_FL} с $\varepsilon / 2$-точностью со следующим параметром
    \begin{longtable}{ | c | c | }
            \hline
            Схема сглаживания с $l_1$-рандомизацией & Схема сглаживания с $l_2$-рандомизацией  \\ [1ex] 
            \hline
            \hline
            \endhead
            & \\
            $
                \gamma = \frac{\sqrt{d} \varepsilon}{4 M_2}
            $ &
            $
               \gamma = \frac{\varepsilon}{2 M_2}
            $ \\ [2ex]
            \hline
            \hline
        \end{longtable}
\end{corollary} где $\varepsilon > 0$ является желаемой точностью решения задачи \eqref{Convex_opt_problem} с точки зрения субоптимальности: $\mathds{E}[f(x^{N}) - f(x_{*})] \leq \varepsilon$.\\

\begin{corollary}\label{Const_Lipshitz_grad}
    Согласно лемме \ref{properties_f_gamma}, подставляя параметр $\gamma$ из следствия \ref{gamma_parameter_clear}, имеем
    \begin{longtable}{ | c | c | }
            \hline
            Схема сглаживания с $l_1$-рандомизацией & Схема сглаживания с $l_2$-рандомизацией  \\ [1ex] 
            \hline
            \hline
            \endhead
            & \\
            $
                L_{f_\gamma} = \frac{\al{2} \sqrt{d} M M_2}{\varepsilon}
            $ &
            $
               L_{f_\gamma} = \frac{2 \sqrt{d} M M_2}{\varepsilon}
            $ \\ [2ex]
            \hline
            \hline
    \end{longtable}
\end{corollary}

\begin{corollary}\label{sigma_for_two_point}
    Согласно лемме \ref{properties_grad_f_gamma} уравнения \eqref{sigma_for_l1_randomization} и \eqref{sigma_for_l2_randomization} для двухточечного оракула примут вид
    \begin{longtable}{ | c | c | }
            \hline
            Схема сглаживания с $l_1$-рандомизацией & Схема сглаживания с $l_2$-рандомизацией  \\ [1ex] 
            \hline
            \hline
            \endhead
            & \\
            $
                \sigma^2 \leq 48(1+\sqrt{2})^2 d^{2 - \frac{2}{p}} M_2^2
            $ &
            $
               \sigma^2 \leq 2 \sqrt{2} \min \left\{ q, \ln d \right\} d^{2 - \frac{2}{p}} M_2^2
            $ \\ [2ex]
            \hline
            \hline
    \end{longtable}
    если $\Delta$ достаточно мала.
\end{corollary}

\begin{corollary}\label{sigma_for_one_point}
    Согласно лемме \ref{properties_grad_f_gamma_one_point} уравнения \eqref{sigma_for_l1_randomization_one_point} и \eqref{sigma_for_l2_randomization_one_point} для одноточечного оракула примут вид
    \begin{longtable}{ | c | c | }
            \hline
            Схема сглаживания с $l_1$-рандомизацией & Схема сглаживания с $l_2$-рандомизацией  \\ [1ex] 
            \hline
            \hline
            \endhead
            & \\
            $
                \sigma^2 \leq 32 d^{3 - \frac{2}{p}} \frac{G^2 M_2^2}{\varepsilon^2}
            $ &
            $
               \sigma^2 \leq 8 \min \left\{ q, \ln d \right\} d^{3 - \frac{2}{p}} \frac{G^2 M_2^2}{\varepsilon^2}
            $ \\ [2ex]
            \hline
            \hline
    \end{longtable}
    если $\Delta$ достаточно мала.
\end{corollary}

\begin{remark}
    \label{remark_convex_concave}
    Если вместо стохастической негладкой выпуклой задачи оптимизации \eqref{Convex_opt_problem} рассматривать стохастическую негладкую выпукло-вогнутую задачу с седловой точкой
    \begin{equation*}
        \min_{x \in Q_x \subseteq \mathbb{R}^{d_x}} \max_{y \in Q_y \subseteq \mathbb{R}^{d_y}} \left\{ f(x,y):= \mathbb{E} \left[ f(x,y,\xi) \right] \right\},
    \end{equation*}
    то применяя схему сглаживания, описанную выше в данном разделе, отдельно к $x$ и $y$ переменным, получим те же результаты, что и для выпуклой оптимизации при соответствующих изменениях $f(x,\xi)$ на $f(z,\xi)$, где $z := (x, y), z \in Q_z:= Q_x \times Q_y$, за исключением одного пункта в лемме \ref{properties_f_gamma}:
    \begin{itemize}
        \item ($\Tilde{e} \in RB_1^d(1)$) вместо
        \begin{equation*}
            f(x) \leq f_\gamma(x) \leq f(x) + \frac{2}{\sqrt{d}} \gamma M_2
        \end{equation*}
        имеем 
        \begin{equation*}
            f(x, y) - \frac{2}{\sqrt{d}} \gamma_y M_{2,y} \leq f_\gamma(x, y) \leq f(x, y) + \frac{2}{\sqrt{d}} \gamma_x M_{2,x};
        \end{equation*}
        
        \item ($\Tilde{e} \in RB_2^d(1)$) вместо
        \begin{equation*}
            f(x) \leq f_\gamma(x) \leq f(x) + \gamma M_2
        \end{equation*}
        имеем 
        \begin{equation*}
            f(x, y) - \gamma_y M_{2,y} \leq f_\gamma(x, y) \leq f(x, y) + \gamma_x M_{2,x},
        \end{equation*}
    \end{itemize}
    
    где $\gamma = (\gamma_x, \gamma_y)$, а $M_{2,x}$ и $M_{2,y}$~--- соответствующие константы Липшица в $l_2$-норме. 
\end{remark}



\section{Федеративное обучение}
\label{section:Federated_Learning}
 Архитектура распределенного обучения выглядит следующим образом: между «компьютерами» распределяется набор данных, каждый компьютер делает одно локальное обновление (локальный шаг, например шаг стохастического градиентного спуска), после чего происходит глобальное обновление (глобальный шаг, коммуникация всех компьютеров со всеми), далее цепочка локально-глобальных обновлений повторяется. Однако, глобальное обновление, например при большом размере данных, тратит много вычислительных ресурсов, в отличие от локального обновления. Тогда вводится архитектура федеративного обучения, представленная на фиг. \ref{Architecture_FL},
\begin{figure}[ht]
    \centering
    \includegraphics[width=0.5\textwidth]{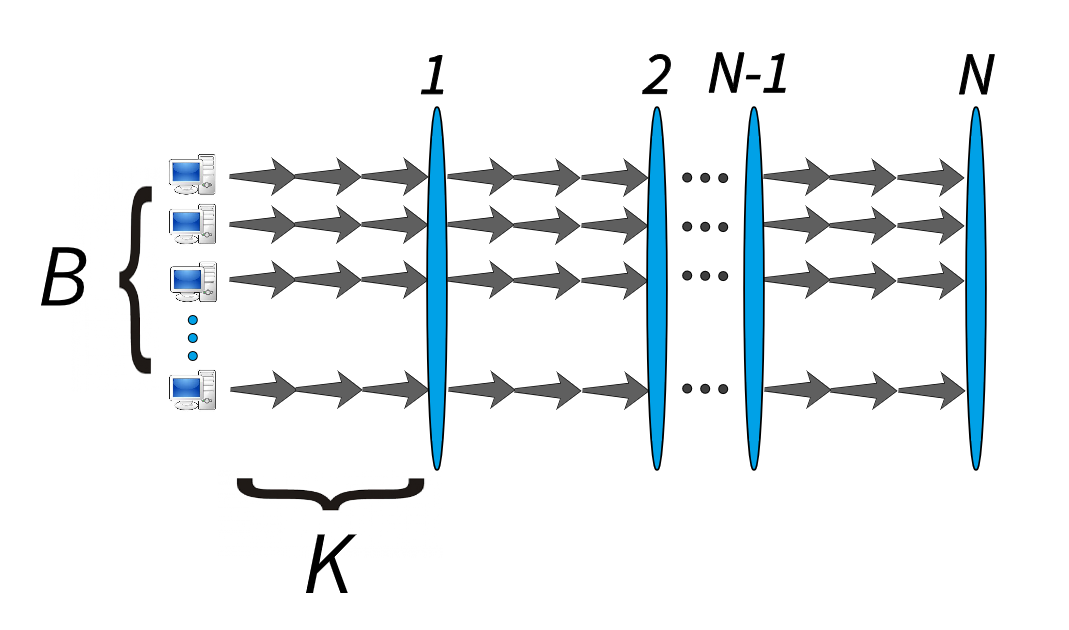}
    \caption{Архитектура федеративного обучения.}
    \label{Architecture_FL}
\end{figure} где в гомогенном случае $B$ компьютеров делают параллельно $K$ локальных обновлений перед каждым коммуникационном раундом, общее число которых составляет $N$. Таким образом, $N K$ является общим числом итераций алгоритма, а $T = N \cdot K \cdot B$ является общим числом вызовов стохастического градиента.

\subsection{Оптимальные алгоритмы первого порядка}\label{subsection:First_order_methods_for_FL}
В данной статье остановимся и рассмотрим класс пробатченных ускоренных методов первого порядка, а именно Minibatch Accelerated SGD (Mb-Ac-SGD) и Single-Machine Accelerated SGD (SM-Ac-SGD) из \cite{Woodworth_2021}, Local-AC-CA из \cite{Woodworth_2020} и FedAc из  \cite{Yuan_Ma_2020}, результаты скорости сходимости которых представлены в таблице \ref{Convergence_rates}. В таблице были использованы следующие обозначения:  $R$--$ \| x^0 - x_* \|_2$; $B$-- количество работающих компьютеров; $K$-- число локальных обновлений; $N$-- количество коммуникационных раундов; $L$-- гладкость.

\setcounter{table}{0}

\begin{table}[ht] 

    \caption{\label{Convergence_rates} Результаты скорости сходимости}
    \begin{tabular}{lll}
    \toprule
    Алгоритм &  $\mathbb{E} \left[ f(\cdot) \right] - f^*  \lesssim \dots$ & Ссылка \\
    \midrule
    Mb-Ac-SGD & $\frac{LR^2}{N^2} + \frac{\sigma R}{\sqrt{BNK}}$ & (Woodworth и др., 2021) \cite{Woodworth_2021} \\
    SM-Ac-SGD & $\frac{LR^2}{N^2 K^2} + \frac{\sigma R}{\sqrt{NK}}$ & (Woodworth и др., 2021) \cite{Woodworth_2021} \\
    Local-AC-CA & $\frac{LR^2}{N^2 K^2} + \frac{\sigma R}{\sqrt{BNK}}$ & (Woodworth и др., 2020) \cite{Woodworth_2020} \\
    FedAc & $\frac{LR^2}{N^2 K} + \frac{\sigma R}{\sqrt{BNK}} + \min \left\{ \frac{L^{\frac{1}{3}} \sigma^{\frac{2}{3}} R^{\frac{4}{3}}}{N K^{\frac{1}{3}}}, \frac{L^{\frac{1}{2}} \sigma^{\frac{1}{2}} R^{\frac{3}{2}}}{N K^{\frac{1}{4}}} \right\}$ & (Yuan, Ma, 2020) \cite{Yuan_Ma_2020} \\
    \midrule
    Mb-SMP & $\max \left\{\frac{ LR^2}{N}, \frac{ \sigma R}{\sqrt{BNK}} \right\}$ & Приложение \ref{Upper_bound} \\
    SM-SMP & $\max \left\{\frac{ LR^2}{N K}, \frac{ \sigma R}{\sqrt{NK}} \right\}$ & Приложение \ref{Upper_bound} \\
    \bottomrule
    \end{tabular}
\end{table}

Для квадратичной целевой функции было доказано, что $K > 1$ локальное обновление приводит к оптимальным оценкам скорости сходимости (алгоритм Local-AC-CA из \cite{Woodworth_2020}). Алгоритм FedAc из \cite{Yuan_Ma_2020} обобщает результаты статьи \cite{Woodworth_2020} на случай выпуклых функций. 

Но уже в 2021 году были получены оптимальные оценки скорости сходимости в статье \cite{Woodworth_2021}. В данной статье утверждается, что оптимальные оценки можно получить только в двух случаях. Первый случай (алгоритм Minibatch Accelerated SGD) предполагает, что каждый компьютер выполняет одно локальное обновление перед коммуникационном раундом. Второй случай (алгоритм Single-Machine Accelerated SGD) предполагает, что работает только один компьютер, который выполняет $NK$ обновлений. Несмотря на доказанные результаты статьи \cite{Woodworth_2021}, на практике в общем случае выпуклых функции получается использовать $K > 1$ локальных шагов, незначительно теряя при этом в точности, однако существенно выигрывая в вычислительных ресурсах. Именно практические результаты позволяют ожидать в будущем позитивные теоретические результаты.

Существующие алгоритмы первого порядка для архитектуры федеративного обучения зачастую решают задачу выпуклой оптимизации. Что касается задач оптимизации с седловой точкой, то алгоритмы в архитектуре федеративного обучения на данный момент отсутствуют. В данной статье мы разработали и получили оптимальные оценки скорости сходимости пробатченных методов первого порядка Minibatch SMP (Mb-SMP) и Single-Machine SMP (SM-SMP) для решения задач оптимизации с седловой точкой, используя аналогичный подход, что и в выпуклой оптимизации \cite{Woodworth_2021}. Оценки сходимости для алгоритмов решения седловых задач также приведены в табл. \ref{Convergence_rates}. Подробное описание получения верхних оценок скорости сходимости для алгоритмов Minibatch SMP (Mb-SMP) и Single-Machine SMP (SM-SMP) приведено в приложении \ref{Upper_bound}. 


\subsection{Оптимальные алгоритмы нулевого порядка}\label{subsection:Gradient_Free_Algorithms}
Здесь представлен главный результат данной статьи, который заключается в объединении двух глобальных идей в области оптимизации, представленные в разд. \ref{section:Smooth_schemes} и п. \ref{subsection:First_order_methods_for_FL}. А именно решение стохастических негладких выпуклых/выпукло-вогнутых задач оптимизации безградиентными алгоритмами архитектуры федеративного обучения. Для разработки и применения безградиентных методов решения негладких задач предлагается выбрать метод первого порядка, используемый для решения гладких задач в архитектуре федеративного обучения. Далее необходимо модифицировать алгоритм, выбранного метода, путем замены вычисления стохастического градиента на безградиентную аппроксимацию с $l_1$ \eqref{grad_l1_randomization} или $l_2$ \eqref{grad_l2_randomization} рандомизацией. Полученный безградиентный алгоритм будет иметь одноименное название алгоритма первого порядка, однако он не будет требовать информацию о стохастическом градиенте. 

Пусть под алгоритмом \textbf{A}$(L, \sigma^2)$ будем понимать ускоренные пакетные алгоритмы решения задач выпуклой оптимизации: Minibatch и Single-Machine Accelerated SGD, Local-AC-CA и FedAc и ускоренные пакетные алгоритмы решения задач с седловой точкой: Minibatch SMP и Single-Machine SMP в архитектуре федеративного обучения. Тогда в теоремах \ref{theorem_1}-\ref{theorem_2} предполагаем, что выполняется данное свойство: алгоритм \textbf{A}$(L, \sigma^2)$ (со смещенным безградиентным оракулом $\nabla f_\gamma(x,\xi,e)$) является надежным, если смещение из следствия \ref{corollary_lemma_3_l1_randomization} для $l_1$-рандомизации, из следствия \ref{corollary_lemma_3_l2_randomization} для $l_2$-рандомизации (рассмотренных в п. \ref{estimate_noise}) не накапливается в течение итерации метода. То есть, если для \textbf{A}$(L, \sigma^2)$ при $\Delta = 0$:
\begin{equation*}
    \mathds{E}[f_{\gamma}(x^{N}) - f(x_{*})] \leq \Theta_A (N),
\end{equation*}
тогда 
\begin{itemize}
    \item при $\Delta > 0$ и $d^{4-\frac{2}{p}} \Delta^2 \gamma^{-2} \lesssim \kappa(p,d)M_2^2 $ из \eqref{sigma_for_l1_randomization}:
    \begin{equation}
    \label{Noise_for_l1}
        \mathds{E}[f_{\gamma}(x^{N}) - f(x_{*})] \leq O \left( \Theta_A (N) + d \Delta R \gamma^{-1} \right);
    \end{equation}
    
    \item при $\Delta > 0$ и $d^2 \Delta^2 \gamma^{-2} \lesssim dM_2^2 $ из \eqref{sigma_for_l2_randomization}:
    \begin{equation}
    \label{Noise_for_l2}
        \mathds{E}[f_{\gamma}(x^{N}) - f(x_{*})] \leq O \left( \Theta_A (N) + \sqrt{d} \Delta R \gamma^{-1} \right).
    \end{equation}
\end{itemize}

Предположение о выполнении данного свойства базируется на статье \cite{Gorbunov_2019}, где был разработан анализ сходимости методов для смещенных стохастических оракулов. Поэтому в теоремах \ref{theorem_1}-\ref{theorem_2} приводим результаты, предполагая, что можно по аналогии провести анализ сходимости методов федеративного обучения, рассмотренных в этой работе, для смещенных стохастических оракулов. Однако, в доказательстве теорем \ref{theorem_1}-\ref{theorem_2} будем рассматривать случай при $\Delta = 0$ и приводить оптимальные оценки параметров разработанных безградиентных алгоритмов.

\begin{theorem}\label{theorem_1}
    Схема сглаживания из разд. \ref{section:Smooth_schemes}, применяемая к задаче \eqref{Convex_opt_problem}, обеспечивает сходимость следующих двухточечных безградиентных алгоритмов: Minibatch и Single-Machine Accelerated SGD \cite{Woodworth_2021}, Local-AC-CA \cite{Woodworth_2020} и FedAc \cite{Yuan_Ma_2020}. Другими словами, для достижения $\varepsilon$ точности решения задачи \eqref{Convex_opt_problem} необходимо проделать $NK$ итераций с максимально допустимым уровнем шума $\Delta$ и общим числом вызова безградиентного оракула $T$ в соответствии с выбранным методом и схемой сглаживания:
    \begin{itemize}
        \item Minibatch Accelerated SGD
        \begin{itemize}
            \item[i)] для $l_1$-рандомизации (\ref{grad_l1_randomization}):
            \begin{equation*}
                \Delta = O\left( \frac{ \varepsilon^2 }{\sqrt{d} M_2 R} \right); 
            \end{equation*}
            \begin{equation*}
                 N = O\left( \frac{d^{1/4} \sqrt{M M_2} R}{\varepsilon} \right); \;\;\;  
                 K = 1; \;\;\; 
                 B = O\left( \frac{\kappa(p,d)  M_2^2 R^2}{KN \varepsilon^2} \right);
            \end{equation*}
            \begin{equation*}
                 T =  O \left(  \frac{\kappa(p,d)  M_2^2 R^2}{ \varepsilon^2}   \right) = \begin{cases}
                  O \left(  \frac{d M_2^2 R^2}{ \varepsilon^2}   \right), & p = 2,\\
                  O \left(  \frac{M_2^2 R^2}{ \varepsilon^2}   \right), & p = 1;
                \end{cases}
            \end{equation*}
        
        \item[ii)] для $l_2$-рандомизации (\ref{grad_l2_randomization}):
            \begin{equation*}
                \Delta = O\left( \frac{ \varepsilon^2 }{\sqrt{d} M_2 R} \right); 
            \end{equation*}
            \begin{equation*}
                 N = O\left( \frac{d^{1/4} \sqrt{M M_2} R}{\varepsilon} \right); \;\;\;  
                 K = 1; \;\;\; 
                 B = O\left( \frac{\kappa(p,d) d M_2^2 R^2}{KN \varepsilon^2} \right);
            \end{equation*}
            \begin{equation*}
                 T = \tilde O \left(  \frac{\kappa(p,d) d M_2^2 R^2}{ \varepsilon^2}   \right) = \begin{cases}
                 \tilde O \left(  \frac{d M_2^2 R^2}{ \varepsilon^2}   \right), & p = 2,\\
                 \tilde O \left(  \frac{(\ln d) M_2^2 R^2}{ \varepsilon^2}   \right), & p = 1;
                \end{cases}
            \end{equation*}
        \end{itemize}
        
        \item Single-Machine Accelerated SGD
        \begin{itemize}
            \item[i)] для $l_1$-рандомизации (\ref{grad_l1_randomization}):
            \begin{equation*}
                \Delta = O\left( \frac{ \varepsilon^2 }{\sqrt{d} M_2 R} \right); 
            \end{equation*}
            \begin{equation*}
                 N = 1; \;\;\;  
                 K = O\left( \frac{\kappa(p,d) M_2^2 R^2}{ \varepsilon^2} \right); \;\;\; 
                 B = 1;
            \end{equation*}
            \begin{equation*}
                 T =  O \left(  \frac{\kappa(p,d) M_2^2 R^2}{ \varepsilon^2}   \right) = \begin{cases}
                  O \left(  \frac{d M_2^2 R^2}{ \varepsilon^2}   \right), & p = 2,\\
                  O \left(  \frac{M_2^2 R^2}{ \varepsilon^2}   \right), & p = 1;
                \end{cases}
            \end{equation*}
        
        \item[ii)] для $l_2$-рандомизации (\ref{grad_l2_randomization}):
            \begin{equation*}
                \Delta = O\left( \frac{ \varepsilon^2 }{\sqrt{d} M_2 R} \right); 
            \end{equation*}
            \begin{equation*}
                 N = 1; \;\;\;  
                 K = O\left( \frac{\kappa(p,d) M_2^2 R^2}{ \varepsilon^2} \right); \;\;\; 
                 B = 1;
            \end{equation*}
            \begin{equation*}
                 T = \tilde O \left(  \frac{\kappa(p,d) M_2^2 R^2}{ \varepsilon^2}   \right) = \begin{cases}
                 \tilde O \left(  \frac{d M_2^2 R^2}{ \varepsilon^2}   \right), & p = 2,\\
                 \tilde O \left(  \frac{(\ln d) M_2^2 R^2}{ \varepsilon^2}   \right), & p = 1;
                \end{cases}
            \end{equation*}
        \end{itemize}
        
        \item Local-AC-SA
        \begin{itemize}
            \item[i)] для $l_1$-рандомизации (\ref{grad_l1_randomization}):
            \begin{equation*}
                \Delta = O\left( \frac{ \varepsilon^2 }{\sqrt{d} M_2 R} \right); 
            \end{equation*}
            \begin{equation*}
                 N = 1; \;\;\;  
                 K = O\left( \frac{d^{1/4} \sqrt{M M_2} R}{\varepsilon} \right); \;\;\; 
                 B = O\left( \frac{\kappa(p,d) M_2^2 R^2}{KN \varepsilon^2} \right);
            \end{equation*}
            \begin{equation*}
                 T =  O \left(  \frac{\kappa(p,d) M_2^2 R^2}{ \varepsilon^2}   \right) = \begin{cases}
                  O \left(  \frac{d M_2^2 R^2}{ \varepsilon^2}   \right), & p = 2,\\
                  O \left(  \frac{ M_2^2 R^2}{ \varepsilon^2}   \right), & p = 1;
                \end{cases}
            \end{equation*}
        
        \item[ii)] для $l_2$-рандомизации (\ref{grad_l2_randomization}):
            \begin{equation*}
                \Delta = O\left( \frac{ \varepsilon^2 }{\sqrt{d} M_2 R} \right); 
            \end{equation*}
            \begin{equation*}
                 N = 1; \;\;\;  
                 K = O\left( \frac{d^{1/4} \sqrt{M M_2} R}{\varepsilon} \right); \;\;\; 
                 B = O\left( \frac{\kappa(p,d) M_2^2 R^2}{KN \varepsilon^2} \right);
            \end{equation*}
            \begin{equation*}
                 T = \tilde O \left(  \frac{\kappa(p,d) M_2^2 R^2}{ \varepsilon^2}   \right) = \begin{cases}
                 \tilde O \left(  \frac{d M_2^2 R^2}{ \varepsilon^2}   \right), & p = 2,\\
                 \tilde O \left(  \frac{(\ln d) M_2^2 R^2}{ \varepsilon^2}   \right), & p = 1;
                \end{cases}
            \end{equation*}
        \end{itemize}
        
        \item Federated Accelerated SGD (FedAc)
        \begin{itemize}
            \item[i)] для $l_1$-рандомизации (\ref{grad_l1_randomization}):
            \begin{equation*}
                \Delta = O\left( \frac{ \varepsilon^2 }{\sqrt{d} M_2 R} \right); 
            \end{equation*}
            \begin{equation*}
                 N = O\left( \frac{d^{1/6} (\kappa(p,d) M)^{1/3} M_2 R^{4/3}}{K^{1/3} \varepsilon^{4/3}} \right); \;\;\;  
                 K = O\left( \frac{ \kappa(p,d) M^2_2 R^2}{B N \varepsilon^2} \right); \;\;\; 
                 B = 1;
            \end{equation*}
            \begin{equation*}
                 T =  O \left(  \frac{\kappa(p,d) M_2^2 R^2}{ \varepsilon^2}   \right) = \begin{cases}
                  O \left(  \frac{d M_2^2 R^2}{ \varepsilon^2}   \right), & p = 2,\\
                  O \left(  \frac{ M_2^2 R^2}{ \varepsilon^2}   \right), & p = 1;
                \end{cases}
            \end{equation*}
        
        \item[ii)] для $l_2$-рандомизации (\ref{grad_l2_randomization}):
            \begin{equation*}
                \Delta = O\left( \frac{ \varepsilon^2 }{\sqrt{d} M_2 R} \right); 
            \end{equation*}
            \begin{equation*}
                 N = O\left( \frac{d^{1/2} (\kappa(p,d) M)^{1/3} M_2 R^{4/3}}{K^{1/3} \varepsilon^{4/3}} \right); \;\;\;  
                 K = O\left( \frac{ \kappa(p,d) M^2_2 R^2}{B N \varepsilon^2} \right); \;\;\; 
                 B = 1;
            \end{equation*}
            \begin{equation*}
                 T = \tilde O \left(  \frac{\kappa(p,d) M_2^2 R^2}{ \varepsilon^2}   \right) = \begin{cases}
                 \tilde O \left(  \frac{d M_2^2 R^2}{ \varepsilon^2}   \right), & p = 2,\\
                 \tilde O \left(  \frac{(\ln d) M_2^2 R^2}{ \varepsilon^2}   \right), & p = 1.
                \end{cases}
            \end{equation*}
        \end{itemize}
    \end{itemize}
\end{theorem}
Подробное доказательство теоремы \ref{theorem_1} приведено в приложении \ref{proof_theorem_1}.

По результатам теоремы \ref{theorem_1} стоит обратить внимание, что количество локальных обновлений $K>1$ с использованием $B>1$ компьютеров при оптимальных значениях $N$ и $T$ получилось только для одного алгоритма, а именно Local-AC-CA, который применяется для частного квадратичного случая. Это подтверждает, что в теории на данный момент существует только два оптимальных (по $\Delta$, $N$ и $T$) алгоритма: Minibatch и Single-Machine Accelerated SGD. 
В работе \cite{Yuan_Ma_2020} показано, что на практике можно получить результат, в котором алгоритм совершает параллельно $K>1$ локальных обновлений на каждых $B>1$ компьютерах. 

В теореме \ref{theorem_2} также получены оптимальные оценки параметров, рассмотренных в п.~\ref{subsection:First_order_methods_for_FL} алгоритмов для решения стохастических негладких задач с седловой точкой, оптимальных по $\Delta$, $N$ и $T$.  

\begin{theorem}
\label{theorem_2}
    Схема сглаживания из разд. \ref{section:Smooth_schemes}, применяемая к седловой задаче (см. замечание \ref{remark_convex_concave}), обеспечивает сходимость следующих двухточечных безградиентных алгоритмов: Minibatch SMP и Single-Machine SMP из приложения \ref{Upper_bound}.  Другими словами, для достижения $\varepsilon$ точности решения седловой задачи (см. замечание \ref{remark_convex_concave}) необходимо проделать $NK$ итераций с максимально допустимым уровнем шума $\Delta$ и общим числом вызова безградиентного оракула $T$ в соответствии с выбранным методом и схемой сглаживания:
    \begin{itemize}
        \item Minibatch SMP
        \begin{itemize}
            \item[i)] для $l_1$-рандомизации (\ref{grad_l1_randomization}):
            \begin{equation*}
                \Delta = O\left( \frac{ \varepsilon^2 }{\sqrt{d} M_2 R} \right); 
            \end{equation*}\\\\
            \begin{equation*}
                 N = 1; \;\;\;  
                 K = 1; \;\;\; 
                 B = O\left( \frac{\kappa(p,d) M_2^2 R^2}{ \varepsilon^2} \right);
            \end{equation*}
            \begin{equation*}
                 T =  O \left(  \frac{\kappa(p,d) M_2^2 R^2}{ \varepsilon^2}   \right) = \begin{cases}
                  O \left(  \frac{d M_2^2 R^2}{ \varepsilon^2}   \right), & p = 2,\\
                  O \left(  \frac{ M_2^2 R^2}{ \varepsilon^2}   \right), & p = 1;
                \end{cases}
            \end{equation*}
        
        \item[ii)] для $l_2$-рандомизации (\ref{grad_l2_randomization}):
            \begin{equation*}
                \Delta = O\left( \frac{ \varepsilon^2 }{\sqrt{d} M_2 R} \right); 
            \end{equation*}
            \begin{equation*}
                 N = 1; \;\;\;  
                 K = 1; \;\;\; 
                 B = O\left( \frac{\kappa(p,d) M_2^2 R^2}{ \varepsilon^2} \right);
            \end{equation*}
            \begin{equation*}
                 T = \tilde O \left(  \frac{\kappa(p,d) M_2^2 R^2}{ \varepsilon^2}   \right) = \begin{cases}
                 \tilde O \left(  \frac{d M_2^2 R^2}{ \varepsilon^2}   \right), & p = 2,\\
                 \tilde O \left(  \frac{(\ln d) M_2^2 R^2}{ \varepsilon^2}   \right), & p = 1;
                \end{cases}
            \end{equation*}
        \end{itemize}
        
        \item Single-Machine SMP
        \begin{itemize}
            \item[i)] для $l_1$-рандомизации (\ref{grad_l1_randomization}):
            \begin{equation*}
                \Delta = O\left( \frac{ \varepsilon^2 }{\sqrt{d} M_2 R} \right); 
            \end{equation*}
            \begin{equation*}
                 N = 1; \;\;\;  
                 K = O\left( \frac{\kappa(p,d) M_2^2 R^2}{ \varepsilon^2} \right); \;\;\; 
                 B = 1;
            \end{equation*}
            \begin{equation*}
                 T = O \left(  \frac{\kappa(p,d) M_2^2 R^2}{ \varepsilon^2}   \right) = \begin{cases}
                  O \left(  \frac{d M_2^2 R^2}{ \varepsilon^2}   \right), & p = 2,\\
                  O \left(  \frac{M_2^2 R^2}{ \varepsilon^2}   \right), & p = 1;
                \end{cases}
            \end{equation*}
        
        \item[ii)] для $l_2$-рандомизации (\ref{grad_l2_randomization}):
            \begin{equation*}
                \Delta = O\left( \frac{ \varepsilon^2 }{\sqrt{d} M_2 R} \right); 
            \end{equation*}
            \begin{equation*}
                 N = 1; \;\;\;  
                 K = O\left( \frac{\kappa(p,d) M_2^2 R^2}{ \varepsilon^2} \right); \;\;\; 
                 B = 1;
            \end{equation*}
            \begin{equation*}
                 T = \tilde O \left(  \frac{\kappa(p,d) M_2^2 R^2}{ \varepsilon^2}   \right) = \begin{cases}
                 \tilde O \left(  \frac{d M_2^2 R^2}{ \varepsilon^2}   \right), & p = 2,\\
                 \tilde O \left(  \frac{(\ln d) M_2^2 R^2}{ \varepsilon^2}   \right), & p = 1.
                \end{cases}
            \end{equation*}
        \end{itemize}
    \end{itemize}
\end{theorem}
Подробное доказательство теоремы \ref{theorem_2} приведено в приложении \ref{proof_theorem_2}.

 По результатам теоремы \ref{theorem_1} и теоремы \ref{theorem_2} не трудно увидеть, что для всех алгоритмов оптимальное число вызовов безградиентного оракула в $l_1$-норме равняется $ \tilde O \left(  \frac{(\ln d) M_2^2 R^2}{ \varepsilon^2} \right)$ с $l_2$-рандомизацией, в то время как для $l_1$-рандомизацией равняется $O \left(  \frac{ M_2^2 R^2}{ \varepsilon^2}  \right)$, где $\varepsilon$ является точностью решения негладкой задачи. Данный результат подтверждает, что схема сглаживания с $l_1$-рандомизацией лучше работает в архитектуре федеративного обучения, чем схема сглаживания с $l_2$-рандомизацией. В разд. \ref{section:Numerical_result} проверим этот результат на практическом эксперименте.
 
 \begin{remark}
    Для получения результатов теорем \ref{theorem_1} и \ref{theorem_2} использовалось предположение о доступности двухточечной обратной связи (см. п.~\ref{subsection:SS_3}). Если вместо двухточечной обратной связи использовать одноточечную (см. п.~\ref{subsection:SS_4}), то это приведет к такой же итерационной сложности (количеству коммуникационных раундов) $N$ и максимальному уровню неточности $\Delta$, однако оракульная сложность возрастет в $O \left( d/\varepsilon^2 \right)$ раз. Данный случай подробно рассмотрен в приложении \ref{case_one_point}.
 \end{remark}
 
 \begin{remark}
    Стоит отметить, что идея объединения двух техник, а именно федеративного обучения со схемами сглаживания, не ограничивается рассмотренными алгоритмами первого порядка и может быть использована для решения негладких задач другими алгоритмами, используемыми в архитектуре федеративного обучения. 
 \end{remark}



\section{Схемы доказательств}\label{section:Elements_for_proof}
В этом разделе приведем схемы для доказательств теоремы \ref{theorem_1} и теоремы \ref{theorem_2}. Подробные доказательства этих теорем можно найти в приложении \ref{proof_theorem_1}, \ref{proof_theorem_2}. Здесь сфокусируемся на доказательстве леммы \ref{properties_f_gamma}-\ref{properties_grad_f_gamma_one_point} и на нахождении оценки уровня неточности (враждебного шума). 

\subsection{Доказательство леммы \ref{properties_f_gamma}}
\label{subsection:5_1}
В данном п. рассмотрим неевклидовый случай, когда случайный вектор $\Tilde{e}$ равномерно распределен на $l_1$-шаре. Доказательство евклидового случая можно найти в теореме 2.1 из \cite{Gasnikov_2022_Smooth_Scheme}.

Для всех $x, y \in Q$
\begin{enumerate}
    \item $f(x) \leq f_\gamma(x) \leq f(x) + \frac{2}{\sqrt{d}} \gamma M_2$;
    
    \item $f_\gamma$~--- $M$-липшицева:
    \begin{equation*}
        |f_\gamma(y) - f_\gamma(x) | \leq \al{M \| y - x \|_p},
    \end{equation*}

    \item $f_\gamma$ имеет $L_{f_{\gamma}} = \frac{d M}{\al{2} \gamma}$-липшицевый градиент:
    \begin{equation*}
        \al{\| \nabla f_\gamma(y) - \nabla f_\gamma(x) \|_q \leq L_{f_{\gamma}} \| y - x \|_2 \leq L_{f_{\gamma}} \| y - x \|_p}.
    \end{equation*}
    
    где $q$ такой, что $1/p + 1/q = 1$.
    
\end{enumerate}

\begin{proof}\renewcommand{\qedsymbol}{}
    Для первого неравенства первого свойства воспользуемся выпуклостью функции $f(x)$
    \begin{equation*}
        f_\gamma(x) = \mathbb{E}_{\Tilde{e}} \left[ f(x + \gamma \Tilde{e}) \right] \geq \mathbb{E}_{\Tilde{e}} \left[ f(x) + \dotprod{\nabla f(x)}{\gamma \Tilde{e}}) \right] = \mathbb{E}_{\Tilde{e}} \left[ f(x) \right] = f(x).
    \end{equation*}
    Для второго неравенства первого свойства, применяя лемму 1 из \cite{Tsybakov_2022}, имеем:
    \begin{equation*}
        | f_\gamma (x)- f(x) | = | \mathbb{E}_{\Tilde{e}} \left[ f(x + \gamma \Tilde{e}) \right] - f(x) | \leq \mathbb{E}_{\Tilde{e}} \left[ | f(x + \gamma \Tilde{e}) - f(x) | \right] \leq \gamma M_2 \mathbb{E}_{\Tilde{e}} \left[ \| \Tilde{e} \|_2 \right] \leq \frac{2}{\sqrt{d}} \gamma M_2,
    \end{equation*}
    используя тот факт, что $f$ является $M_2$-липшицевой.
    
    Для второго свойства:
    \begin{equation*}
    \begin{split}
        | f_\gamma (y)- f_\gamma(x) | = | \mathbb{E}_{\Tilde{e}} \left[ f(y + \gamma \Tilde{e})  - f(x + \gamma \Tilde{e}) \right]| & \leq \mathbb{E}_{\Tilde{e}} \left[ | f(y + \gamma \Tilde{e})  - f(x + \gamma \Tilde{e}) | \right] \leq \al{M \| y - x \|_p}
    \end{split}
    \end{equation*}
    
    И для третьего свойства, применяя лемму 11 из \cite{Duchi_2012} имеем для любых $x,y \in Q$,
    	    \begin{equation*}
    	    \begin{split}
    	        \| \nabla f_\gamma(y) - \nabla f_\gamma(x) \|_q 
    	        & = \al{\left\| \int_{B^d_1(\gamma)} \nabla f(y + \Tilde{e}) \mu(\Tilde{e}) d\Tilde{e} - \int_{B^d_1(\gamma)} \nabla f(x + \Tilde{e}) \mu(\Tilde{e}) d\Tilde{e} \right\|_q} \leq \\ 
    	        & \leq \left\| \int_{Q_\gamma} \nabla f(\al{z}) \mu(\al{z} - y) d\al{z} - \int_{Q_\gamma} \nabla f(\al{z}) \mu(\al{z} - x) d\al{z} \right\|_q \leq \\
    	        & \leq  M  \underbrace{\int_{Q_\gamma} | \mu(\al{z} - y) - \mu(\al{z} - x) | d\al{z}}_{I_1}, 
    	    \end{split}
    	    \end{equation*}
    	    где \al{второе неравенство следует из $z = x + \Tilde{e}$ и} $\mu(x) = 
    	    \begin{cases}
                \frac{1}{V(B^d_1(\gamma))}, & x \in B^d_1(\gamma), \\
                0, & \text{иначе}.
            \end{cases}$ Далее для оценки интеграла $I_1$ применим тот же подход, что и в доказательстве леммы 8 из \cite{Yousefian_2012} и рассмотрим случаи где $\| y - x \|_1 > 2\gamma$ и $\| y - x \|_1 \leq 2\gamma$. 
            
            \textit{Случай 1} ($\| y - x \|_1 > 2\gamma$). \\
            Для всех $\Tilde{e}$ пусть $\| \Tilde{e} - x \|_1 \leq \gamma$, тогда имеем $ \| \Tilde{e} - y \|_1 > \gamma  \Rightarrow$  $\mu (\Tilde{e} - y) = 0$, поэтому 
            \begin{equation*}
            \begin{split}
                & \int_{\| \Tilde{e} - x \|_1 \leq \gamma} | \mu(\Tilde{e} - y) - \mu(\Tilde{e} - x) | d\Tilde{e} = 1,
            \end{split}
            \end{equation*}
             Аналогично для всех $\Tilde{e}$ пусть $\| \Tilde{e} - y \|_1 \leq \gamma$, тогда имеем $ \mu (\Tilde{e} - x) = 0$, поэтому 
            \begin{equation*}
            \begin{split}
                & \int_{\| \Tilde{e} - y \|_1 \leq \gamma} | \mu(\Tilde{e} - y) - \mu(\Tilde{e} - x) | d\Tilde{e} = 1,
            \end{split}
            \end{equation*} Следовательно 
            \begin{equation*}
                \int_{Q_\gamma} | \mu(\Tilde{e} - y) - \mu(\Tilde{e} - x) | d\Tilde{e} = \int_{\| \Tilde{e} - x \|_1 \leq \gamma} | \mu(\Tilde{e} - y) - \mu(\Tilde{e} - x) | d\Tilde{e} + \int_{\| \Tilde{e} - y \|_1 \leq \gamma} | \mu(\Tilde{e} - y) - \mu(\Tilde{e} - x) | d\Tilde{e} = 2.
            \end{equation*}
            
            \al{Поскольку $2 < \frac{\| y - x \|_1}{\gamma} \al{\leq \frac{d^{1 - 1/p} \| y - x \|_p}{\gamma}}$, учитывая тот факт, что для $p \in [1,2]$ выполнено $\| y - x \|_1 \leq d^{1- \frac{1}{p}} \| y - x \|_p $, тогда получим следующее неравенство:} 
            
            \begin{equation*}
                 \int_{Q_\gamma} | \mu(\Tilde{e} - y) - \mu(\Tilde{e} - x) | d\Tilde{e} \leq \frac{\al{d^{1 - 1/p} \| y - x \|_p}}{\gamma}.
            \end{equation*}
            
            \textit{Случай 2} ($\| y - x \|_1 \leq 2\gamma$): Разложим интеграл $I_1$ на 4 интеграла.
            
            \begin{equation*}
            \begin{split}
                 \int_{Q_\gamma} | \mu(\Tilde{e} - y) - \mu(\Tilde{e} - x) | d\Tilde{e} =\\ &  \hspace{-13.15em}
                 = \int_{\| \Tilde{e} - x \|_1 \leq \gamma \; \& \; \| \Tilde{e} - y \|_1 \leq \gamma} | \mu(\Tilde{e} - y) - \mu(\Tilde{e} - x) | d\Tilde{e} + \int_{\| \Tilde{e} - x \|_1 \leq \gamma \; \& \; \| \Tilde{e} - y \|_1 \geq \gamma} | \mu(\Tilde{e} - y) - \mu(\Tilde{e} - x) | d\Tilde{e} + \\ &  \hspace{-13.15em}
                 + \int_{\| \Tilde{e} - x \|_1 \geq \gamma \; \& \; \| \Tilde{e} - y \|_1 \leq \gamma} | \mu(\Tilde{e} - y) - \mu(\Tilde{e} - x) | d\Tilde{e} + \int_{\| \Tilde{e} - x \|_1 \geq \gamma \; \& \; \| \Tilde{e} - y \|_1 \geq \gamma} | \mu(\Tilde{e} - y) - \mu(\Tilde{e} - x) | d\Tilde{e}.
            \end{split}
            \end{equation*}
            
            Теперь рассмотрим каждый интеграл из разложения по отдельности:
            \begin{enumerate}
                \item Для первого и четвертого интегралов справедливо следующее 
                    \begin{equation*}
                        \mu(\Tilde{e} - x) = \mu(\Tilde{e} - y).
                    \end{equation*} Тогда, подставляя в первый интеграл, получим
                    \begin{equation}
                    \label{case_1_for_proof_lemma_3}
                        \int_{\| \Tilde{e} - x \|_1 \leq \gamma \; \& \; \| \Tilde{e} - y \|_1 \leq \gamma} | \mu(\Tilde{e} - y) - \mu(\Tilde{e} - x) | d\Tilde{e} = 0.
                    \end{equation}
                    Аналогично, подставляя в четвертый интеграл, получим
                    \begin{equation}
                    \label{case_2_for_proof_lemma_3}
                        \int_{\| \Tilde{e} - x \|_1 \geq \gamma \; \& \; \| \Tilde{e} - y \|_1 \geq \gamma} | \mu(\Tilde{e} - y) - \mu(\Tilde{e} - x) | d\Tilde{e} = 0.
                    \end{equation}
                    
                \item Когда $\| \Tilde{e} - x \|_1 \leq \gamma \; \& \; \| \Tilde{e} - y \|_1  \geq \gamma$ или $\| \Tilde{e} - x \|_1 \geq \gamma \; \& \; \| \Tilde{e} - y \|_1  \leq \gamma$, имеем 
                    \begin{equation*}
                        \int_{\| \Tilde{e} - x \|_1 \leq \gamma \; \& \; \| \Tilde{e} - y \|_1 \geq \gamma} | \mu(\Tilde{e} - y) - \mu(\Tilde{e} - x) | d\Tilde{e} = \int_{\| \Tilde{e} - x \|_1 \geq \gamma \; \& \; \| \Tilde{e} - y \|_1 \leq \gamma} | \mu(\Tilde{e} - y) - \mu(\Tilde{e} - x) | d\Tilde{e},
                    \end{equation*}
                    где $\mu (\Tilde{e} - x)$ и $\mu (\Tilde{e} - y)$ не пересекаются, поэтому, используя это и симметрию интегралов, и определив множество $S = \left\{ \Tilde{e} \in \mathds{R}^d \; | \; \| \Tilde{e} - x \|_1 \leq \gamma \text{ and } \| \Tilde{e} - y \|_1 \geq \gamma \right\}$, получим сумму второго и третьего интегралов
                    \begin{equation}
                    \label{case_3_for_proof_lemma_3}
                        2 \int_{\| \Tilde{e} - x \|_1 \leq \gamma \; \& \; \| \Tilde{e} - y \|_1 \geq \gamma} | \mu(\Tilde{e} - y) - \mu(\Tilde{e} - x) | d\Tilde{e} = \frac{2}{c_d \gamma^d} V_S,
                    \end{equation}
                    где $V_S$~--- объем на множестве $S$.
                    
            \end{enumerate}
            
           
            Суммируя значения четырех интегралов \eqref{case_1_for_proof_lemma_3}-\eqref{case_3_for_proof_lemma_3}, получим:
            \begin{equation}
            \label{eq_1_for_proof_lemma_1}
                \int_{Q_\gamma} | \mu(\Tilde{e} - y) - \mu(\Tilde{e} - x) | d\Tilde{e} = \frac{2}{c_d \gamma^d}V_S,
            \end{equation}
            где $V_S$~--- объем на множестве $S$.
            
            Теперь найдем верхнюю границу $V_S$ \al{через $\| y - x \|_p$}. Пусть $V_{cap} (r)$~--- объем сферического колпака с расстоянием $r$ до центра $l_1$-сферы. Тогда
            \begin{equation}
            \label{eq_2_for_proof_lemma_1}
                V_S = c_d \gamma^d - 2 V_{cap} \left( \frac{\al{\| y - x \|_p}}{\al{4}} \right).
            \end{equation}
            Объем $d$-мерного сферического колпака $V_{cap}$ может быть вычислен через $(d-1)$-мерную $l_1$-сферу следующим образом:
            \begin{equation*}
                V_{cap}(r) = \int_r^\gamma c_{d-1} \left( \gamma - \rho \right)^{d-1} d \rho \;\;\;\;\; \text{ для } r \in [0, \gamma],
            \end{equation*}
            где $c_d = \frac{2^d}{d!}$, $d \geq 1$. Для $r \in [0,\gamma]$ имеем
            \begin{equation*}
            \begin{split}
                V_{cap}'(r) = -c_{d-1}\left( \gamma - r \right)^{d-1} \leq 0, \\
                V_{cap}''(r) = (d-1) c_{d-1} (\gamma - r)^{d-2} \geq 0,
                \end{split}
            \end{equation*}
            где $V_{cap}', \;V_{cap}''$~--- первая и вторая производная по $r$ соответственно. Следовательно, $V_{cap}$ выпукла на $[0, \gamma]$, и по определению субградиента имеем
            \begin{equation*}
                V_{cap}(0) + V_{cap}'(0)r \leq V_{cap}(r)\;\;\;\;\; \text{ для } r \in [0, \gamma].
            \end{equation*}
            Так как $ V_{cap}(0) = \frac{1}{2}c_d \gamma^d$ и $ V_{cap}'(0) = -c_{d-1} \gamma^{d-1}$, отсюда следует, что
            \begin{equation}
            \label{eq_3_for_proof_lemma_1}
                \frac{1}{2}c_d \gamma^d - c_{d-1} \gamma^{d-1} \leq V_{cap}(r)\;\;\;\;\; \text{ для } r \in [0, \gamma].
            \end{equation}
            Поскольку $\al{\| y - x\|_1 }/ 2 \leq \gamma$ \al{ и отмечая, что $\al{\| y - x\|_p }/ 4 \leq \al{\| y - x\|_p }/ 2 \leq \al{\| y - x\|_1 }/ 2 $ выполняется  для $p \in [1,2]$}, можем задать $r = \al{\| y - x\|_p } / \al{4} \leq \gamma$ и подставить в (\ref{eq_3_for_proof_lemma_1}). Сделав это и используя (\ref{eq_2_for_proof_lemma_1}) получим
            \begin{equation*}
                V_S = c_d \gamma^d - 2 V_{cap} \left( \frac{\al{\| y - x\|_p }}{\al{4} } \right) \leq 2 c_{d-1} \gamma^{d-1} \frac{\al{\| y - x\|_p }}{\al{4} }.
            \end{equation*}
            Теперь подставляя полученную оценку $V_S$ в (\ref{eq_1_for_proof_lemma_1}), имеем
            \begin{equation*}
                \int_{Q_\gamma} | \mu(\Tilde{e} - y) - \mu(\Tilde{e} - x) | d\Tilde{e} \leq \frac{ c_{d-1}}{c_d}\frac{\al{\| y - x\|_p }}{\gamma}.
            \end{equation*}
            Поскольку $c_d = \frac{2^d}{d!}$, можно увидеть, что 
            \begin{equation*}
                \int_{Q_\gamma} | \mu(\Tilde{e} - y) - \mu(\Tilde{e} - x) | d\Tilde{e} \leq \frac{d}{\al{2} \gamma} \al{\| y - x\|_p }.
            \end{equation*}
            Теперь, получив оценку для интеграла $I_1$, имеем, что
            \begin{equation*}
                \| \nabla f_\gamma(y) - \nabla f_\gamma(x) \|_q \leq\frac{d M}{\al{2}\gamma} \al{\| y - x\|_p }.
            \end{equation*}
\end{proof}


\subsection{Доказательство леммы \ref{properties_grad_f_gamma}}
\label{proof_properties_grad_f_gamma}
В данном подразделе сосредоточимся на доказательство леммы \ref{properties_grad_f_gamma} для $l_2$-рандомизации. Во многих работах, например в \cite{Gasnikov_2022_Smooth_Scheme, Shamir_2017, Beznosikov_2020}, оценка второго момента для безградиентной аппроксимации \eqref{grad_l2_randomization} приведена и доказана с точностью до некоторой числовой константы $c$. В данном доказательстве разберемся чему же равна эта числовая константа $c$. А подробное доказательство леммы \ref{properties_grad_f_gamma} для $l_1$-рандомизации приведено в лемме 4 \cite{Tsybakov_2022}. 

Для всех $x \in Q$ с предположениями \ref{assumption:SS_1} и \ref{assumption:SS_2} $\nabla f_\gamma(x, \xi, e)$ из \eqref{grad_l2_randomization} имеет нижнюю оценку (второй момент):
\begin{equation*}
   \mathds{E}_{\xi, e} \left[ \| \nabla f_\gamma (x, \xi, e) \|^2_q \right] \leq \kappa(p,d) \left( dM_2^2 + \frac{d^2 \Delta^2}{\sqrt{2} \gamma^2}  \right), 
\end{equation*}
где $1/p + 1/q = 1$ и 
\begin{equation*}
    \kappa(p,d) = \sqrt{2} \min \left\{ q, \ln d \right\} d^{2/q - 1}.
\end{equation*}


\begin{proof}\renewcommand{\qedsymbol}{}
    Давайте рассмотрим 
    \begin{equation}
    \label{eq_1_for_proof_second_moment}
        \begin{split} &  \hspace{-0.1em}
            \mathbb{E}_{\xi, e} \left[ \| \nabla f_\gamma (x, \xi, e) \|_q^2 \right] = \mathbb{E}_{\xi, e} \left[ \left\| \frac{d}{2 \gamma}(f_\delta(x+\gamma e,\xi) - f_\delta(x-\gamma e,\xi)) e \right\|_q^2 \right] = \\ &  \hspace{-0.1em}
            = \frac{d^2}{4 \gamma^2} \mathbb{E}_{\xi, e} \left[ \| e \|_q^2 \left( f(x+\gamma e,\xi) + \delta(x+\gamma e) - f(x-\gamma e,\xi) - \delta(x-\gamma e) \right)^2 \right] \leq \\ &  \hspace{-0.1em}
            \leq \frac{d^2}{2 \gamma^2} \left( \mathbb{E}_{\xi, e} \left[ \| e \|_q^2 \left( f(x+\gamma e,\xi) - f(x-\gamma e,\xi) \right)^2 \right] + \mathbb{E}_{e}  \left[ \| e \|_q^2 \left( \delta(x+\gamma e) - \delta(x-\gamma e) \right)^2 \right] \right), 
        \end{split}
    \end{equation}
    где использовался тот факт, что для всех $a,b, (a + b)^2 \leq 2 a^2 + 2 b^2$. Для первого слагаемого \eqref{eq_1_for_proof_second_moment} выполняется следующее с произвольным параметром $\alpha$ с учетом симметричного распределения $e$
    
    \begin{equation}
    \label{eq_2_for_proof_second_moment}
        \begin{split} &  \hspace{-0.1em}
            \frac{d^2}{2 \gamma^2} \mathbb{E}_{\xi, e} \left[ \| e \|_q^2 \left( f(x+\gamma e,\xi) - f(x-\gamma e,\xi) \right)^2 \right] = 
            \\ &  \hspace{-0.1em}
            = \frac{d^2}{2 \gamma^2} \mathbb{E}_{\xi, e} \left[ \| e \|_q^2 \left( (f(x+\gamma e,\xi) - \alpha) - (f(x-\gamma e,\xi) - \alpha) \right)^2 \right] \leq 
            \\ &  \hspace{-0.1em}
            \leq \frac{d^2}{\gamma^2} \mathbb{E}_{\xi, e} \left[ \| e \|_q^2 \left( f(x+\gamma e,\xi) - \alpha \right)^2 + \left( f(x-\gamma e,\xi) - \alpha \right)^2 \right] = 
            \\ &  \hspace{-0.1em}
            = \frac{d^2}{\gamma^2} \left( \mathbb{E}_{\xi, e} \left[ \| e \|_q^2 \left( f(x+\gamma e,\xi) - \alpha \right)^2 \right] + \mathbb{E}_{\xi, e} \left[ \left( f(x-\gamma e,\xi) - \alpha \right)^2 \right] \right) =
            \\ &  \hspace{-0.1em}
            = \frac{2 d^2}{\gamma^2} \mathbb{E}_{\xi, e} \left[ \| e \|_q^2 \left( f(x+\gamma e,\xi) - \alpha \right)^2 \right].
        \end{split}
    \end{equation}
    Применяя неравенство Коши--Шварца для \eqref{eq_2_for_proof_second_moment} и используя $\sqrt{\mathbb{E} \left[ \| e \|_q^4 \right]} \leq \kappa'(p,d)$, где $\kappa'(p,d) = \min \left\{ q, \ln d \right\} d^{2/q - 1}$, получим
    \begin{equation}
        \label{eq_3_for_proof_second_moment}
        \begin{split}
            \frac{2 d^2}{\gamma^2} \mathbb{E}_{\xi, e} \left[ \| e \|_q^2 \left( f(x+\gamma e,\xi) - \alpha \right)^2 \right] \leq \frac{2 d^2}{\gamma^2} \mathbb{E}_{\xi} \left[ \sqrt{\mathbb{E} \left[ \| e \|_q^4 \right]} \sqrt{ \mathbb{E}_e \left[ \left( f(x+\gamma e,\xi) - \alpha \right)^4 \right]} \right] \leq 
            \\ &  \hspace{-22.5em}
            \leq \frac{2 d^2 \kappa'(p,d)}{\gamma^2} \mathbb{E}_{\xi} \left[ \sqrt{ \mathbb{E}_e \left[ \left( f(x+\gamma e,\xi) - \alpha \right)^4 \right]} \right].
        \end{split}
    \end{equation}
     Следующая лемма, которую рассмотрим, является уточнением леммы 9 \cite{Shamir_2017} с указанием числовой константы $c$.
     
    \begin{lemma}
        \label{Theorem_Shamir}
        Для любой функции $f(e)$, которая является $M$-Липшицевой относительно $l_2$-нормы, имеет место, что если $e$ равномерно распределен на евклидовой сфере, тогда
        \begin{equation*}
            \sqrt{\mathbb{E} \left[ \left( f(e) - \mathbb{E} \left[ f(e) \right] \right)^4  \right]} \leq \frac{M_2^2}{\sqrt{2} d}.
        \end{equation*}
    \end{lemma}
    \begin{proof}\renewcommand{\qedsymbol}{}
        Пусть неравенство концентрации меры, использующее математическое ожидание, имеет вид
        \begin{equation*}
            \Pr ( |f - \mathbb{E} f| > t) \leq K \exp (-\eta t^2),
        \end{equation*}
        где $K$ и $\eta$ неизвестные параметры. Тогда, чтобы найти параметры $K$ и $\eta$, выпишем неравенство, использующую медиану функции используя параметры $K$ и $\eta$ (так как в литературе, например в \cite{Ledoux_2005},  обычно записывается неравенство концентрации меры на сфере, используя медиану функции $M_f$, а не математическое ожидание): 
        \begin{equation}
        \label{shamir_conect}
            \Pr ( |f - M_f| > t) \leq 2 K \exp (-\eta t^2/4) = 4 \exp (- d t^2 / 2 M_2^2).
        \end{equation}
        Подставляя параметры $K = 2$ и $\eta = 2 d / M_2^2$ из неравенства \eqref{shamir_conect}, запишем стандартный результат о концентрации липцицевых функций на евклидовой единичной сфере
        \begin{equation*}
            \Pr (|f(e) - \mathbb{E} \left[ f(e) \right] | > t) \leq 2 \exp (-2 d t^2 / M_2^2).
        \end{equation*}
        Следовательно 
        \begin{align*} &  \hspace{-0.1em}
            \begin{split}
                \sqrt{\mathbb{E} \left[ \left( f(e) - \mathbb{E} \left[ f(e) \right] \right)^4  \right]} = \sqrt{\int_{t=0}^\infty \Pr \left( \left( f(e) - \mathbb{E} \left[ f(e) \right] \right)^4 > t \right) dt} = \\  &  \hspace{-28.1em}
                = \sqrt{\int_{t=0}^\infty \Pr \left( \left| f(e) - \mathbb{E} \left[ f(e) \right] \right| > \sqrt[4]{t} \right) dt} \leq \sqrt{\int_{t=0}^\infty 2 \exp \left( - \frac{2d\sqrt{t}}{M_2^2} \right) dt} = \sqrt{2 \frac{M_2^4}{(2d)^2}},
            \end{split}
        \end{align*}
        где в последнем шаге использовался факт, что $\int_{t=0}^\infty \exp (-\sqrt{x})dx = 2$. Таким образом, найдена числовая константа $c = \frac{1}{\sqrt{2}}$ из леммы 9 \cite{Shamir_2017}. 
    \end{proof}
    
    Затем используем лемму \ref{Theorem_Shamir} вместе с тем фактом, что $f(x + \gamma e, \xi )$ является $\gamma M_2(\xi)$-Липшицевой относительно $e$ с точки зрения $l_2$-нормы. Таким образом, для \eqref{eq_3_for_proof_second_moment} и $\alpha := \mathbb{E} \left[ f(x + \gamma e, \xi) \right]$ выполняется
    \begin{equation}
        \label{eq_4_for_proof_second_moment}
        \frac{2 d^2 \kappa'(p,d)}{\gamma^2} \mathbb{E}_{\xi} \left[ \sqrt{ \mathbb{E}_e \left[ \left( f(x+\gamma e,\xi) - \alpha \right)^4 \right]} \right] \leq \frac{2 d^2 \kappa'(p,d)}{\gamma^2} \cdot \frac{\gamma^2 \mathbb{E} \left[ M_2^2(\xi) \right]}{\sqrt{2} d} = \sqrt{2} \kappa'(p,d) d M_2^2. 
    \end{equation}
    Для второго слагаемого \eqref{eq_1_for_proof_second_moment} выполняется следующее:
    \begin{equation}
        \label{eq_5_for_proof_second_moment}
        \frac{d^2}{2 \gamma^2} \mathbb{E}_{e}  \left[ \| e \|_q^2 \left( \delta(x+\gamma e) - \delta(x-\gamma e) \right)^2 \right] \leq \frac{d^2 \Delta^2}{\gamma^2} \mathbb{E}_{e}  \left[ \| e \|_q^2 \right] \leq \frac{\kappa'(p,d) d^2 \Delta^2}{\gamma^2}.
    \end{equation}
    Подставляя \eqref{eq_4_for_proof_second_moment} и \eqref{eq_5_for_proof_second_moment} в неравенство \eqref{eq_1_for_proof_second_moment} и занося коэффициент $\sqrt{2}$ в $\kappa'(p,d)$, получим утверждение леммы \ref{properties_grad_f_gamma} для $l_2$-рандомизации c $\kappa(p,d) = \sqrt{2} \kappa'(p,d) = \sqrt{2} \min \left\{ q, \ln d \right\} d^{2/q - 1}$.

\end{proof}

\subsection{Доказательство леммы \ref{properties_grad_f_gamma_one_point}}
\label{proof_properties_grad_f_gamma_one_point}
В данном подразделе рассмотрим краткое доказательство леммы \ref{properties_grad_f_gamma_one_point} для случая с $l_1$-рандомизацией. С евклидовым случаем можно познакомиться в следующих работах \cite{Gasnikov_2022_Smooth_Scheme, Gasnikov_2017}.

Для всех $x \in Q$ с предположениями \ref{assumption:SS_2} и \ref{assumption:SS_3} $\nabla f_\gamma(x, \xi, e)$ из \eqref{grad_l1_randomization_one_point} имеет нижнюю оценку (второй момент):
\begin{equation*}
    \mathds{E}_e \left[ \| \nabla f_\gamma (x, \xi, e) \|^2_q \right] \leq \frac{d^{4 - \frac{2}{p}}}{\gamma^2} \left( G^2 + \Delta^2 \right),
\end{equation*}
где $1/p + 1/q = 1$.
\begin{proof}\renewcommand{\qedsymbol}{}
    Доказательство этой леммы будет базироваться на доказательстве леммы 4 из \cite{Tsybakov_2022}. Используя определение $\nabla f_\gamma(x, \xi, e)$, получим:
    \begin{equation*}
        \begin{split}
           \mathds{E} \left[ \| \nabla f_\gamma (x, \xi, e) \|^2_q \right] = & \; \frac{d^2}{\gamma^2} \mathds{E} \left[ (f(x + \gamma e) + \delta(x))^2 \| \text{sign}(e) \|_q^2 \right] = \\
           = & \; \frac{d^{4 - \frac{2}{p}}}{\gamma^2} \mathds{E} \left[ (f(x + \gamma e) + \delta(x))^2 \right].
        \end{split}
    \end{equation*}
    Далее с предположениями \ref{assumption:SS_2} и \ref{assumption:SS_3} получим утверждение леммы:
    \begin{equation*}
        \begin{split}
            \mathds{E} \left[ \| \nabla f_\gamma (x, \xi, e) \|^2_q \right] = & \;  \frac{d^{4 - \frac{2}{p}}}{\gamma^2} \mathds{E} \left[ (f(x + \gamma e) + \delta(x))^2 \right] \leq \frac{d^{4 - \frac{2}{p}}}{\gamma^2} \left( G^2 + \Delta^2 \right).
         \end{split}
    \end{equation*}
\end{proof}


\subsection{Оценки уровня неточности}
\label{estimate_noise}
В данном подразделе приведем необходимые леммы и следствия для нахождения двух оценок уровня неточности (враждебного шума): для $l_1$-рандомизация (\ref{grad_l1_randomization}) и для $l_2$-рандомизация (\ref{grad_l2_randomization}). Для этого воспользуемся известными результатами и теми же суждениями, что и в \cite{Dvinskikh_2022}.
\begin{lemma}[см. \cite{Tsybakov_2022}]
    \label{lemma_1_proof_noise_l1_randomization}
    Функция $f_\gamma (x)$ дифференцируема со следующим градиентом с $l_1$-рандомизацией:
    \begin{equation*}
        \nabla f_\gamma (x) =\mathds{E}_e \left[ \frac{d}{\gamma}f(x+\gamma e) \text{sign}(e) \right].   
    \end{equation*}
\end{lemma}
\begin{lemma}[см. \cite{Flaxman_2005}]
    \label{lemma_1_proof_noise_l2_randomization}
    Функция $f_\gamma (x)$ дифференцируема со следующим градиентом с $l_2$-рандомизацией:
    \begin{equation*}
        \nabla f_\gamma (x) =\mathds{E}_e \left[ \frac{d}{\gamma}f(x+\gamma e)e \right].   
    \end{equation*}
\end{lemma}
\begin{lemma}[см.\cite{Dvurechensky_2021}]
    \label{lemma_2_proof_noise_l2_randomization}
    Пусть вектор $e$ случайный единичный вектор из евклидовой единичной сферы $\{ e: \| e \|_2 = 1  \}$. Тогда для всех $r \in \mathds{R}^d$ следует
    \begin{equation*}
         \mathds{E} [ | \langle e, r \rangle | ] \leq  \frac{\| r \|_2}{ \sqrt{d}} .
    \end{equation*}
\end{lemma}
\begin{lemma}
    \label{lemma_3_proof_noise_lp_randomization}
    Для $ \nabla f_\gamma (x,\xi,e) $ и $\nabla f_\gamma (x)$ с предположением \ref{assumption:SS_2}, выполняется следующее.
    \begin{equation*}
         \mathds{E}_{\xi, e} [\langle \nabla f_\gamma (x,\xi,e), r \rangle] \geq \dotprod{ \nabla f_\gamma (x)}{r} - \frac{d \Delta \mathds{E}_{e} [ | \langle \text{sign}(e), r \rangle | ]}{\gamma},
    \end{equation*}
    где $\nabla f_\gamma$ с $l_1$-рандомизацией;
    
    \begin{equation*}
         \mathds{E}_{\xi, e} [\langle \nabla f_\gamma (x,\xi,e), r \rangle] \geq \dotprod{ \nabla f_\gamma (x)}{r} - \frac{d \Delta \mathds{E}_{e} [ | \langle e, r \rangle | ]}{\gamma},
    \end{equation*}
    где $\nabla f_\gamma$ с $l_2$-рандомизацией.
\end{lemma}
\begin{proof}\renewcommand{\qedsymbol}{}
	Рассмотрим
	\begin{itemize}
        \item[i)] для $l_1$-рандомизации (\ref{grad_l1_randomization}):
        	\begin{equation*}
        	\begin{split}
        	    \nabla f_\gamma (x,\xi,e) = \frac{d}{2\gamma}(f_\delta(x+\gamma e, \xi) - f_\delta(x-\gamma e, \xi))\text{sign}(e) = \\ &  \hspace{-17.8em}
        	    = \frac{d}{2\gamma}(f(x+\gamma e, \xi) + \delta(x+\gamma e) - f(x-\gamma e, \xi) - \delta(x-\gamma e))\text{sign}(e) =\\ &  \hspace{-17.8em}
        	    = \frac{d}{2\gamma}((f(x+\gamma e, \xi) - f(x-\gamma e, \xi))\text{sign}(e) + (\delta(x+\gamma e)  - \delta(x-\gamma e))\text{sign}(e)).
        	\end{split}
        	\end{equation*}
        	Из данного равенства следует 
        	\begin{equation}
        	\label{eq_1_for_proof_noise_l1_randomization}
        	\begin{split}
        	    \mathds{E}_{\xi, e} [\langle \nabla f_\gamma (x,\xi,e), r \rangle] = \frac{d}{2\gamma} \mathds{E}_{\xi, e} [\langle (f(x+\gamma e, \xi) - f(x-\gamma e, \xi))\text{sign}(e), r \rangle] + \\ 
        	    + \frac{d}{2\gamma} \mathds{E}_{ e} [\langle (\delta(x+\gamma e) - \delta(x-\gamma e))\text{sign}(e), r \rangle].
        	\end{split}
        	\end{equation}
        	Применяя лемму \ref{lemma_1_proof_noise_l1_randomization} к первому слагаемому (\ref{eq_1_for_proof_noise_l1_randomization}), получим
        	\begin{equation}
            	\begin{split}
            	\label{eq_2_for_proof_noise_l1_randomization} &  \hspace{-0.1em} 
            	    \frac{d}{2\gamma} \mathds{E}_{\xi, e} [\langle (f(x+\gamma e, \xi) - f(x-\gamma e, \xi))\text{sign}(e), r \rangle] = 
            	    \\ &  \hspace{-0.1em}
            	    = \frac{d}{2\gamma} \mathds{E}_{\xi, e} [\langle f(x+\gamma e, \xi)\text{sign}(e), r \rangle] + \frac{d}{2\gamma} \mathds{E}_{\xi, e} [\langle f(x-\gamma e, \xi)\text{sign}(e), r \rangle] = 
            	    \\  &  \hspace{-0.1em}
    	            = \frac{d}{\gamma} \mathds{E}_{e} \left[ \dotprod{\mathds{E}_{\xi} \left[ f(x+\gamma e, \xi) \right] \text{sign}(e)}{r} \right] = \frac{d}{\gamma} \mathds{E}_{e} [\langle f(x+\gamma e)\text{sign}(e), r \rangle] = \dotprod{\nabla f_\gamma(x)}{r}.  
            	\end{split}
        	\end{equation}
        	Для второго слагаемого (\ref{eq_1_for_proof_noise_l1_randomization}) с учетом $|\delta(x)|\leq \Delta$ получим
        	\begin{equation}
        	\label{eq_3_for_proof_noise_l1_randomization}
            \begin{split}
        	    \frac{d}{\gamma} \mathds{E}_{e} [\langle (\delta(x+\gamma e) - \delta(x-\gamma e))\text{sign}(e), r \rangle] \geq  -\frac{d}{2\gamma} 2\Delta \mathds{E}_{e} [ | \langle \text{sign}(e), r \rangle | ] = 
        	    \\
        	    = -\frac{d}{\gamma} \Delta \mathds{E}_{e} [ | \langle \text{sign}(e), r \rangle | ].
        	\end{split}
        	\end{equation}
        	Используя уравнения (\ref{eq_2_for_proof_noise_l1_randomization}) и (\ref{eq_3_for_proof_noise_l1_randomization}) для уравнения (\ref{eq_1_for_proof_noise_l1_randomization}), получим утверждение леммы для $l_1$-рандомизации.
    	
    	\item[ii)] для $l_2$-рандомизации (\ref{grad_l2_randomization}):
    	\begin{equation*}
    	\begin{split}
    	    \nabla f_\gamma (x,\xi,e) = \frac{d}{2\gamma}(f_\delta(x+\gamma e, \xi) - f_\delta(x-\gamma e, \xi))e = \\ &  \hspace{-15.2em}
    	    = \frac{d}{2\gamma}(f(x+\gamma e, \xi) + \delta(x+\gamma e) - f(x-\gamma e, \xi) - \delta(x-\gamma e))e =\\ &  \hspace{-15.2em}
    	    = \frac{d}{2\gamma}((f(x+\gamma e, \xi) - f(x-\gamma e, \xi))e + (\delta(x+\gamma e)  - \delta(x-\gamma e))e).
    	\end{split}
    	\end{equation*}
    	Из данного равенства следует 
    	\begin{equation}
        \label{eq_1_for_proof_noise_l2_randomization}
            \begin{split}
                \mathds{E}_{\xi, e} [\langle \nabla f_\gamma (x,\xi,e), r \rangle] = \frac{d}{2\gamma} \mathds{E}_{\xi, e} [\langle (f(x+\gamma e, \xi) - f(x-\gamma e, \xi))e, r \rangle] + \\ 
                + \frac{d}{2\gamma} \mathds{E}_{e} [\langle (\delta(x+\gamma e) - \delta(x-\gamma e))e, r \rangle].
            	\end{split}
        	\end{equation}
        	Применяя лемму \ref{lemma_1_proof_noise_l2_randomization} к первому слагаемому (\ref{eq_1_for_proof_noise_l2_randomization}), получим
        	\begin{equation}
            	\begin{split}
            	\label{eq_2_for_proof_noise_l2_randomization} &  \hspace{-0.1em}
            	    \frac{d}{2\gamma} \mathds{E}_{\xi, e} [\langle (f(x+\gamma e, \xi) - f(x-\gamma e, \xi))e, r \rangle] = 
            	    \\ &  \hspace{-0.1em}
            	    = \frac{d}{2\gamma} \mathds{E}_{\xi, e} [\langle f(x+\gamma e, \xi)e, r \rangle] + \frac{d}{2\gamma} \mathds{E}_{\xi, e} [\langle f(x-\gamma e, \xi)e, r \rangle] = 
            	    \\  &  \hspace{-0.1em}
    	            = \frac{d}{\gamma} \mathds{E}_{e} \left[ \dotprod{\mathds{E}_{\xi} \left[ f(x+\gamma e, \xi) \right] e}{r} \right] = \frac{d}{\gamma} \mathds{E}_{e} [\langle f(x+\gamma e)e, r \rangle] = \dotprod{\nabla f_\gamma(x)}{r}.	    
            	\end{split}
        	\end{equation}
        	Для второго слагаемого (\ref{eq_1_for_proof_noise_l2_randomization}) с учетом $|\delta(x)|\leq \Delta$ получим
        	\begin{equation}
        	\label{eq_3_for_proof_noise_l2_randomization}
        	    \frac{d}{\gamma} \mathds{E}_{e} [\langle (\delta(x+\gamma e) - \delta(x-\gamma e))e, r \rangle] \geq  -\frac{d}{2\gamma} 2\Delta \mathds{E}_{e} [ | \langle e, r \rangle | ] = -\frac{d}{\gamma} \Delta \mathds{E}_{e} [ | \langle e, r \rangle | ].
        	\end{equation}
        	Используя уравнения (\ref{eq_2_for_proof_noise_l2_randomization}) и (\ref{eq_3_for_proof_noise_l2_randomization}) для уравнения (\ref{eq_1_for_proof_noise_l2_randomization}) получим утверждение леммы для $l_2$-рандомизации.
	\end{itemize}
\end{proof}
\begin{corollary} \label{corollary_lemma_3_l1_randomization}
    Заметим, что вектор $\text{sign}(e) = \left( \text{sign}(e_1), ..., \text{sign}(e_d) \right)$ в $l_1$-норме при большой размерности пространства $d$ ведет себя схожим образом, как вектор из независимых радемахеровских случайных величин. Тогда из неравенства Хинчина следует $\mathds{E} [ | \langle \text{sign}(e), r \rangle | ] \leq  \| r \|_2  $, где $r \in \mathds{R}^d$. Применяя данное неравенство к лемме \ref{lemma_3_proof_noise_lp_randomization}, получим
    \begin{equation*}
        \mathds{E}_e \langle [\nabla f_\gamma (x,e)] - \nabla f_\gamma (x), r \rangle \lesssim \frac{d \Delta \| r \|_2}{\gamma}.
    \end{equation*}
\end{corollary}
\begin{corollary} \label{corollary_lemma_3_l2_randomization}
    Применяя утверждение леммы \ref{lemma_2_proof_noise_l2_randomization} к лемме \ref{lemma_3_proof_noise_lp_randomization}, получим то же неравенство, что и в работе \cite{Dvinskikh_2022} для всех $r \in \mathds{R}^d$
	\begin{equation*}
	    \mathds{E}_e \langle [\nabla f_\gamma (x,e)] - \nabla f_\gamma (x), r \rangle \lesssim \frac{\sqrt{d} \Delta \| r \|_2}{\gamma}.
	\end{equation*}
\end{corollary}
Теперь рассмотрим как получить оценки уровня неточности (шума). \\
Из \eqref{Noise_for_l1} и \eqref{Noise_for_l2}, подставляя значение параметра $\gamma$ из следствия \ref{gamma_parameter_clear} ( $\gamma = \frac{\sqrt{d} \varepsilon}{4 M_2}$ для $l_1$-рандомизации, $\gamma = \frac{\varepsilon}{2 M_2}$ для $l_2$-рандомизации) и $R = \|x^0 - x_* \|_2$, получим ограничительные условия на уровень неточности (шума) для $l_1$ и $l_2$-рандомизаций:
\begin{longtable}{ | c | c | }
            \hline
            Схема сглаживания с $l_1$-рандомизацией & Схема сглаживания с $l_2$-рандомизацией  \\ [1ex] 
            \hline
            \hline
            \endhead
            & \\
            $
                \Delta \lesssim \frac{\varepsilon^2}{\sqrt{d} M_2 R}
            $ &
            $
               \Delta \lesssim \frac{\varepsilon^2}{\sqrt{d} M_2 R}
            $ \\ [2ex]
            \hline
            \hline
        \end{longtable}
\begin{remark}
\label{remark_1}
    Стоит отметить, что результаты леммы \ref{lemma_3_proof_noise_lp_randomization} и следствий \ref{corollary_lemma_3_l1_randomization} и \ref{corollary_lemma_3_l2_randomization} для седловых задач такие же и доказываются аналогично. Следовательно оценки уровня неточности для задачи с седловой точкой совпадают с оценками для выпуклой оптимизации. С целью избежания повторений будем ссылаться на это замечания, как на результат получения оценок уровня неточности для седловых задач.
\end{remark}
    


 \section{Численные эксперименты}\label{section:Numerical_result}

Приведем численное сравнение двух техник сглаживания в архитектуре федеративного обучения. Рассматриваем стохастическую негладкую задачу оптимизации на множестве симплекса $Q = \left\{ x \in \mathbb{R}^d \; : \; \| x \|_1 = 1, x \geq 0  \right\}$ с функцией $f: \mathbb{R}^d \rightarrow \mathbb{R}$, определенную следующим образом:
\begin{equation*}
    f(x) = \dotprod{b}{x} + \| x \|_{\infty},
\end{equation*}
где $b \in \mathbb{R}^d$ -- случайный вектор, равномерно распределенный на отрезке $[0,1]$. В качестве метода оптимизации используется безградиентный двухточечный алгоритм Minibatch Accelerated SGD, рассмотренный в разд. \ref{section:Federated_Learning}. На фиг. \ref{figure:2} представлена зависимость ошибки $f(\overline{x})$ на последней итерации от изменения количества локальных вызовов безградиентного оракула $K = 3^0,...,3^{8}$ при различном количестве машин $B = 8$ и $B = 32$, работающих параллельно, где $\overline{x} = \frac{1}{t} \sum_{i =1}^t x_i$, а $t$ -- номер итерации. Количество итераций было выбрано $KN = 3^{8}$. Чем больше $K$, тем меньше количество коммуникационных раундов $N$, обратное также верно. Уровень неточности $\Delta = 0$ (без враждебного шума), размерность задачи $d = 100$.

\begin{figure}[ht]
    \includegraphics[width=1.0\textwidth]{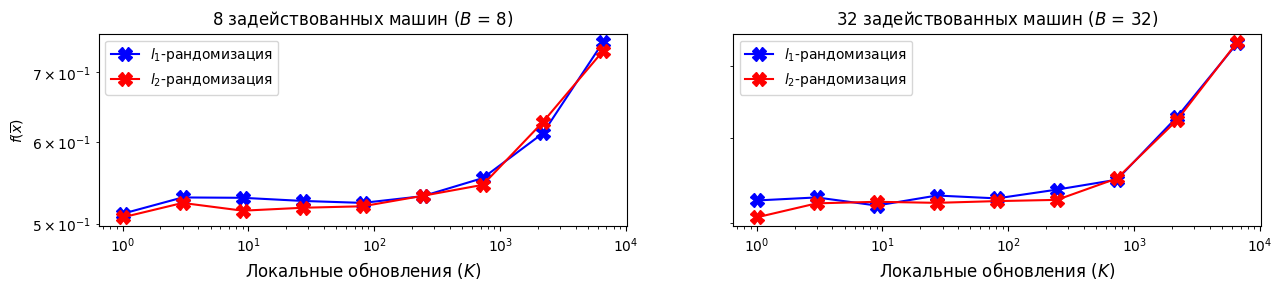}
    \caption{Зависимость ошибки от числа машин $B$ и различных локальных обновлений $K$.}
    \label{figure:2}
\end{figure}

Не трудно заметить из фиг. \ref{figure:2}, что при увеличении числа локальных обновлений $K$ (уменьшении раундов связи $N$), точность ухудшается (тем самым подтверждая теоретические результаты), однако не критично. То есть при решении практических задач, несмотря на теорию, можно брать число локальных обновлений $K \leq 3^6$, чтобы получить достаточно хороший результат. Также стоит обратить внимание, что на практике схема сглаживания с $ l_2$-рандомизацией не только не уступает схеме сглаживания с $ l_1$-рандомизацией, но иногда и превосходит её. Однако возникает интерес узнать в каких случаях схема сглаживания с $ l_1$-рандомизацией будет превосходить схему сглаживания с $ l_2$-рандомизацией. Для этого рассмотрим два случая вычисления безградиентного оракула: с враждебным шумом и без враждебного шума. На фиг. \ref{figure:3} представлена зависимость ошибки $f(\overline{x})$ от количество итераций и уровня неточности (наличия враждебного шума) $\Delta = 0$ и $\Delta = 6 \cdot 10^{-5}$. Число работающих параллельно машин выбрано $B=8$. На каждой машине количество локальных вызов безградиентного оракула равняется $K=3^2$, а количество коммуникационных раундов $3^6$. Таким образом, общее число итераций $NK = 3^8$. Размерность задачи $d=100$, количество запусков равняется 20.

\begin{figure}[ht]
    \includegraphics[width=1.0\textwidth]{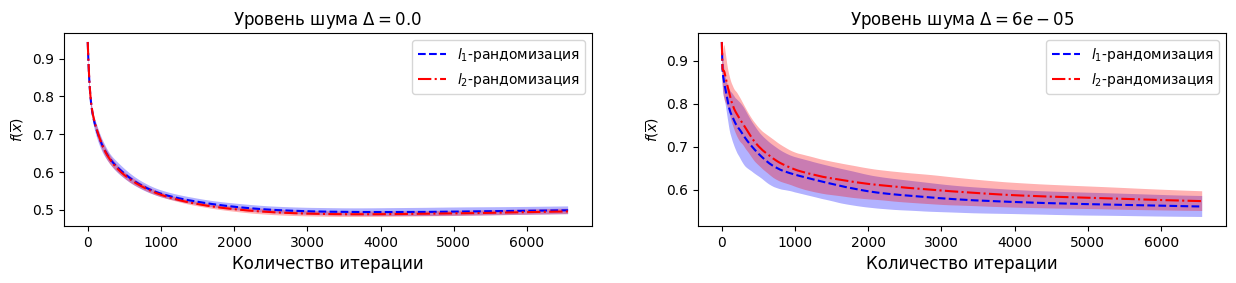}
    \caption{Зависимость ошибки от количества итераций при различных значениях уровня неточности при 20 запусках.}
    \label{figure:3}
\end{figure}

По фиг. \ref{figure:3} можно сделать вывод, что при добавления враждебного шума схема сглаживания с $l_1$ рандомизацией работает лучше, чем схема сглаживания с $l_2$ рандомизацией в архитектуре федеративного обучения.


\section{Заключение}\label{section:conclusions}
В данной работе были получены верхние оценки оптимального алгоритма решения седловых задач в архитектуре федеративного обучения и нашли константу Липшица для градиента в схеме сглаживания с $l_1$-рандомизацией. Применяя сглаживающие схемы, создали оптимальные безградиентые двухточечные и одноточечные алгоритмы c $l_1$ и $l_2$-рандомизацией, благодаря которым можно решать стохастические негладкие выпуклые задачи оптимизации и выпукло-вогнутые задачи оптимизации в настройке федеративного обучения с неточным безградиентным оракулом. Показали при каких условия схема сглаживания с $l_1$-рандомизацией  работает лучше, чем с $l_2$-рандомизацией в архитектуре федеративного обучения. 
Было показано на практике, что   число локальных обновлений может быть увеличено за счет уменьшения количества коммуникационных раундов связи, при этом  общее количество итераций остается прежним.




\printbibliography

\newpage

\appendix
\renewcommand{\theequation}{П.\arabic{equation}}
\section*{{$\qquad \qquad \qquad \qquad \qquad \qquad \qquad \qquad \qquad$\itПРИЛОЖЕНИЕ} } \label{sec:apCATD}

\section{Верхние оценки}
\label{Upper_bound}
Рассмотрим гладкую выпукло-вогнутую седловую задачу
\begin{equation}
    \label{smooth_Saddle-point_problem}
    \min_{x \in Q_x \subseteq \mathbb{R}^{d_x}} \max_{y \in Q_y \subseteq \mathbb{R}^{d_y}} f(x,y),
\end{equation}
где $f:Q_x \times Q_y \rightarrow \mathbb{R}$~--- выпукло-вогнутая липшицева непрерывная функция, $Q_x$ и $Q_y$~--- выпуклые множества. Для простоты изложения введем множество $Q_z = Q_x \times Q_y$, $z = (x,y)$ и монотонный оператор $F$:
\begin{equation}
    \label{Operator_F}
    F(z) = F(x, y) = 
    \begin{pmatrix}
        \nabla_x f(x,y) \\
        -\nabla_y f(x,y)
    \end{pmatrix}.
\end{equation}
Тогда $z_* \in Q_z$~--- решение вариационного неравенства (ВН) выглядит следующим образом:
\begin{equation}
    \label{Var_ineq}
    \dotprod{F(z)}{ z_* - z} \leq 0 \;\;\; \forall z \in Q_z.
\end{equation}
Оценивать неточность возможного решения $z \in Q_z$ будем по ошибке
\begin{equation}
    \label{Err_vi}
    \text{Err}_{\text{vi}}(z) = \max_{u \in Q_z} \dotprod{F(u)}{z - u}.
\end{equation}
В дальнейшем накладываем на $F$, помимо монотонности, требование
\begin{equation}
    \label{requirement_on_F}
    \forall(z, z' \in Q_z): \| F(z) - F(z') \|_* \leq L \| z - z' \| + V,
\end{equation}
где $L \geq 0$, $V \geq 0$~--- известные константы, $ \| \cdot \|_* = \max_{z: \| z \| \leq 1} \dotprod{\cdot}{z}.$ 
И предполагаем, что $F(z,\xi_t)$ является несмещенным и имеет ограниченую дисперсию, т.е. $\forall z \in Q_z$ справедливо
\begin{equation*}
    \mathbb{E}[F(z,\xi_t)] = F(z), \;\;\; \mathbb{E}[\| F(z,\xi_t) - F(z) \|^2_*] \leq \sigma^2.
\end{equation*}

Для распространения основного результата статьи \cite{Woodworth_2021} на седловые задачи и ВН в качестве алгоритма выберем Stochastic Mirror Prox (SMP), описанного в работе \cite{Juditsky_2011}.  В дальнейшем рассмотрим два метода архитектуры FL, предложенные в \cite{Woodworth_2021} и назовем их Minibatch SMP и Single-Machine SMP. Для этого воспользуемся известными результатами и предположениями.
\begin{algorithm}[ht]
	\caption{Алгоритм SMP.}
	\begin{algorithmic}[1]
		\STATE  Инициализация. Выберите $r_0 \in Z^0$ и размер шага $\eta_\tau, \,\,\, 1 \leq \tau \leq t$. 
		\STATE Шаг $\tau, \,\,\, \tau = 1,2,...,t$: По известному $r_{\tau - 1} \in Z^0$, вычислите
	    \begin{equation}
	        \begin{cases*}
	        \label{eq_algorithm_SMP}
                w_\tau = P_{r_{\tau - 1}}(\eta_\tau \hat{F}(r_{\tau - 1})),\\
                r_\tau = P_{r_{\tau - 1}}(\eta_\tau \hat{F}(w_{\tau})).
            \end{cases*}
	    \end{equation}
	    Когда $\tau < t$, выполнять цикл до шага $t+1$. 
	    \STATE На шаге $t$ выведите
	    \begin{equation*}
	        \hat{z}_t = \left[ \sum_{\tau = 1}^{t} \eta_\tau \right]^{-1} \sum_{\tau = 1}^{t} \eta_\tau w_\tau.
	    \end{equation*}
	\end{algorithmic}
	\label{alg:SMP}
\end{algorithm}
\begin{assumption}
    \label{assumption_SMP_1}
    Для всех $z \in Q_z$ с $\mu \in [0, \infty)$ имеем
    \begin{equation*}
        \| \mathbb{E}[F(z,\xi_t) - F(z)] \|_* \leq \mu, \;\;\; \mathbb{E}[\| F(z,\xi_t) - F(z) \|^2_*] \leq \sigma^2.
    \end{equation*}
\end{assumption}
\begin{assumption}
    \label{assumption_SMP_2}
    Для всех $z \in Q_z$ и всех $t$ имеем
    \begin{equation*}
       \mathbb{E} \left[ \exp \left\{ \| F(z,\xi_t) - F(z) \|^2_* / \sigma^2 \right\} \right] \leq \exp \left\{ 1 \right\}. 
    \end{equation*}
\end{assumption}
\begin{lemma}[см. \cite{Juditsky_2011}]
    \label{lemma_for_SMP}
    Пусть ВН (\ref{Var_ineq}) с монотонным оператором $F$ (\ref{Operator_F}), удовлетворяющим требованию (\ref{requirement_on_F}) решается с помощью $t$ шагового алгоритма \ref{alg:SMP}, используя стохастический оракул ($\hat{F} = F(z,\xi_t)$), и пусть размеры шага $\eta_t \equiv \eta$ удовлетворяют $0 \leq \eta \leq \frac{1}{\sqrt{3}L}$. Тогда
    \begin{enumerate}
        \item[(i)] с предположием \ref{assumption_SMP_1} имеем
        
        \begin{equation*}
            \mathbb{E} \left\{ \text{Err}_{\text{vi}} (\hat{z}_t) \right\} \leq K_0(t) \equiv  \left[ \frac{R^2}{t \eta} + \frac{7 \eta}{2} \left[ V^2 + 2\sigma^2 \right]  \right] + 2 \mu R;
        \end{equation*}
        
        \item[(ii)] с предположениями \ref{assumption_SMP_1}, \ref{assumption_SMP_2} для любого $\Lambda > 0$
        
        \begin{equation*}
            \text{Prob} \left\{ \text{Err}_{\text{vi}}(\hat{z}_t) > K_0(t) + \Lambda K_1(t) \right\} \leq \exp \left\{-\Lambda^2/3 \right\}+\exp\left\{-\Lambda t \right\},
        \end{equation*}
        где
        \begin{equation*}
            K_1(t) = \frac{7 \sigma^2 \eta}{2} + \frac{2 \sigma R}{\sqrt{t}}.
        \end{equation*}
    \end{enumerate}
\end{lemma}
\begin{corollary}
    \label{Optima_alg_VI}
    Используя результаты леммы \ref{lemma_for_SMP} расспространим идею оптимального алгоритма \cite{Woodworth_2021} и получим верхние оценки.
    
    \begin{itemize}
        \item Minibatch SMP
        
         Данный алгоритм выполняет $N$ итераций SMP, используя пробатченные градиенты размера $BK$. Во время каждого раунда связи каждая машина вычисляет $K$ стохастического оракула, затем машины посылают свои минибатчи, усредняя в один большой минибатч размером $BK$, затем они обновляют $w_\tau$ и $r_\tau$ в соответствии с (\ref{eq_algorithm_SMP}).  Так как вычисление пробатченного стохастического оракула уменьшает дисперсию в $BK$ раз, то обозначим $\sigma^2_{BK} = \frac{\sigma^2}{BK}$. Предположим, что размер шага $\eta_t \equiv \eta $ алгоритма \ref{alg:SMP} равен
         \begin{equation}
            \label{step_MB_SMP}
            \eta_t = \min \left[ \frac{1}{\sqrt{3}L}, 7 R \sqrt{\frac{2BK}{7N(V^2 + 2\sigma^2)}} \right],
        \end{equation}
        где $N$~--- заданное число итераций. Тогда с предположением \ref{assumption_SMP_1}
        \begin{equation}
        \label{upper_MB_SMP}
            \mathbb{E} \left\{ \text{Err}_{\text{vi}} (\hat{z}_N) \right\} \leq K^*_0(N) \equiv \max \left[ \frac{7}{4} \frac{L R^2}{N}, 7R \sqrt{\frac{V^2 + 2 \sigma^2 }{3 B K N}}  \right] + 2 \mu R.
        \end{equation}
        Так как верхняя граница (\ref{upper_MB_SMP}) для ошибки алгоритма \ref{alg:SMP} со стратегией шага (\ref{step_MB_SMP}) зависит схожим образом от $\sigma$ и $V$, то $V$ можем заменить на $\sigma$. Тогда  
        \begin{equation*}
            \mathbb{E} \left\{ \text{Err}_{\text{vi}} (\hat{z}_N) \right\} \leq K^*_0(N) \equiv \max \left[ \frac{7}{4} \frac{L R^2}{N}, 7 \frac{ \sigma R}{\sqrt{BKN}}  \right] + 2 \mu R,
        \end{equation*}
    если выполняются предположения \ref{assumption_SMP_1} и \ref{assumption_SMP_2}, то
        \begin{equation*}
            \text{Prob} \left\{ \text{Err}_{\text{vi}}(\hat{z}_N) > K^*_0(N) + \Lambda K_1^*(N) \right\} \leq \exp \left\{-\Lambda^2/3 \right\}+\exp\left\{-\Lambda N \right\},
        \end{equation*}
        где
        \begin{equation*}
            K_1^*(N) = \frac{7}{2} \frac{\sigma_{BK} R}{\sqrt{N}} = \frac{7}{2} \frac{\sigma R}{\sqrt{BKN}}.
        \end{equation*}
        
    \item Single-Machine SMP
    
    Данный алгоритм в отличие от Mini-batch SMP игнорирует $B-1$ машин и выполняет $KN$ шагов алгоритма SMP. Тогда предположим, что размер шага $\eta_t$ алгоритма SMP равен
    \begin{equation}
    \label{step_SM_SMP}
        \eta_t = \min \left[ \frac{1}{\sqrt{3}L}, 7 R \sqrt{\frac{2}{7KN(V^2 + 2\sigma^2)}} \right],
    \end{equation}
    где $KN$ - заданное число итераций. Тогда с предположением \ref{assumption_SMP_1}, 
	\begin{equation}
	\label{upper_SM_SMP}
            \mathbb{E} \left\{ \text{Err}_{\text{vi}} (\hat{z}_{KN}) \right\} \leq K^*_0(KN) \equiv \max \left[ \frac{7}{4} \frac{L R^2}{KN}, 7R \sqrt{\frac{V^2 + 2\sigma^2 }{3KN}}  \right] + 2 \mu R,
        \end{equation}
     Так как верхняя граница (\ref{upper_SM_SMP}) для ошибки алгоритма \ref{alg:SMP} со стратегией шага (\ref{step_SM_SMP}) зависит схожим образом от $\sigma$ и $V$, то $V$ можем заменить на $\sigma$. Тогда
     \begin{equation*}
            \mathbb{E} \left\{ \text{Err}_{\text{vi}} (\hat{z}_{KN}) \right\} \leq K^*_0(KN) \equiv \max \left[ \frac{7}{4} \frac{L R^2}{KN}, 7 \frac{\sigma R}{\sqrt{KN}}  \right] + 2 \mu R,
        \end{equation*}
    если выполняются предположения \ref{assumption_SMP_1} и \ref{assumption_SMP_2}, то
        \begin{equation*}
            \text{Prob} \left\{ \text{Err}_{\text{vi}}(\hat{z}_{KN}) > K^*_0(KN) + \Lambda K_1^*(KN) \right\} \leq \exp \left\{-\Lambda^2/3 \right\}+\exp\left\{-\Lambda KN \right\},
        \end{equation*}
        где
        \begin{equation*}
            K_1^*(KN) = \frac{7}{2} \frac{\sigma R}{\sqrt{KN}}.
        \end{equation*}
    \end{itemize}
\end{corollary}
Запишем лемму, используя такой же подход что и Woodworth \cite{Woodworth_2021}, чтобы объединить два алгоритма (Mini-batch SMP и Single-Machine SMP) в один оптимальный алгоритм.
\begin{corollary}
    \label{lemma_opt_alg}
    Для любых $L, R, \sigma, K, N, M > 0$ алгоритм использует \textit{Minibatch SMP}, когда $K \leq \frac{\sigma^2 N}{L^2 R^2}$, и использует \textit{Single-Machine SMP}, когда $K > \frac{\sigma^2 N}{L^2 R^2}$, тогда с предположением \ref{assumption_SMP_1}
    \begin{equation*}
        \mathbb{E} \left\{ \text{Err}_{\text{vi}} (\hat{z}) \right\} \leq c \cdot \left(  \frac{L R^2}{KN} + \frac{ \sigma R}{\sqrt{BKN}} + \min \left[ \frac{L R^2}{N} , \frac{\sigma R}{\sqrt{KN}} \right]    \right) + 2 \mu R,
    \end{equation*}
  где $c$~--- некоторая числовая константа.
\end{corollary}


\section{Доказательство теоремы \ref{theorem_1}}
\label{proof_theorem_1}
Приведем полное доказательство теоремы \ref{theorem_1}. Для этого разобьем доказательство на две части: доказательство для $l_1$-рандомизации и доказательство для $l_2$-рандомизации.

\underline{Для $l_1$-рандомизации имеем}:
\begin{itemize}
    \item Minibatch Accelerated SGD
    
    Данный алгоритм после $N$ раундов связи дает скорость сходимости для $f_\gamma (x)$ (см. \cite{Woodworth_2021, Lan_2012}) в соответствии со следствием \ref{corollary_lemma_3_l1_randomization}:
    \begin{equation*}
        \mathds{E}[f_{\gamma}(x_{ag}^{N+1}) - f(x_{*})] \leq \frac{4L_{f_\gamma}R^2}{N^2} + \frac{4 \sigma R}{\sqrt{BKN}} + \frac{d \Delta R}{\gamma},
    \end{equation*}
    где $x_{*}(\gamma) = \underset{x \in Q_\gamma}{\mathrm{argmin}}  f_{\gamma}(x)$.
    
    Если имеем  $\frac{\varepsilon}{2}$-точность для функции $f_{\gamma}(x)$ с $\gamma = \frac{\sqrt{d} \varepsilon}{4 M_2}$ (из следствия \ref{gamma_parameter_clear}), то имеем $\varepsilon$-точность для функции $f(x)$:
    \begin{equation*}
        f(x_{ag}^{N+1}) - f(x_{*}) \leq f(x_{ag}^{N+1}) - f(x_{*}(\gamma)) \leq f_{\gamma}(x_{ag}^{N+1}) - f(x_{*}(\gamma)) + \frac{2}{\sqrt{d}} \gamma M_2  \leq \frac{\varepsilon}{2 } + \frac{\varepsilon}{2} =\varepsilon.
    \end{equation*}
    Чтобы была $\frac{\varepsilon}{2}$-точность для $f_{\gamma}(x) $ нужно
	\begin{equation}
	    \label{eq1_for_proof1:1_l1_randomization}
	    \frac{d \Delta R}{\gamma} \leq \frac{\varepsilon}{6},
	\end{equation}
	\begin{equation}
	    \label{eq2_for_proof1:1_l1_randomization}
	    \frac{4L_{f_\gamma}R^2}{ N^2} \leq \frac{\varepsilon}{6},
	\end{equation}
	
	\begin{equation}
	    \label{eq3_for_proof1:1_l1_randomization}
	    \frac{ 4\sigma R}{\sqrt{B K N}} \leq \frac{\varepsilon}{6}.
	\end{equation}

        Подставляя $L_{f_\gamma} =  \frac{d M}{\gamma}$ (из следствия \ref{Const_Lipshitz_grad}), где $\gamma = \frac{\sqrt{d} \varepsilon}{4 M_2}$  и $ \sigma^2 = 2 \kappa(p,d) M_2^2$ (из следствия \ref{sigma_for_two_point}) в неравенства (\ref{eq1_for_proof1:1_l1_randomization})-(\ref{eq3_for_proof1:1_l1_randomization}) получим
        \begin{equation*}
            \Delta \leq \frac{\gamma \varepsilon }{6 d R} \,\,\, \Rightarrow \,\,\, \Delta \leq \frac{ \varepsilon^2 }{24 \sqrt{d} M_2 R} \,\,\, \Rightarrow
        \end{equation*}
        \begin{equation*}
            \Delta = O\left( \frac{ \varepsilon^2 }{\sqrt{d} M_2 R} \right)
        \end{equation*}
        уровень неточности,
        \begin{equation*}
            N^2 \geq \frac{96 \sqrt{d} M M_2 R^2}{\varepsilon^2} \Rightarrow N \geq \frac{4 \sqrt{6} d^{1/4} \sqrt{M M_2} R}{\varepsilon} \Rightarrow
        \end{equation*}
        \begin{equation*}
            \Rightarrow N = O\left( \frac{d^{1/4} \sqrt{M M_2} R}{\varepsilon} \right)
        \end{equation*}
        количество коммуникационных раундов,
        \begin{equation*}
            B \geq \frac{576 \sigma^2 R^2}{KN \varepsilon^2} \Rightarrow B \geq \frac{1152 \kappa(p,d) M_2^2 R^2}{KN \varepsilon^2} \Rightarrow
        \end{equation*}
        \begin{equation*}
            \Rightarrow B = O\left( \frac{\kappa(p,d) M_2^2 R^2}{KN \varepsilon^2} \right)
        \end{equation*}
        количество работающих параллельно машин и
        \begin{equation*}
             T = N \cdot K \cdot B = \frac{1152 \kappa(p,d) M_2^2 R^2}{ \varepsilon^2} \,\,\, \Rightarrow
        \end{equation*}
        \begin{equation*}
            \Rightarrow \,\,\, T =  O \left(  \frac{\kappa(p,d) M_2^2 R^2}{ \varepsilon^2}   \right) =
            \begin{cases}
                 O \left(  \frac{d M_2^2 R^2}{ \varepsilon^2}   \right), & p = 2 \;\;\; (q = 2),\\
                 O \left(  \frac{ M_2^2 R^2}{ \varepsilon^2}   \right), & p = 1 \;\;\; (q = \infty),
            \end{cases}
        \end{equation*}
        общее количество вызовов двухточечного безградиентного оракула;

    \item Single-Machine Accelerated SGD
    
    Данный алгоритм после $NK$ итераций дает скорость сходимости для $f_\gamma (x)$ (см. \cite{Woodworth_2021, Lan_2012}) в соответствии со следствием \ref{corollary_lemma_3_l1_randomization}:
    \begin{equation*}
        \mathds{E}[f_{\gamma}(x_{ag}^{NK+1}) - f(x_{*})] \leq \frac{4L_{f_\gamma}R^2}{N^2K^2} + \frac{4 \sigma R}{\sqrt{NK}} + \frac{d \Delta R}{\gamma},
    \end{equation*}
    где $x_{*}(\gamma) = \underset{x \in Q_\gamma}{\mathrm{argmin}}  f_{\gamma}(x)$.
    
    Если имеем  $\frac{\varepsilon}{2}$-точность для функции $f_{\gamma}(x)$ с $\gamma = \frac{\sqrt{d} \varepsilon}{4 M_2}$ (из следствия \ref{gamma_parameter_clear}), то имеем $\varepsilon$-точность для функции $f(x)$:
    \begin{equation*}
        f(x_{ag}^{N+1}) - f(x_{*}) \leq f(x_{ag}^{N+1}) - f(x_{*}(\gamma)) \leq f_{\gamma}(x_{ag}^{N+1}) - f(x_{*}(\gamma)) + \frac{2}{\sqrt{d}} \gamma M_2  \leq \frac{\varepsilon}{2 } + \frac{\varepsilon}{2} =\varepsilon.
    \end{equation*}
    Чтобы была $\frac{\varepsilon}{2}$-точность для $f_{\gamma}(x) $ нужно
	\begin{equation}
	    \label{eq1_for_proof1:2_l1_randomization}
	    \frac{d \Delta R}{\gamma} \leq \frac{\varepsilon}{6}
	\end{equation}
	\begin{equation}
	    \label{eq2_for_proof1:2_l1_randomization}
	    \frac{4L_{f_\gamma}R^2}{K^2 N^2} \leq \frac{\varepsilon}{6},
	\end{equation}
	
	\begin{equation}
	    \label{eq3_for_proof1:2_l1_randomization}
	    \frac{ 4\sigma R}{\sqrt{K N}} \leq \frac{\varepsilon}{6}
	\end{equation}

        Подставляя $L_{f_\gamma} =  \frac{d M}{\gamma}$ (из следствия \ref{Const_Lipshitz_grad}), где $\gamma = \frac{\sqrt{d} \varepsilon}{4 M_2}$  и $ \sigma^2 = 2 \kappa(p,d) M_2^2$ (из следствия \ref{sigma_for_two_point}), в неравенства (\ref{eq1_for_proof1:2_l1_randomization})-(\ref{eq3_for_proof1:2_l1_randomization}), получим:
        \begin{equation*}
            \Delta \leq \frac{\gamma \varepsilon }{6 d R} \,\,\, \Rightarrow \,\,\, \Delta \leq \frac{ \varepsilon^2 }{24 \sqrt{d} M_2 R} \,\,\, \Rightarrow
        \end{equation*}
        \begin{equation*}
            \Delta = O\left( \frac{ \varepsilon^2 }{\sqrt{d} M_2 R} \right)
        \end{equation*}
        уровень неточности,
        \begin{equation*}
            NK  \geq \frac{576 \sigma^2 R^2}{\varepsilon^2} \Rightarrow N K  \geq \frac{1152 \kappa(p,d) M_2^2 R^2}{ \varepsilon^2}.
        \end{equation*}
        Так как $N$ напрямую зависит от $K$, то количество коммуникачионных раундов можно взять $N = 1$, тогда
        \begin{equation*}
            NK = K = O \left( \frac{ \kappa(p,d) M_2^2 R^2}{ \varepsilon^2} \right)
        \end{equation*}
        количество локальных вызовов безградиентного оракула и
        \begin{equation*}
             T = N \cdot K \cdot B = \frac{1152 \kappa(p,d) M_2^2 R^2}{ \varepsilon^2} \,\,\, \Rightarrow
        \end{equation*}
        \begin{equation*}
            \Rightarrow \,\,\, T = O \left(  \frac{\kappa(p,d) M_2^2 R^2}{ \varepsilon^2}   \right) =
            \begin{cases}
                 O \left(  \frac{d M_2^2 R^2}{ \varepsilon^2}   \right), & p = 2 \;\;\; (q = 2),\\
                 O \left(  \frac{ M_2^2 R^2}{ \varepsilon^2}   \right), & p = 1 \;\;\; (q = \infty),
            \end{cases}
        \end{equation*}
        общее количество вызовов двухточечного безградиентного оракула;
    
    \item Local-AC-SA
    
    Данный алгоритм после $N$ раундов связи дает скорость сходимости для $f_\gamma (x)$ (см. \cite{Woodworth_2020, Lan_2012}) в соответствии со следствием \ref{corollary_lemma_3_l1_randomization}:
    \begin{equation*}
        \mathds{E}[f_{\gamma}(x_{ag}^{N+1}) - f(x_{*})] \leq \frac{4L_{f_\gamma} R^2}{K^2 N^2} + \frac{4 \sigma R}{\sqrt{B K N}} + \frac{d \Delta R}{\gamma},
    \end{equation*}
    где $x_{*}(\gamma) = \underset{x \in Q_\gamma}{\mathrm{argmin}}  f_{\gamma}(x)$.
    
    Если имеем  $\frac{\varepsilon}{2}$-точность для функции $f_{\gamma}(x)$ с $\gamma = \frac{\sqrt{d} \varepsilon}{4 M_2}$ (из следствия \ref{gamma_parameter_clear}), то имеем $\varepsilon$-точность для функции $f(x)$:
    \begin{equation*}
        f(x_{ag}^{N+1}) - f(x_{*}) \leq f(x_{ag}^{N+1}) - f(x_{*}(\gamma)) \leq f_{\gamma}(x_{ag}^{N+1}) - f(x_{*}(\gamma)) + \frac{2}{\sqrt{d}} \gamma M_2  \leq \frac{\varepsilon}{2 } + \frac{\varepsilon}{2} =\varepsilon.
    \end{equation*}
    Чтобы была $\frac{\varepsilon}{2}$-точность для $f_{\gamma}(x) $ нужно
	\begin{equation}
	    \label{eq1_for_proof1:3_l1_randomization}
	    \frac{d \Delta R}{\gamma} \leq \frac{\varepsilon}{6},
	\end{equation}
	\begin{equation}
	    \label{eq2_for_proof1:3_l1_randomization}
	    \frac{4L_{f_\gamma}R^2}{K^2 N^2} \leq \frac{\varepsilon}{6},
	\end{equation}
	
	\begin{equation}
	    \label{eq3_for_proof1:3_l1_randomization}
	    \frac{ 4\sigma R}{\sqrt{B K N}} \leq \frac{\varepsilon}{6}.
	\end{equation}

        Подставляя $L_{f_\gamma} =  \frac{d M}{\gamma}$ (из следствия \ref{Const_Lipshitz_grad}), где $\gamma = \frac{\sqrt{d} \varepsilon}{4 M_2}$  и $ \sigma^2 = 2 \kappa(p,d) M_2^2$ (из следствия \ref{sigma_for_two_point}), в неравенства (\ref{eq1_for_proof1:3_l1_randomization})-(\ref{eq3_for_proof1:3_l1_randomization}), получим:
        \begin{equation*}
            \Delta \leq \frac{\gamma \varepsilon }{6 d R} \,\,\, \Rightarrow \,\,\, \Delta \leq \frac{ \varepsilon^2 }{24 \sqrt{d} M_2 R} \,\,\, \Rightarrow
        \end{equation*}
        \begin{equation*}
            \Delta = O\left( \frac{ \varepsilon^2 }{\sqrt{d} M_2 R} \right)
        \end{equation*}
        уровень неточности,
        \begin{equation*}
            N^2 K^2 \geq \frac{96 \sqrt{d} M M_2 R^2}{\varepsilon^2} \Rightarrow NK \geq \frac{4 \sqrt{6} d^{1/4} \sqrt{M M_2} R}{\varepsilon}.
        \end{equation*}
        Так как $N$ напрямую зависит от $K$, то количество коммуникачионных раундов можно взять $N = 1$, тогда
        \begin{equation*}
            NK = K = O \left( \frac{d^{1/4} \sqrt{M M_2} R}{\varepsilon} \right)
        \end{equation*}
        количество локальных вызовов безградиентного оракула,
        \begin{equation*}
            B \geq \frac{576 \sigma^2 R^2}{KN \varepsilon^2} \Rightarrow B \geq \frac{1152 \kappa(p,d) M_2^2 R^2}{KN \varepsilon^2} \Rightarrow
        \end{equation*}
        \begin{equation*}
            \Rightarrow B = O\left( \frac{\kappa(p,d) M_2^2 R^2}{KN \varepsilon^2} \right)
        \end{equation*}
        количество работающих параллельно машин и
        \begin{equation*}
             T = N \cdot K \cdot B = \frac{1152 \kappa(p,d) M_2^2 R^2}{ \varepsilon^2} \,\,\, \Rightarrow
        \end{equation*}
        \begin{equation*}
            \Rightarrow \,\,\, T =  O \left(  \frac{\kappa(p,d) M_2^2 R^2}{ \varepsilon^2}   \right) =
            \begin{cases}
                 O \left(  \frac{d M_2^2 R^2}{ \varepsilon^2}   \right), & p = 2 \;\;\; (q = 2),\\
                 O \left(  \frac{ M_2^2 R^2}{ \varepsilon^2}   \right), & p = 1 \;\;\; (q = \infty),
            \end{cases}
        \end{equation*}
        общее количество вызовов двухточечного безградиентного оракула;

    \item Federated Accelerated SGD (FedAc)
    
    Данный алгоритм после $N$ раундов связи дает скорость сходимости для $f_\gamma (x)$ (см. \cite{Yuan_Ma_2020}) в соответствии со следствием \ref{corollary_lemma_3_l1_randomization}:
    \begin{equation*}
        \mathds{E}[f_{\gamma}(x_{ag}^{N+1}) - f(x_{*})] \leq \frac{L_{f_\gamma} R^2}{K N^2} + \frac{ \sigma R}{\sqrt{B K N}} + \min \left\{ \frac{L_{f_\gamma}^{1/3} \sigma^{2/3} R^{4/3}}{K^{1/3} N}, \frac{L_{f_\gamma}^{1/2} \sigma^{1/2} R^{3/2}}{K^{1/4} N} \right\} + \frac{d \Delta R}{\gamma},
    \end{equation*}
    где $x_{*}(\gamma) = \underset{x \in Q_\gamma}{\mathrm{argmin}}  f_{\gamma}(x)$.
    
    Если имеем  $\frac{\varepsilon}{2}$-точность для функции $f_{\gamma}(x)$ с $\gamma = \frac{\sqrt{d} \varepsilon}{4 M_2}$ (из следствия \ref{gamma_parameter_clear}), то имеем $\varepsilon$-точность для функции $f(x)$:
    \begin{equation*}
        f(x_{ag}^{N+1}) - f(x_{*}) \leq f(x_{ag}^{N+1}) - f(x_{*}(\gamma)) \leq f_{\gamma}(x_{ag}^{N+1}) - f(x_{*}(\gamma)) + \frac{2}{\sqrt{d}} \gamma M_2  \leq \frac{\varepsilon}{2 } + \frac{\varepsilon}{2} =\varepsilon.
    \end{equation*}
    Чтобы была $\frac{\varepsilon}{2}$-точность для $f_{\gamma}(x) $ нужно
	\begin{equation}
	    \label{eq1_for_proof1:4_l1_randomization}
	    \frac{d \Delta R}{\gamma} \leq \frac{\varepsilon}{8},
	\end{equation}
	\begin{equation}
	    \label{eq2_for_proof1:4_l1_randomization}
	    \frac{L_{f_\gamma}R^2}{K N^2} \leq \frac{\varepsilon}{8},
	\end{equation}
	
	\begin{equation}
	    \label{eq3_for_proof1:4_l1_randomization}
	    \frac{ \sigma R}{\sqrt{B K N}} \leq \frac{\varepsilon}{8},
	\end{equation}
	\begin{equation}
	    \label{eq4_for_proof1:4_l1_randomization}
	    \min \left\{ \frac{L_{f_\gamma}^{1/3} \sigma^{2/3} R^{4/3}}{K^{1/3} N}, \frac{L_{f_\gamma}^{1/2} \sigma^{1/2} R^{3/2}}{K^{1/4} N} \right\} \leq \frac{\varepsilon}{8}.
	\end{equation}

        Подставляя $L_{f_\gamma} =  \frac{d M}{\gamma}$ (из следствия \ref{Const_Lipshitz_grad}), где $\gamma = \frac{\sqrt{d} \varepsilon}{4 M_2}$  и $ \sigma^2 = 2 \kappa(p,d) M_2^2$ (из следствия \ref{sigma_for_two_point}), в неравенства (\ref{eq1_for_proof1:4_l1_randomization})-(\ref{eq4_for_proof1:4_l1_randomization}), получим
        \begin{equation*}
            \Delta \leq \frac{\gamma \varepsilon }{6 d R} \,\,\, \Rightarrow \,\,\, \Delta \leq \frac{ \varepsilon^2 }{24 \sqrt{d} M_2 R} \,\,\, \Rightarrow
        \end{equation*}
        \begin{equation*}
            \Delta = O\left( \frac{ \varepsilon^2 }{\sqrt{d} M_2 R} \right)
        \end{equation*}
        уровень неточности,
        \begin{equation*}
            NK \geq \frac{8^{1/2} K^{1/2} L^{1/2}_{f_\gamma}R}{ \varepsilon^{1/2}} \,\,\,  \Rightarrow \,\,\,  NK \geq \frac{16^{1/2} K^{1/2} d^{1/4} M^{1/2} M_2^{1/2} R}{ \varepsilon} \,\,\,  \Rightarrow \,\,\, NK \gtrsim \frac{ K^{1/2} }{ \varepsilon},
        \end{equation*}
        \begin{equation*}
            NK \geq \frac{64 \sigma^2 R^2}{B \varepsilon^2}  \,\,\, \Rightarrow \,\,\, NK \geq \frac{72 \kappa(p,d) d M_2^2 R^2}{B \varepsilon^2} \,\,\,  \Rightarrow \,\,\, NK \gtrsim \frac{1}{B \varepsilon^2}
        \end{equation*}
        и
        \begin{equation*}
            NK \geq \min \left\{ \frac{K^{2/3} L_{f_\gamma}^{1/3} \sigma^{2/3} R^{4/3}}{\varepsilon}, \frac{K^{3/4} L_{f_\gamma}^{1/2} \sigma^{1/2} R^{3/2}}{ \varepsilon} \right\} \,\,\, \Rightarrow 
        \end{equation*}
        \begin{equation*}
            \Rightarrow \,\,\, NK \geq \min \left\{  \frac{2^{1/3} K^{2/3} d^{1/6} (M M_2)^{1/3} \sigma^{2/3} R^{4/3}}{ \varepsilon^{4/3}}, \frac{2^{1/2} K^{3/4} d^{1/4} (M M_2)^{1/2} \sigma^{1/2} R^{3/2}}{ \varepsilon^{3/2}} \right\} \,\,\, \Rightarrow 
        \end{equation*}
        \begin{equation*}
            \Rightarrow \,\,\, NK \geq \frac{2^{1/3} K^{2/3} d^{1/6} (M M_2)^{1/3} \sigma^{2/3} R^{4/3}}{ \varepsilon^{4/3}} \,\,\,  \Rightarrow \,\,\, NK \gtrsim \frac{K^{2/3}}{ \varepsilon^{4/3}}.
        \end{equation*}
       Таким образом, наименьшее число коммуникационных раундов при условиии, что $NK \in [1, \varepsilon^{-2}]$ и $T \in [1, \varepsilon^{-2}]$, будет выглядить следующим образом:
        \begin{equation*}
            N \sim \frac{1}{\varepsilon}: \;\;\;\;\; N \geq \frac{8 L_{f_\gamma}^{1/3} \sigma^{2/3} R^{4/3}}{K^{1/3} \varepsilon} \Rightarrow \;\;\; N = O\left( \frac{d^{1/6} (\kappa(p,d) M)^{1/3} M_2 R^{4/3}}{K^{1/3} \varepsilon^{4/3}} \right)
        \end{equation*}
        тогда
        \begin{equation*}
            K \sim \frac{1}{\varepsilon}: \;\;\;\;\; K \geq \frac{ \sigma^2 R^2}{B N \varepsilon^2} \Rightarrow \;\;\; K = O\left( \frac{ \kappa(p,d) M^2_2 R^2}{B N \varepsilon^2} \right)
        \end{equation*}
        количество локальных вызовов безградиентного оракула, $B = 1$~--- количество работающих параллельно машин и
        \begin{equation*}
             T \sim \frac{1}{\varepsilon^2}: \;\;\;\;\; T = N \cdot K \cdot B =  O \left(  \frac{\kappa(p,d) M_2^2 R^2}{ \varepsilon^2}   \right) =
            \begin{cases}
                 O \left(  \frac{d M_2^2 R^2}{ \varepsilon^2}   \right), & p = 2 \;\;\; (q = 2),\\
                 O \left(  \frac{M_2^2 R^2}{ \varepsilon^2}   \right), & p = 1 \;\;\; (q = \infty),
            \end{cases}
        \end{equation*}
        общее количество вызовов двухточечного безградиентного оракула.
\end{itemize}

\newpage
\underline{Для $l_2$-рандомизации имеем}:
\begin{itemize}
    \item Minibatch Accelerated SGD
    
    Данный алгоритм после $N$ раундов связи дает скорость сходимости для $f_\gamma (x)$ (см. \cite{Woodworth_2021, Lan_2012}) в соответствии со следствием \ref{corollary_lemma_3_l2_randomization}:
    \begin{equation*}
        \mathds{E}[f_{\gamma}(x_{ag}^{N+1}) - f(x_{*})] \leq \frac{4L_{f_\gamma}R^2}{N^2} + \frac{4 \sigma R}{\sqrt{BKN}} + \frac{\sqrt{d} \Delta R}{\gamma},
    \end{equation*}
    где $x_{*}(\gamma) = \underset{x \in Q_\gamma}{\mathrm{argmin}}  f_{\gamma}(x)$.
    
    Если имеем  $\frac{\varepsilon}{2}$-точность для функции $f_{\gamma}(x)$ с $\gamma = \frac{\varepsilon}{2 M_2}$ (из следствия \ref{gamma_parameter_clear}), то имеем $\varepsilon$-точность для функции $f(x)$:
    \begin{equation*}
        f(x_{ag}^{N+1}) - f(x_{*}) \leq f(x_{ag}^{N+1}) - f(x_{*}(\gamma)) \leq f_{\gamma}(x_{ag}^{N+1}) - f(x_{*}(\gamma)) + \gamma M_2  \leq \frac{\varepsilon}{2 } + \frac{\varepsilon}{2} =\varepsilon.
    \end{equation*}
    Чтобы была $\frac{\varepsilon}{2}$-точность для $f_{\gamma}(x) $ нужно
	\begin{equation}
	    \label{eq1_for_proof1:1_l2_randomization}
	    \frac{\sqrt{d} \Delta R}{\gamma} \leq \frac{\varepsilon}{6},
	\end{equation}
	\begin{equation}
	    \label{eq2_for_proof1:1_l2_randomization}
	    \frac{4L_{f_\gamma}R^2}{ N^2} \leq \frac{\varepsilon}{6},
	\end{equation}
	
	\begin{equation}
	    \label{eq3_for_proof1:1_l2_randomization}
	    \frac{ 4\sigma R}{\sqrt{B K N}} \leq \frac{\varepsilon}{6}.
	\end{equation}

        Подставляя $L_{f_\gamma} =  \frac{\sqrt{d} M}{\gamma}$ (из следствия \ref{Const_Lipshitz_grad}), где $\gamma = \frac{\varepsilon}{2 M_2}$  и $ \sigma^2 = 2 \kappa(p,d) M_2^2$ (из следствия \ref{sigma_for_two_point}), в неравенства (\ref{eq1_for_proof1:1_l2_randomization})-(\ref{eq3_for_proof1:1_l2_randomization}), получим
        \begin{equation*}
            \Delta \leq \frac{\gamma \varepsilon }{6 \sqrt{d} R} \,\,\, \Rightarrow \,\,\, \Delta \leq \frac{ \varepsilon^2 }{12 \sqrt{d} M_2 R} \,\,\, \Rightarrow
        \end{equation*}
        \begin{equation*}
            \Delta = O\left( \frac{ \varepsilon^2 }{\sqrt{d} M_2 R} \right)
        \end{equation*}
        уровень неточности,
        \begin{equation*}
            N^2 \geq \frac{48 \sqrt{d} M M_2 R^2}{\varepsilon^2} \Rightarrow N \geq \frac{4 \sqrt{3} d^{1/4} \sqrt{M M_2} R}{\varepsilon} \Rightarrow
        \end{equation*}
        \begin{equation*}
            \Rightarrow N = O\left( \frac{d^{1/4} \sqrt{M M_2} R}{\varepsilon} \right)
        \end{equation*}
        количество коммуникационных раундов,
        \begin{equation*}
            B \geq \frac{576 \sigma^2 R^2}{KN \varepsilon^2} \Rightarrow B \geq \frac{1152 \kappa(p,d) M_2^2 R^2}{KN \varepsilon^2} \Rightarrow
        \end{equation*}
        \begin{equation*}
            \Rightarrow B = O\left( \frac{\kappa(p,d) M_2^2 R^2}{KN \varepsilon^2} \right)
        \end{equation*}
        количество работающих параллельно машин и
        \begin{equation*}
             T = N \cdot K \cdot B = \frac{1152 \kappa(p,d) M_2^2 R^2}{ \varepsilon^2} \,\,\, \Rightarrow
        \end{equation*}
        \begin{equation*}
            \Rightarrow \,\,\, T = \tilde O \left(  \frac{\kappa(p,d) M_2^2 R^2}{ \varepsilon^2}   \right) =
            \begin{cases}
                \tilde O \left(  \frac{d M_2^2 R^2}{ \varepsilon^2}   \right), & p = 2 \;\;\; (q = 2),\\
                \tilde O \left(  \frac{(\ln d) M_2^2 R^2}{ \varepsilon^2}   \right), & p = 1 \;\;\; (q = \infty),
            \end{cases}
        \end{equation*}
        общее количество вызовов двухточечного безградиентного оракула;

    \item Single-Machine Accelerated SGD
    
    Данный алгоритм после $NK$ итераций дает скорость сходимости для $f_\gamma (x)$ (см. \cite{Woodworth_2021, Lan_2012}) в соответствии со следствием \ref{corollary_lemma_3_l2_randomization}:
    \begin{equation*}
        \mathds{E}[f_{\gamma}(x_{ag}^{NK+1}) - f(x_{*})] \leq \frac{4L_{f_\gamma}R^2}{N^2K^2} + \frac{4 \sigma R}{\sqrt{NK}} + \frac{\sqrt{d} \Delta R}{\gamma},
    \end{equation*}
    где $x_{*}(\gamma) = \underset{x \in Q_\gamma}{\mathrm{argmin}}  f_{\gamma}(x)$.
    
    Если имеем  $\frac{\varepsilon}{2}$-точность для функции $f_{\gamma}(x)$ с $\gamma = \frac{\varepsilon}{2 M_2}$ (из следствия \ref{gamma_parameter_clear}), то имеем $\varepsilon$-точность для функции $f(x)$:
    \begin{equation*}
       f(x_{ag}^{N+1}) - f(x_{*}) \leq f(x_{ag}^{N+1}) - f(x_{*}(\gamma)) \leq f_{\gamma}(x_{ag}^{N+1}) - f(x_{*}(\gamma)) + \gamma M_2  \leq \frac{\varepsilon}{2 } + \frac{\varepsilon}{2} =\varepsilon.
    \end{equation*}
    Чтобы была $\frac{\varepsilon}{2}$-точность для $f_{\gamma}(x) $ нужно
	\begin{equation}
	    \label{eq1_for_proof1:2_l2_randomization}
	    \frac{\sqrt{d} \Delta R}{\gamma} \leq \frac{\varepsilon}{6},
	\end{equation}
	\begin{equation}
	    \label{eq2_for_proof1:2_l2_randomization}
	    \frac{4L_{f_\gamma}R^2}{K^2 N^2} \leq \frac{\varepsilon}{6},
	\end{equation}
	
	\begin{equation}
	    \label{eq3_for_proof1:2_l2_randomization}
	    \frac{ 4\sigma R}{\sqrt{K N}} \leq \frac{\varepsilon}{6}.
	\end{equation}

        Подставляя $L_{f_\gamma} =  \frac{\sqrt{d} M}{\gamma}$ (из следствия \ref{Const_Lipshitz_grad}), где $\gamma = \frac{\varepsilon}{2 M_2}$  и $ \sigma^2 = 2 \kappa(p,d) M_2^2$ (из следствия \ref{sigma_for_two_point}), в неравенства (\ref{eq1_for_proof1:2_l2_randomization})-(\ref{eq3_for_proof1:2_l2_randomization}), получим
        \begin{equation*}
            \Delta \leq \frac{\gamma \varepsilon }{6 \sqrt{d} R} \,\,\, \Rightarrow \,\,\, \Delta \leq \frac{ \varepsilon^2 }{12 \sqrt{d} M_2 R} \,\,\, \Rightarrow
        \end{equation*}
        \begin{equation*}
            \Delta = O\left( \frac{ \varepsilon^2 }{\sqrt{d} M_2 R} \right)
        \end{equation*}
        уровень неточности,
        \begin{equation*}
            NK  \geq \frac{576 \sigma^2 R^2}{\varepsilon^2} \Rightarrow N K  \geq \frac{1152 \kappa(p,d) M_2^2 R^2}{ \varepsilon^2}.
        \end{equation*}
        Так как $N$ напрямую зависит от $K$, то количество коммуникачионных раундов можно взять $N = 1$, тогда
        \begin{equation*}
            NK = K = O \left( \frac{ \kappa(p,d) M_2^2 R^2}{ \varepsilon^2} \right)
        \end{equation*}
        количество локальных вызовов безградиентного оракула и
        \begin{equation*}
             T = N \cdot K \cdot B = \frac{1152 \kappa(p,d) M_2^2 R^2}{ \varepsilon^2} \,\,\, \Rightarrow
        \end{equation*}
        \begin{equation*}
            \Rightarrow \,\,\, T = \tilde O \left(  \frac{\kappa(p,d) M_2^2 R^2}{ \varepsilon^2}   \right) =
            \begin{cases}
                \tilde O \left(  \frac{d M_2^2 R^2}{ \varepsilon^2}   \right), & p = 2 \;\;\; (q = 2),\\
                \tilde O \left(  \frac{(\ln d) M_2^2 R^2}{ \varepsilon^2}   \right), & p = 1 \;\;\; (q = \infty),
            \end{cases}
        \end{equation*}
        общее количество вызовов двухточечного безградиентного оракула;
    
    \item Local-AC-SA
    
    Данный алгоритм после $N$ раундов связи дает скорость сходимости для $f_\gamma (x)$ (см. \cite{Woodworth_2020, Lan_2012}) в соответствии со следствием \ref{corollary_lemma_3_l2_randomization}:
    \begin{equation*}
        \mathds{E}[f_{\gamma}(x_{ag}^{N+1}) - f(x_{*})] \leq \frac{4L_{f_\gamma} R^2}{K^2 N^2} + \frac{4 \sigma R}{\sqrt{B K N}} + \frac{\sqrt{d} \Delta R}{\gamma},
    \end{equation*}
    где $x_{*}(\gamma) = \underset{x \in Q_\gamma}{\mathrm{argmin}}  f_{\gamma}(x)$.
    
    Если имеем  $\frac{\varepsilon}{2}$-точность для функции $f_{\gamma}(x)$ с $\gamma = \frac{\varepsilon}{2 M_2}$ (из следствия \ref{gamma_parameter_clear}), то имеем $\varepsilon$-точность для функции $f(x)$:
    \begin{equation*}
       f(x_{ag}^{N+1}) - f(x_{*}) \leq f(x_{ag}^{N+1}) - f(x_{*}(\gamma)) \leq f_{\gamma}(x_{ag}^{N+1}) - f(x_{*}(\gamma)) + \gamma M_2  \leq \frac{\varepsilon}{2 } + \frac{\varepsilon}{2} =\varepsilon.
    \end{equation*}
    Чтобы была $\frac{\varepsilon}{2}$-точность для $f_{\gamma}(x) $ нужно
	\begin{equation}
	    \label{eq1_for_proof1:3_l2_randomization}
	    \frac{\sqrt{d} \Delta R}{\gamma} \leq \frac{\varepsilon}{6},
	\end{equation}
	\begin{equation}
	    \label{eq2_for_proof1:3_l2_randomization}
	    \frac{4L_{f_\gamma}R^2}{K^2 N^2} \leq \frac{\varepsilon}{6},
	\end{equation}
	
	\begin{equation}
	    \label{eq3_for_proof1:3_l2_randomization}
	    \frac{ 4\sigma R}{\sqrt{B K N}} \leq \frac{\varepsilon}{6}.
	\end{equation}

        Подставляя $L_{f_\gamma} =  \frac{\sqrt{d} M}{\gamma}$ (из следствия \ref{Const_Lipshitz_grad}), где $\gamma = \frac{\varepsilon}{2 M_2}$  и $ \sigma^2 = 2 \kappa(p,d) M_2^2$ (из следствия \ref{sigma_for_two_point}), в неравенства (\ref{eq1_for_proof1:3_l2_randomization})-(\ref{eq3_for_proof1:3_l2_randomization}), получим
        \begin{equation*}
            \Delta \leq \frac{\gamma \varepsilon }{6 \sqrt{d} R} \,\,\, \Rightarrow \,\,\, \Delta \leq \frac{ \varepsilon^2 }{12 \sqrt{d} M_2 R} \,\,\, \Rightarrow
        \end{equation*}
        \begin{equation*}
            \Delta = O\left( \frac{ \varepsilon^2 }{\sqrt{d} M_2 R} \right)
        \end{equation*}
        уровень неточности,
        \begin{equation*}
            N^2 K^2 \geq \frac{48 \sqrt{d} M M_2 R^2}{\varepsilon^2} \Rightarrow NK \geq \frac{4 \sqrt{3} d^{1/4} \sqrt{M M_2} R}{\varepsilon}.
        \end{equation*}
        Так как $N$ напрямую зависит от $K$, то количество коммуникачионных раундов можно взять $N = 1$, тогда
        \begin{equation*}
            NK = K = O \left( \frac{d^{1/4} \sqrt{M M_2} R}{\varepsilon} \right)
        \end{equation*}
        количество локальных вызовов безградиентного оракула,
        \begin{equation*}
            B \geq \frac{576 \sigma^2 R^2}{KN \varepsilon^2} \Rightarrow B \geq \frac{1152 \kappa(p,d) M_2^2 R^2}{KN \varepsilon^2} \Rightarrow
        \end{equation*}
        \begin{equation*}
            \Rightarrow B = O\left( \frac{\kappa(p,d) M_2^2 R^2}{KN \varepsilon^2} \right)
        \end{equation*}
        количество работающих параллельно машин и
        \begin{equation*}
             T = N \cdot K \cdot B = \frac{1152 \kappa(p,d) M_2^2 R^2}{ \varepsilon^2} \,\,\, \Rightarrow
        \end{equation*}
        \begin{equation*}
            \Rightarrow \,\,\, T = \tilde O \left(  \frac{\kappa(p,d) M_2^2 R^2}{ \varepsilon^2}   \right) =
            \begin{cases}
                \tilde O \left(  \frac{d M_2^2 R^2}{ \varepsilon^2}   \right), & p = 2 \;\;\; (q = 2),\\
                \tilde O \left(  \frac{(\ln d) M_2^2 R^2}{ \varepsilon^2}   \right), & p = 1 \;\;\; (q = \infty),
            \end{cases}
        \end{equation*}
        общее количество вызовов двухточечного безградиентного оракула;

    \item Federated Accelerated SGD (FedAc)
    
    Данный алгоритм после $N$ раундов связи дает скорость сходимости для $f_\gamma (x)$ (см. \cite{Yuan_Ma_2020}) в соответствии со следствием \ref{corollary_lemma_3_l2_randomization}:
    \begin{equation*}
        \mathds{E}[f_{\gamma}(x_{ag}^{N+1}) - f(x_{*})] \leq \frac{L_{f_\gamma} R^2}{K N^2} + \frac{ \sigma R}{\sqrt{B K N}} + \min \left\{ \frac{L_{f_\gamma}^{1/3} \sigma^{2/3} R^{4/3}}{K^{1/3} N}, \frac{L_{f_\gamma}^{1/2} \sigma^{1/2} R^{3/2}}{K^{1/4} N} \right\} + \frac{\sqrt{d} \Delta R}{\gamma},
    \end{equation*}
    где $x_{*}(\gamma) = \underset{x \in Q_\gamma}{\mathrm{argmin}}  f_{\gamma}(x)$.
    
    Если имеем  $\frac{\varepsilon}{2}$-точность для функции $f_{\gamma}(x)$ с $\gamma = \frac{\varepsilon}{2 M_2}$ (из следствия \ref{gamma_parameter_clear}), то имеем $\varepsilon$-точность для функции $f(x)$:
    \begin{equation*}
        f(x_{ag}^{N+1}) - f(x_{*}) \leq f(x_{ag}^{N+1}) - f(x_{*}(\gamma)) \leq f_{\gamma}(x_{ag}^{N+1}) - f(x_{*}(\gamma)) + \gamma M_2  \leq \frac{\varepsilon}{2 } + \frac{\varepsilon}{2} =\varepsilon.
    \end{equation*}
    Чтобы была $\frac{\varepsilon}{2}$-точность для $f_{\gamma}(x) $ нужно
	\begin{equation}
	    \label{eq1_for_proof1:4_l2_randomization}
	    \frac{\sqrt{d} \Delta R}{\gamma} \leq \frac{\varepsilon}{8},
	\end{equation}
	\begin{equation}
	    \label{eq2_for_proof1:4_l2_randomization}
	    \frac{L_{f_\gamma}R^2}{K N^2} \leq \frac{\varepsilon}{8},
	\end{equation}
	
	\begin{equation}
	    \label{eq3_for_proof1:4_l2_randomization}
	    \frac{ \sigma R}{\sqrt{B K N}} \leq \frac{\varepsilon}{8},
	\end{equation}
	\begin{equation}
	    \label{eq4_for_proof1:4_l2_randomization}
	    \min \left\{ \frac{L_{f_\gamma}^{1/3} \sigma^{2/3} R^{4/3}}{K^{1/3} N}, \frac{L_{f_\gamma}^{1/2} \sigma^{1/2} R^{3/2}}{K^{1/4} N} \right\} \leq \frac{\varepsilon}{8}.
	\end{equation}

        Подставляя $L_{f_\gamma} =  \frac{\sqrt{d} M}{\gamma}$ (из следствия \ref{Const_Lipshitz_grad}), где $\gamma = \frac{\varepsilon}{2 M_2}$  и $ \sigma^2 = 2 \kappa(p,d) d M_2^2$ (из следствия \ref{sigma_for_two_point}), в неравенства (\ref{eq1_for_proof1:4_l2_randomization})-(\ref{eq4_for_proof1:4_l2_randomization}), получим
        \begin{equation*}
            \Delta \leq \frac{\gamma \varepsilon }{8 \sqrt{d} R} \,\,\, \Rightarrow \,\,\, \Delta \leq \frac{ \varepsilon^2 }{16 \sqrt{d} M_2 R} \,\,\, \Rightarrow
        \end{equation*}
        \begin{equation*}
            \Delta = O\left( \frac{ \varepsilon^2 }{\sqrt{d} M_2 R} \right)
        \end{equation*}
        уровень неточности,
        \begin{equation*}
            NK \geq \frac{8^{1/2} K^{1/2} L^{1/2}_{f_\gamma}R}{ \varepsilon^{1/2}} \,\,\,  \Rightarrow \,\,\,  NK \geq \frac{16^{1/2} K^{1/2} d^{1/4} M^{1/2} M_2^{1/2} R}{ \varepsilon} \,\,\,  \Rightarrow \,\,\, NK \gtrsim \frac{ K^{1/2} }{ \varepsilon},
        \end{equation*}
        \begin{equation*}
            NK \geq \frac{64 \sigma^2 R^2}{B \varepsilon^2}  \,\,\, \Rightarrow \,\,\, NK \geq \frac{72 \kappa(p,d) d M_2^2 R^2}{B \varepsilon^2} \,\,\,  \Rightarrow \,\,\, NK \gtrsim \frac{1}{B \varepsilon^2}
        \end{equation*}
        и
        \begin{equation*}
            NK \geq \min \left\{ \frac{K^{2/3} L_{f_\gamma}^{1/3} \sigma^{2/3} R^{4/3}}{\varepsilon}, \frac{K^{3/4} L_{f_\gamma}^{1/2} \sigma^{1/2} R^{3/2}}{ \varepsilon} \right\} \,\,\, \Rightarrow 
        \end{equation*}
        \begin{equation*}
            \Rightarrow \,\,\, NK \geq \min \left\{  \frac{2^{1/3} K^{2/3} d^{1/6} (M M_2)^{1/3} \sigma^{2/3} R^{4/3}}{ \varepsilon^{4/3}}, \frac{2^{1/2} K^{3/4} d^{1/4} (M M_2)^{1/2} \sigma^{1/2} R^{3/2}}{ \varepsilon^{3/2}} \right\} \,\,\, \Rightarrow 
        \end{equation*}
        \begin{equation*}
            \Rightarrow \,\,\, NK \geq \frac{2^{1/3} K^{2/3} d^{1/6} (M M_2)^{1/3} \sigma^{2/3} R^{4/3}}{ \varepsilon^{4/3}} \,\,\,  \Rightarrow \,\,\, NK \gtrsim \frac{K^{2/3}}{ \varepsilon^{4/3}}.
        \end{equation*}
       Таким образом, наименьшее число коммуникационных раундов при условиии, что $NK \in [1, \varepsilon^{-2}]$ и $T \in [1, \varepsilon^{-2}]$, будет выглядить следующим образом:
        \begin{equation*}
            N \sim \frac{1}{\varepsilon}: \;\;\;\;\; N \geq \frac{8 L_{f_\gamma}^{1/3} \sigma^{2/3} R^{4/3}}{K^{1/3} \varepsilon} \Rightarrow \;\;\; N = O\left( \frac{d^{1/2} (\kappa(p,d) M)^{1/3} M_2 R^{4/3}}{K^{1/3} \varepsilon^{4/3}} \right)
        \end{equation*}
        тогда
        \begin{equation*}
            K \sim \frac{1}{\varepsilon}: \;\;\;\;\; K \geq \frac{ \sigma^2 R^2}{B N \varepsilon^2} \Rightarrow \;\;\; K = O\left( \frac{ \kappa(p,d) d M^2_2 R^2}{B N \varepsilon^2} \right)
        \end{equation*}
        количество локальных вызовов безградиентного оракула, $B = 1$~--- количество работающих параллельно машин и
        \begin{equation*}
             T \sim \frac{1}{\varepsilon^2}: \;\;\;\;\; T = N \cdot K \cdot B = \tilde O \left(  \frac{\kappa(p,d) d M_2^2 R^2}{ \varepsilon^2}   \right) =
            \begin{cases}
                \tilde O \left(  \frac{d M_2^2 R^2}{ \varepsilon^2}   \right), & p = 2 \;\;\; (q = 2),\\
                \tilde O \left(  \frac{(\ln d) M_2^2 R^2}{ \varepsilon^2}   \right), & p = 1 \;\;\; (q = \infty),
            \end{cases}
        \end{equation*}
        общее количество вызовов двухточечного безградиентного оракула.

\end{itemize}


\section{Доказательство теоремы \ref{theorem_2}}
\label{proof_theorem_2}
В данном подразделе приведем полное доказательство теоремы \ref{theorem_2}. Для этого разобьем доказательство на две части: доказательство для $l_1$-рандомизации и доказательство для $l_2$-рандомизации.

\underline{Для $l_1$-рандомизации имеем}:
\begin{itemize}
    \item Minibatch SMP
    
    Данный алгоритм после $N$ раундов связи дает скорость сходимости для $f_\gamma (x)$ (см. следствие \ref{Optima_alg_VI}) в соответствии с замечанием \ref{remark_1}:
    \begin{equation*}
        \mathds{E}[f_{\gamma}(z_{N}) ] \leq  \max \left\{ \frac{7}{4} \frac{L R^2}{N}, 7 \frac{ \sigma R}{\sqrt{BKN}}  \right\} + \frac{d \Delta R}{\gamma}.
    \end{equation*}
    
    Если имеем  $\frac{\varepsilon}{2}$-точность для функции $f_{\gamma}(z)$ с $\gamma = \frac{\sqrt{d} \varepsilon}{4 M_2}$ (из следствия \ref{gamma_parameter_clear}), то имеем $\varepsilon$-точность для функции $f(z)$:
    \begin{equation*}
         f(z_{N}) \leq f_{\gamma}(z_{N})  + \frac{2}{\sqrt{d}}\gamma M_2  \leq \frac{\varepsilon}{2 } + \frac{\varepsilon}{2} =\varepsilon.
    \end{equation*}
    Чтобы была $\frac{\varepsilon}{2}$-точность для $f_{\gamma}(z) $ нужно
	\begin{equation}
	    \label{eq1_for_proof2:1_l1_randomization}
	    \frac{d \Delta R}{\gamma} \leq \frac{\varepsilon}{4},
	\end{equation}
	\begin{equation}
	    \label{eq2_for_proof2:1_l1_randomization}
	    \max \left\{ \frac{7}{4} \frac{L R^2}{N}, 7 \frac{ \sigma R}{\sqrt{BKN}}  \right\} \leq \frac{\varepsilon}{4}.
	\end{equation}

       Подставляя $L_{f_\gamma} =  \frac{d M}{\gamma}$ (из следствия \ref{Const_Lipshitz_grad}), где $\gamma = \frac{\sqrt{d} \varepsilon}{4 M_2}$  и $ \sigma^2 = 2 \kappa(p,d) M_2^2$ (из следствия \ref{sigma_for_two_point}), в неравенства (\ref{eq1_for_proof2:1_l1_randomization}) и (\ref{eq2_for_proof2:1_l1_randomization}), получим
        \begin{equation*}
            \Delta \leq \frac{\gamma \varepsilon }{6 d R} \,\,\, \Rightarrow \,\,\, \Delta \leq \frac{ \varepsilon^2 }{24 \sqrt{d} M_2 R} \,\,\, \Rightarrow
        \end{equation*}
        \begin{equation*}
            \Delta = O\left( \frac{ \varepsilon^2 }{\sqrt{d} M_2 R} \right)
        \end{equation*}
        уровень неточности,
        \begin{equation*}
            N \geq \frac{784 \sigma^2 R^2}{BK \varepsilon^2} \Rightarrow N \geq \frac{1568 \kappa(p,d) M_2^2 R^2}{BK \varepsilon^2}.
        \end{equation*}
        Так как $K = 1$, а $N$ напрямую зависит от $B$, то количество коммуникационных раундов можно взять $N = 1$, тогда
        \begin{equation*}
            \Rightarrow B = O\left( \frac{\kappa(p,d) M_2^2 R^2}{KN \varepsilon^2} \right)
        \end{equation*}
        количество работающих параллельно машин и
        \begin{equation*}
             T = N \cdot K \cdot B = \frac{1568 \kappa(p,d) M_2^2 R^2}{ \varepsilon^2} \,\,\, \Rightarrow
        \end{equation*}
        \begin{equation*}
            \Rightarrow \,\,\, T =  O \left(  \frac{\kappa(p,d) M_2^2 R^2}{ \varepsilon^2}   \right) =
            \begin{cases}
                 O \left(  \frac{d M_2^2 R^2}{ \varepsilon^2}   \right), & p = 2 \;\;\; (q = 2),\\
                 O \left(  \frac{ M_2^2 R^2}{ \varepsilon^2}   \right), & p = 1 \;\;\; (q = \infty),
            \end{cases}
        \end{equation*}
        общее количество вызовов двухточечного безградиентного оракула;

    \item Single-Machine SMP
    
    Данный алгоритм после $NK$ итераций дает скорость сходимости для $f_\gamma (x)$ (см. следствие \ref{Optima_alg_VI}) в соответствии с замечанием \ref{remark_1}:
    \begin{equation*}
        \mathds{E}[f_{\gamma}(z_{NK}) ] \leq  \max \left\{  \frac{7}{4} \frac{L R^2}{KN}, 7 \frac{\sigma R}{\sqrt{KN}} \right\} + \frac{\sqrt{d} \Delta R}{\gamma}.
    \end{equation*}
    
    Если имеем  $\frac{\varepsilon}{2}$-точность для функции $f_{\gamma}(z)$ с $\gamma = \frac{\sqrt{d} \varepsilon}{4 M_2}$ (из следствия \ref{gamma_parameter_clear}), то имеем $\varepsilon$-точность для функции $f(z)$:
    \begin{equation*}
         f(z_{N}) \leq f_{\gamma}(z_{N})  + \frac{2}{\sqrt{d}}\gamma M_2  \leq \frac{\varepsilon}{2 } + \frac{\varepsilon}{2} =\varepsilon.
    \end{equation*}
    Чтобы была $\frac{\varepsilon}{2}$-точность для $f_{\gamma}(z) $ нужно
	\begin{equation}
	    \label{eq1_for_proof2:2_l1_randomization}
	    \frac{\sqrt{d} \Delta R}{\gamma} \leq \frac{\varepsilon}{4},
	\end{equation}
	\begin{equation}
	    \label{eq2_for_proof2:2_l1_randomization}
	   \max \left\{  \frac{7}{4} \frac{L R^2}{KN}, 7 \frac{\sigma R}{\sqrt{KN}} \right\} \leq \frac{\varepsilon}{4}.
	\end{equation}

        Подставляя $L_{f_\gamma} =  \frac{d M}{\gamma}$ (из следствия \ref{Const_Lipshitz_grad}), где $\gamma = \frac{\sqrt{d} \varepsilon}{4 M_2}$  и $ \sigma^2 = 2 \kappa(p,d) M_2^2$ (из \ref{sigma_for_two_point}), в неравенства (\ref{eq1_for_proof2:2_l1_randomization}) и (\ref{eq2_for_proof2:2_l1_randomization}), получим
        \begin{equation*}
            \Delta \leq \frac{\gamma \varepsilon }{6 d R} \,\,\, \Rightarrow \,\,\, \Delta \leq \frac{ \varepsilon^2 }{24 \sqrt{d} M_2 R} \,\,\, \Rightarrow
        \end{equation*}
        \begin{equation*}
            \Delta = O\left( \frac{ \varepsilon^2 }{\sqrt{d} M_2 R} \right)
        \end{equation*}
        уровень неточности,
        \begin{equation*}
            N \geq \frac{784 \sigma^2 R^2}{K \varepsilon^2} \Rightarrow N \geq \frac{1568 \kappa(p,d) M_2^2 R^2}{K \varepsilon^2}.
        \end{equation*}
        Так как $N$ напрямую зависит от $K$, то количество коммуникационных раундов можно взять $N = 1$, тогда
        \begin{equation*}
            \Rightarrow K = O\left( \frac{\kappa(p,d) M_2^2 R^2}{KN \varepsilon^2} \right)
        \end{equation*}
        количество локальных вызовов безградиентного оракула и
        \begin{equation*}
             T = N \cdot K \cdot B = \frac{1568 \kappa(p,d) M_2^2 R^2}{ \varepsilon^2} \,\,\, \Rightarrow
        \end{equation*}
        \begin{equation*}
            \Rightarrow \,\,\, T =  O \left(  \frac{\kappa(p,d) M_2^2 R^2}{ \varepsilon^2}   \right) =
            \begin{cases}
                 O \left(  \frac{d M_2^2 R^2}{ \varepsilon^2}   \right), & p = 2 \;\;\; (q = 2),\\
                 O \left(  \frac{ M_2^2 R^2}{ \varepsilon^2}   \right), & p = 1 \;\;\; (q = \infty),
            \end{cases}
        \end{equation*}
        общее количество вызовов двухточечного безградиентного оракула.
\end{itemize}

\underline{Для $l_2$-рандомизации имеем}:
\begin{itemize}
    \item Minibatch SMP
    
    Данный алгоритм после $N$ раундов связи дает скорость сходимости для $f_\gamma (x)$ (см. следствие \ref{Optima_alg_VI}) в соответствии с замечанием \ref{remark_1}:
    \begin{equation*}
        \mathds{E}[f_{\gamma}(z_{N}) ] \leq  \max \left\{ \frac{7}{4} \frac{L R^2}{N}, 7 \frac{ \sigma R}{\sqrt{BKN}}  \right\} + \frac{\sqrt{d} \Delta R}{\gamma}.
    \end{equation*}
    
    Если имеем  $\frac{\varepsilon}{2}$-точность для функции $f_{\gamma}(z)$ с $\gamma = \frac{\varepsilon}{2 M_2}$ (из следствия \ref{gamma_parameter_clear}), то имеем $\varepsilon$-точность для функции $f(z)$:
    \begin{equation*}
         f(z_{N}) \leq f_{\gamma}(z_{N})  + \gamma M_2  \leq \frac{\varepsilon}{2 } + \frac{\varepsilon}{2} =\varepsilon.
    \end{equation*}
    Чтобы была $\frac{\varepsilon}{2}$-точность для $f_{\gamma}(z) $ нужно
	\begin{equation}
	    \label{eq1_for_proof2:1_l2_randomization}
	    \frac{\sqrt{d} \Delta R}{\gamma} \leq \frac{\varepsilon}{4},
	\end{equation}
	\begin{equation}
	    \label{eq2_for_proof2:1_l2_randomization}
	    \max \left\{ \frac{7}{4} \frac{L R^2}{N}, 7 \frac{ \sigma R}{\sqrt{BKN}}  \right\} \leq \frac{\varepsilon}{4}.
	\end{equation}

       Подставляя $L_{f_\gamma} =  \frac{\sqrt{d} M}{\gamma}$ (из следствия \ref{Const_Lipshitz_grad}), где $\gamma = \frac{\varepsilon}{2 M_2}$  и $ \sigma^2 = 2 \kappa(p,d) M_2^2$ (из \ref{sigma_for_two_point}), в неравенства (\ref{eq1_for_proof2:1_l2_randomization}) и (\ref{eq2_for_proof2:1_l2_randomization}), получим
        \begin{equation*}
            \Delta \leq \frac{\gamma \varepsilon }{4 \sqrt{d} R} \,\,\, \Rightarrow \,\,\, \Delta \leq \frac{ \varepsilon^2 }{8 \sqrt{d} M_2 R} \,\,\, \Rightarrow
        \end{equation*}
        \begin{equation*}
            \Delta = O\left( \frac{ \varepsilon^2 }{\sqrt{d} M_2 R} \right)
        \end{equation*}
        уровень неточности,
        \begin{equation*}
            N \geq \frac{784 \sigma^2 R^2}{BK \varepsilon^2} \Rightarrow N \geq \frac{1568 \kappa(p,d) M_2^2 R^2}{BK \varepsilon^2}.
        \end{equation*}
        Так как $K = 1$, а $N$ напрямую зависит от $B$, то количество коммуникационных раундов можно взять $N = 1$, тогда
        \begin{equation*}
            \Rightarrow B = O\left( \frac{\kappa(p,d) M_2^2 R^2}{KN \varepsilon^2} \right)
        \end{equation*}
        количество работающих параллельно машин и
        \begin{equation*}
             T = N \cdot K \cdot B = \frac{1568 \kappa(p,d) M_2^2 R^2}{ \varepsilon^2} \,\,\, \Rightarrow
        \end{equation*}
        \begin{equation*}
            \Rightarrow \,\,\, T = \tilde O \left(  \frac{\kappa(p,d) M_2^2 R^2}{ \varepsilon^2}   \right) =
            \begin{cases}
                \tilde O \left(  \frac{d M_2^2 R^2}{ \varepsilon^2}   \right), & p = 2 \;\;\; (q = 2),\\
                \tilde O \left(  \frac{(\ln d) M_2^2 R^2}{ \varepsilon^2}   \right), & p = 1 \;\;\; (q = \infty),
            \end{cases}
        \end{equation*}
        общее количество вызовов двухточечного безградиентного оракула;

    \item Single-Machine SMP
    
    Данный алгоритм после $NK$ итераций дает скорость сходимости для $f_\gamma (x)$ (см. следствие \ref{Optima_alg_VI}) в соответствии с замечанием \ref{remark_1}:
    \begin{equation*}
        \mathds{E}[f_{\gamma}(z_{NK}) ] \leq  \max \left\{  \frac{7}{4} \frac{L R^2}{KN}, 7 \frac{\sigma R}{\sqrt{KN}} \right\} + \frac{\sqrt{d} \Delta R}{\gamma}.
    \end{equation*}
    
    Если имеем  $\frac{\varepsilon}{2}$-точность для функции $f_{\gamma}(z)$ с $\gamma = \frac{\varepsilon}{2 M_2}$ (из следствия \ref{gamma_parameter_clear}), то имеем $\varepsilon$-точность для функции $f(z)$:
    \begin{equation*}
         f(z_{N}) \leq f_{\gamma}(z_{N})  + \gamma M_2  \leq \frac{\varepsilon}{2 } + \frac{\varepsilon}{2} =\varepsilon.
    \end{equation*}
    Чтобы была $\frac{\varepsilon}{2}$-точность для $f_{\gamma}(z) $ нужно
	\begin{equation}
	    \label{eq1_for_proof2:2_l2_randomization}
	    \frac{\sqrt{d} \Delta R}{\gamma} \leq \frac{\varepsilon}{4},
	\end{equation}
	\begin{equation}
	    \label{eq2_for_proof2:2_l2_randomization}
	   \max \left\{  \frac{7}{4} \frac{L R^2}{KN}, 7 \frac{\sigma R}{\sqrt{KN}} \right\} \leq \frac{\varepsilon}{4}.
	\end{equation}

        Подставляя $L_{f_\gamma} =  \frac{\sqrt{d} M}{\gamma}$ (из следствия \ref{Const_Lipshitz_grad}), где $\gamma = \frac{\varepsilon}{2 M_2}$  и $ \sigma^2 = 2 \kappa(p,d) M_2^2$ (из \ref{sigma_for_two_point}), в неравенства (\ref{eq1_for_proof2:2_l2_randomization}) и (\ref{eq2_for_proof2:2_l2_randomization}), получим
        \begin{equation*}
            \Delta \leq \frac{\gamma \varepsilon }{4 \sqrt{d} R} \,\,\, \Rightarrow \,\,\, \Delta \leq \frac{ \varepsilon^2 }{8 \sqrt{d} M_2 R} \,\,\, \Rightarrow
        \end{equation*}
        \begin{equation*}
            \Delta = O\left( \frac{ \varepsilon^2 }{\sqrt{d} M_2 R} \right)
        \end{equation*}
        уровень неточности,
        \begin{equation*}
            N \geq \frac{784 \sigma^2 R^2}{K \varepsilon^2} \Rightarrow N \geq \frac{1568 \kappa(p,d) M_2^2 R^2}{K \varepsilon^2}.
        \end{equation*}
        Так как $N$ напрямую зависит от $K$, то количество коммуникационных раундов можно взять $N = 1$, тогда
        \begin{equation*}
            \Rightarrow K = O\left( \frac{\kappa(p,d) M_2^2 R^2}{KN \varepsilon^2} \right)
        \end{equation*}
        количество локальных вызовов безградиентного оракула и
        \begin{equation*}
             T = N \cdot K \cdot B = \frac{1568 \kappa(p,d) M_2^2 R^2}{ \varepsilon^2} \,\,\, \Rightarrow
        \end{equation*}
        \begin{equation*}
            \Rightarrow \,\,\, T = \tilde O \left(  \frac{\kappa(p,d) M_2^2 R^2}{ \varepsilon^2}   \right) =
            \begin{cases}
                \tilde O \left(  \frac{d M_2^2 R^2}{ \varepsilon^2}   \right), & p = 2 \;\;\; (q = 2),\\
                \tilde O \left(  \frac{(\ln d) M_2^2 R^2}{ \varepsilon^2}   \right), & p = 1 \;\;\; (q = \infty),
            \end{cases}
        \end{equation*}
        общее количество вызовов двухточечного безградиентного оракула.
        
\end{itemize}


\section{Одноточечные безградиентные алгоритмы}
\label{case_one_point}
Проделаем такую же процедуру, как и в п. \ref{subsection:Gradient_Free_Algorithms} для создания безградиентных одноточечных алгоритмов.

Для начала на примере леммы \ref{lemma_3_proof_noise_one_point} покажем, что лемма \ref{lemma_3_proof_noise_lp_randomization} также выполняется и для одноточечного оракула.
\begin{lemma}
    \label{lemma_3_proof_noise_one_point}
    Для $ \nabla f_\gamma (x,\xi,e) $ и $\nabla f_\gamma (x)$ с предположением \ref{assumption:SS_2}, выполняется следующее
    \begin{equation*}
         \mathds{E}_{\xi, e} [\langle \nabla f_\gamma (x,\xi,e), r \rangle] \geq \dotprod{ \nabla f_\gamma (x)}{r} - \frac{d \Delta \mathds{E}_{e} [ | \langle \text{sign}(e), r \rangle | ]}{\gamma},
    \end{equation*}
    где $\nabla f_\gamma$ с $l_1$-рандомизацией;
    
    \begin{equation*}
         \mathds{E}_{\xi, e} [\langle \nabla f_\gamma (x,\xi,e), r \rangle] \geq \dotprod{ \nabla f_\gamma (x)}{r} - \frac{d \Delta \mathds{E}_{e} [ | \langle e, r \rangle | ]}{\gamma},
    \end{equation*}
    где $\nabla f_\gamma$ с $l_2$-рандомизацией.
\end{lemma}
\begin{proof}\renewcommand{\qedsymbol}{}
	Рассмотрим
	\begin{itemize}
        \item[i)] для $l_1$-рандомизации (\ref{grad_l1_randomization_one_point}):
        \begin{equation*}
        	\begin{split}
        	    \nabla f_\gamma (x,\xi,e) = & \; \frac{d}{\gamma} f_\delta(x+\gamma e, \xi) \text{ sign}(e) = \frac{d}{\gamma}(f(x+\gamma e, \xi) + \delta(x+\gamma e))\text{ sign}(e) =\\
        	    = & \; \frac{d}{\gamma}(f(x+\gamma e, \xi) \text{ sign}(e) + \delta(x+\gamma e) \text{ sign}(e)).
        	\end{split}
        	\end{equation*}
        	Из данного равенства следует 
        	\begin{equation}
        	\label{eq_1_for_proof_noise_l1_randomization_one_point}
        	\begin{split}
        	    \mathds{E}_{\xi, e} [\langle \nabla f_\gamma (x,\xi,e), r \rangle] = \frac{d}{\gamma} \mathds{E}_{\xi, e} [\langle f(x+\gamma e, \xi) \text{ sign}(e), r \rangle] + \frac{d}{\gamma} \mathds{E}_{ e} [\langle \delta(x+\gamma e) \text{ sign}(e), r \rangle].
        	\end{split}
        	\end{equation}
        	Применяя лемму \ref{lemma_1_proof_noise_l1_randomization} к первому слагаемому (\ref{eq_1_for_proof_noise_l1_randomization_one_point}), получим
        	\begin{equation}
            	\begin{split}
            	\label{eq_2_for_proof_noise_l1_randomization_one_point}
            	    \frac{d}{\gamma} \mathds{E}_{\xi, e} [\langle f(x+\gamma e, \xi) \text{ sign}(e), r \rangle] = & \frac{d}{\gamma} \mathds{E}_{e} \left[ \dotprod{\mathds{E}_{\xi} \left[ f(x+\gamma e, \xi) \right] \text{ sign}(e)}{r} \right] = \\ 
            	    = & \frac{d}{\gamma} \mathds{E}_{e} [\langle f(x+\gamma e)\text{ sign}(e), r \rangle] = \dotprod{\nabla f_\gamma(x)}{r}.  
            	\end{split}
        	\end{equation}
        	Для второго слагаемого (\ref{eq_1_for_proof_noise_l1_randomization_one_point}) с предположением \ref{assumption:SS_2} получим
        	\begin{equation}
        	\label{eq_3_for_proof_noise_l1_randomization_one_point}
            \begin{split}
        	    \frac{d}{\gamma} \mathds{E}_{e} [\langle \delta(x+\gamma e)\text{ sign}(e), r \rangle] \geq  -\frac{d}{\gamma} \Delta \mathds{E}_{e} [ | \langle \text{sign}(e), r \rangle | ].
        	\end{split}
        	\end{equation}
        	Используя уравнения (\ref{eq_2_for_proof_noise_l1_randomization_one_point}) и (\ref{eq_3_for_proof_noise_l1_randomization_one_point}), для уравнения (\ref{eq_1_for_proof_noise_l1_randomization_one_point}), получим утверждение леммы для $l_1$-рандомизации.
    
    \item[ii)] для $l_2$-рандомизации (\ref{grad_l2_randomization_one_point}):
    \begin{equation*}
    	\begin{split}
    	    \nabla f_\gamma (x,\xi,e) = & \; \frac{d}{\gamma} f_\delta(x+\gamma e, \xi)e = \frac{d}{\gamma}(f(x+\gamma e, \xi) + \delta(x+\gamma e))e =\\
    	    = & \; \frac{d}{\gamma} \left( f(x+\gamma e, \xi) e + \delta(x+\gamma e) e \right).
    	\end{split}
    	\end{equation*}
    	Из данного равенства следует 
    	\begin{equation}
        \label{eq_1_for_proof_noise_l2_randomization_one_point}
            \begin{split}
                \mathds{E}_{\xi, e} [\langle \nabla f_\gamma (x,\xi,e), r \rangle] = \frac{d}{\gamma} \mathds{E}_{\xi, e} [\langle f(x+\gamma e, \xi) e, r \rangle] + \frac{d}{\gamma} \mathds{E}_{e} [\langle \delta(x+\gamma e) e, r \rangle].
            	\end{split}
        	\end{equation}
        	Применяя лемму \ref{lemma_1_proof_noise_l2_randomization} к первому слагаемому (\ref{eq_1_for_proof_noise_l2_randomization_one_point}), получим
        	\begin{equation}
            	\begin{split}
            	\label{eq_2_for_proof_noise_l2_randomization_one_point}
            	    \frac{d}{\gamma} \mathds{E}_{\xi, e} [\langle f(x+\gamma e, \xi)e, r \rangle] = & \; \frac{d}{\gamma} \mathds{E}_{e} \left[ \dotprod{\mathds{E}_{\xi} \left[ f(x+\gamma e, \xi) \right] e}{r} \right] =\\
            	    = & \; \frac{d}{\gamma} \mathds{E}_{e} [\langle f(x+\gamma e)e, r \rangle] = \dotprod{\nabla f_\gamma(x)}{r}.	    
            	\end{split}
        	\end{equation}
        	Для второго слагаемого (\ref{eq_1_for_proof_noise_l2_randomization_one_point}) с предположением \ref{assumption:SS_2} получим
        	\begin{equation}
        	\label{eq_3_for_proof_noise_l2_randomization_one_point}
        	    \frac{d}{\gamma} \mathds{E}_{e} [\langle \delta(x+\gamma e)e, r \rangle] \geq  -\frac{d}{\gamma} \Delta \mathds{E}_{e} [ | \langle e, r \rangle | ].
        	\end{equation}
        	Используя уравнения (\ref{eq_2_for_proof_noise_l2_randomization_one_point}) и (\ref{eq_3_for_proof_noise_l2_randomization_one_point}) для уравнения (\ref{eq_1_for_proof_noise_l2_randomization_one_point}) получим утверждение леммы для $l_2$-рандомизации.
\end{itemize}
\end{proof}
Так как лемма \ref{lemma_3_proof_noise_lp_randomization} выполняется для одноточечного оракула, то следствия \ref{corollary_lemma_3_l1_randomization}-\ref{corollary_lemma_3_l2_randomization} и замечание \ref{remark_1} также выполняются для одноточечного оракула. Таким образом, теперь можем получить оценки параметров одноточечных безградиентных методов для $l_1$ и $l_2$-рандомизации, аналогичным образом, как и в п. \ref{subsection:Gradient_Free_Algorithms}, используя п. \ref{subsection:SS_4}. В теореме \ref{theorem_3} представлены оценки безградиентных методов выпуклой оптимизации, а в теореме \ref{theorem_4} оценки для седловых задач.
\begin{theorem}
    \label{theorem_3}
    Схема сглаживания из разд. \ref{section:Smooth_schemes}, применяемая к задаче \eqref{Convex_opt_problem}, обеспечивает сходимость следующих одноточечных безградиентных алгоритмов: Minibatch и Single-Machine Accelerated SGD \cite{Woodworth_2021}, Local-AC-CA \cite{Woodworth_2020} и FedAc \cite{Yuan_Ma_2020}. Другими словами, для достижения $\varepsilon$ точности решения задачи \eqref{Convex_opt_problem} необходимо проделать $NK$ итераций с максимально допустимым уровнем шума $\Delta$ и общим числом вызова безградиентного оракула $T$ в соответствии с выбранным методом и схемой сглаживания:
    \begin{itemize}
        \item Minibatch Accelerated SGD
        \begin{itemize}
            \item[i)] для $l_1$-рандомизации (\ref{grad_l1_randomization_one_point}):
                \begin{equation*}
                    \Delta = O\left( \frac{ \varepsilon^2 }{\sqrt{d} M_2 R} \right); 
                \end{equation*}
                \begin{equation*}
                    N = O\left( \frac{d^{1/4} \sqrt{M M_2} R}{\varepsilon} \right); \;\;\;
                    K = 1; \;\;\;
                    B = O\left( \frac{\kappa(p,d) M_2^2 G^2 R^2}{K N \varepsilon^4} \right);
                \end{equation*}
                \begin{equation*}
                    T = \tilde O \left(  \frac{\kappa(p,d) M_2^2 G^2 R^2}{ \varepsilon^4}   \right) = \begin{cases}
                    \tilde O \left(  \frac{d^2 M_2^2 G^2 R^2}{ \varepsilon^4}   \right), & p = 2 \;\;\; (q = 2),\\
                    \tilde O \left(  \frac{d M_2^2 G^2 R^2}{ \varepsilon^4}   \right), & p = 1 \;\;\; (q = \infty),
                    \end{cases}
                \end{equation*}
                
            \item[ii)] для $l_2$-рандомизации  (\ref{grad_l2_randomization_one_point}):
                \begin{equation*}
                    \Delta = O\left( \frac{ \varepsilon^2 }{\sqrt{d} M_2 R} \right); 
                \end{equation*}
                \begin{equation*}
                    N = O\left( \frac{d^{1/4} \sqrt{M M_2} R}{\varepsilon} \right); \;\;\;
                    K = 1; \;\;\;
                    B = O\left( \frac{\kappa(p,d) M_2^2 G^2 R^2}{K N \varepsilon^4} \right);
                \end{equation*}
                \begin{equation*}
                    T = \tilde O \left(  \frac{\kappa(p,d) M_2^2 G^2 R^2}{ \varepsilon^4}   \right) = \begin{cases}
                    \tilde O \left(  \frac{d^2 M_2^2 G^2 R^2}{ \varepsilon^4}   \right), & p = 2 \;\;\; (q = 2),\\
                    \tilde O \left(  \frac{d (\ln d) M_2^2 G^2 R^2}{ \varepsilon^4}   \right), & p = 1 \;\;\; (q = \infty),
                    \end{cases}
                \end{equation*}
        \end{itemize}
        
        \item Single-Machine Accelerated SGD
        \begin{itemize}
            \item[i)] для $l_1$-рандомизации (\ref{grad_l1_randomization_one_point}):
            \begin{equation*}
                \Delta = O\left( \frac{ \varepsilon^2 }{\sqrt{d} M_2 R} \right); 
                \end{equation*}
                \begin{equation*}
                N = 1; \;\;\;
                K = O\left( \frac{\kappa(q,d) d^2 G^2 R^2}{ \varepsilon^4} \right); \;\;\;
                B = 1;
            \end{equation*}
            \begin{equation*}
                T = \tilde O \left(  \frac{\kappa(q,d) d^2 G^2 R^2}{ \varepsilon^4}   \right) = \begin{cases}
                \tilde O \left(  \frac{d^2 G^2 R^2}{ \varepsilon^4}   \right), & p = 2 \;\;\; (q = 2),\\
                \tilde O \left(  \frac{d G^2 R^2}{ \varepsilon^4}   \right), & p = 1 \;\;\; (q = \infty),
                \end{cases}
            \end{equation*}
            
            \item[ii)] для $l_2$-рандомизации  (\ref{grad_l2_randomization_one_point}):
            \begin{equation*}
                \Delta = O\left( \frac{ \varepsilon^2 }{\sqrt{d} M_2 R} \right); 
                \end{equation*}
                \begin{equation*}
                N = 1; \;\;\;
                K = O\left( \frac{\kappa(q,d) d^2 G^2 R^2}{ \varepsilon^4} \right); \;\;\;
                B = 1;
            \end{equation*}
            \begin{equation*}
                T = \tilde O \left(  \frac{\kappa(q,d) d^2 G^2 R^2}{ \varepsilon^4}   \right) = \begin{cases}
                \tilde O \left(  \frac{d^2 G^2 R^2}{ \varepsilon^4}   \right), & p = 2 \;\;\; (q = 2),\\
                \tilde O \left(  \frac{d (\ln d) G^2 R^2}{ \varepsilon^4}   \right), & p = 1 \;\;\; (q = \infty),
                \end{cases}
            \end{equation*}
        \end{itemize}
        
        \item Local-AC-SA
        \begin{itemize}
            \item[i)] для $l_1$-рандомизации (\ref{grad_l1_randomization_one_point}):
            \begin{equation*}
                \Delta = O\left( \frac{ \varepsilon^2 }{\sqrt{d} M_2 R} \right); 
                \end{equation*}
                \begin{equation*}
                N = 1; \;\;\;
                K = O\left( \frac{d^{1/4} \sqrt{M M_2} R}{\varepsilon} \right); \;\;\;
                B = O\left( \frac{\kappa(q,d) d^2 G^2 R^2}{K N \varepsilon^4} \right);
            \end{equation*}
            \begin{equation*}
                T = \tilde O \left(  \frac{\kappa(q,d) d^2 G^2 R^2}{ \varepsilon^4}   \right) = \begin{cases}
                \tilde O \left(  \frac{d^2 G^2 R^2}{ \varepsilon^4}   \right), & p = 2 \;\;\; (q = 2),\\
                \tilde O \left(  \frac{d G^2 R^2}{ \varepsilon^4}   \right), & p = 1 \;\;\; (q = \infty),
                \end{cases}
            \end{equation*}
            
            \item[ii)] для $l_2$-рандомизации  (\ref{grad_l2_randomization_one_point}):
            \begin{equation*}
                \Delta = O\left( \frac{ \varepsilon^2 }{\sqrt{d} M_2 R} \right); 
                \end{equation*}
                \begin{equation*}
                N = 1; \;\;\;
                K = O\left( \frac{d^{1/4} \sqrt{M M_2} R}{\varepsilon} \right); \;\;\;
                B = O\left( \frac{\kappa(q,d) d^2 G^2 R^2}{K N \varepsilon^4} \right);
            \end{equation*}
            \begin{equation*}
                T = \tilde O \left(  \frac{\kappa(q,d) d^2 G^2 R^2}{ \varepsilon^4}   \right) = \begin{cases}
                \tilde O \left(  \frac{d^2 G^2 R^2}{ \varepsilon^4}   \right), & p = 2 \;\;\; (q = 2),\\
                \tilde O \left(  \frac{d (\ln d) G^2 R^2}{ \varepsilon^4}   \right), & p = 1 \;\;\; (q = \infty),
                \end{cases}
            \end{equation*}
        \end{itemize}
        
        \item Federated Accelerated SGD (FedAc)
        \begin{itemize}
            \item[i)] для $l_1$-рандомизации (\ref{grad_l1_randomization_one_point}):
            \begin{equation*}
                \Delta = O\left( \frac{ \varepsilon^2 }{\sqrt{d} M_2 R} \right); 
            \end{equation*}
            \begin{equation*}
                N = O\left( \frac{d^{1/6} (\kappa(q,d) M M_2)^{1/3} G^{2/3} R^{4/3}}{K^{1/3} \varepsilon^2} \right); \;\;\;  
                K = O\left( \frac{ \kappa(q,d) d^2 G^2 R^2}{B N \varepsilon^4} \right); \;\;\; 
                B = 1;
            \end{equation*}
            \begin{equation*}
                T = \tilde O \left(  \frac{\kappa(q,d) d^2 G^2 R^2}{ \varepsilon^4}   \right) = \begin{cases}
                \tilde O \left(  \frac{d^2 G^2 R^2}{ \varepsilon^4}   \right), & p = 2 \;\;\; (q = 2),\\
                \tilde O \left(  \frac{d G^2 R^2}{ \varepsilon^4}   \right), & p = 1 \;\;\; (q = \infty),
                \end{cases}
            \end{equation*}
            
            \item[ii)] для $l_2$-рандомизации  (\ref{grad_l2_randomization_one_point}):
            \begin{equation*}
                \Delta = O\left( \frac{ \varepsilon^2 }{\sqrt{d} M_2 R} \right); 
            \end{equation*}
            \begin{equation*}
                N = O\left( \frac{d^{1/6} (\kappa(q,d) M M_2)^{1/3} G^{2/3} R^{4/3}}{K^{1/3} \varepsilon^2} \right); \;\;\;  
                K = O\left( \frac{ \kappa(q,d) d^2 G^2 R^2}{B N \varepsilon^4} \right); \;\;\; 
                B = 1;
            \end{equation*}
            \begin{equation*}
                T = \tilde O \left(  \frac{\kappa(q,d) d^2 G^2 R^2}{ \varepsilon^4}   \right) = \begin{cases}
                \tilde O \left(  \frac{d^2 G^2 R^2}{ \varepsilon^4}   \right), & p = 2 \;\;\; (q = 2),\\
                \tilde O \left(  \frac{d (\ln d) G^2 R^2}{ \varepsilon^4}   \right), & p = 1 \;\;\; (q = \infty),
                \end{cases}
            \end{equation*}
            \end{itemize}
    \end{itemize}
\end{theorem}
\begin{proof}\renewcommand{\qedsymbol}{} Рассмотрим доказательство для каждой рандомизации и каждого метода отдельно

\underline{Для $l_1$-рандомизации имеем}:
    \begin{itemize}
        \item Minibatch Accelerated SGD
        
        Данный алгоритм после $N$ раундов связи дает скорость сходимости для $f_\gamma (x)$ (см. \cite{Woodworth_2021, Lan_2012}) в соответствии со следствием \ref{corollary_lemma_3_l1_randomization}:
        \begin{equation*}
            \mathds{E}[f_{\gamma}(x_{ag}^{N+1}) - f(x_{*})] \leq \frac{4L_{f_\gamma}R^2}{N^2} + \frac{4 \sigma R}{\sqrt{BKN}} + \frac{\sqrt{d} \Delta R}{\gamma},
        \end{equation*}
        где $x_{*}(\gamma) = \underset{x \in Q_\gamma}{\mathrm{argmin}}  f_{\gamma}(x)$.
        
        Если имеем  $\frac{\varepsilon}{2}$-точность для функции $f_{\gamma}(x)$ с $\gamma = \frac{ \sqrt{d} \varepsilon}{4 M_2}$ (из следствия \ref{gamma_parameter_clear}), то имеем $\varepsilon$-точность для функции $f(x)$:
        \begin{equation*}
            f_{\gamma}(x_{ag}^{N+1}) - f(x_{*}) \leq f_{\gamma}(x_{ag}^{N+1}) - f(x_{*}(\gamma)) \leq f_{\gamma}(x_{ag}^{N+1}) - f(x_{*}(\gamma)) + \frac{2}{\sqrt{d}} \gamma M_2  \leq \frac{\varepsilon}{2 } + \frac{\varepsilon}{2} =\varepsilon.
        \end{equation*}
        Чтобы была $\frac{\varepsilon}{2}$-точность для $f_{\gamma}(x) $ нужно
    	\begin{equation}
    	    \label{eq1_for_proof1:1_one_point_l1}
    	    \frac{\sqrt{d} \Delta R}{\gamma} \leq \frac{\varepsilon}{6},
    	\end{equation}
    	\begin{equation}
    	    \label{eq2_for_proof1:1_one_point_l1}
    	    \frac{4L_{f_\gamma}R^2}{ N^2} \leq \frac{\varepsilon}{6},
    	\end{equation}
    	
    	\begin{equation}
    	    \label{eq3_for_proof1:1_one_point_l1}
    	    \frac{ 4\sigma R}{\sqrt{B K N}} \leq \frac{\varepsilon}{6}.
    	\end{equation}

        Подставляя $L_{f_\gamma} =  \frac{d M}{\gamma}$ (из следствия \ref{Const_Lipshitz_grad}), где $\gamma = \frac{\sqrt{d} \varepsilon}{4 M_2}$ и $ \sigma^2 = \frac{4\kappa(p,d) M_2^2 G^2}{\varepsilon^2}$ (из следствия \ref{sigma_for_one_point}), в неравенства (\ref{eq1_for_proof1:1_one_point_l1})-(\ref{eq3_for_proof1:1_one_point_l1}), получим:
        \begin{equation*}
            \Delta \leq \frac{\gamma \varepsilon }{6 \sqrt{d} R} \,\,\, \Rightarrow \,\,\, \Delta \leq \frac{ \varepsilon^2 }{12 \sqrt{d} M_2 R} \,\,\, \Rightarrow
        \end{equation*}
        \begin{equation*}
            \Delta = O\left( \frac{ \varepsilon^2 }{\sqrt{d} M_2 R} \right)
        \end{equation*}
        уровень неточности,
        \begin{equation*}
            N^2 \geq \frac{48 \sqrt{d} M M_2 R^2}{\varepsilon^2} \Rightarrow N \geq \frac{4 \sqrt{3} d^{1/4} \sqrt{M M_2} R}{\varepsilon} \Rightarrow
        \end{equation*}
        \begin{equation*}
            \Rightarrow N = O\left( \frac{d^{1/4} \sqrt{M M_2} R}{\varepsilon} \right)
        \end{equation*}
        количество коммуникационных раундов,
        \begin{equation*}
            B \geq \frac{576 \sigma^2 R^2}{KN \varepsilon^2} \Rightarrow B \geq \frac{2304 \kappa(p,d) M_2^2 G^2 R^2}{K N \varepsilon^4} \Rightarrow
        \end{equation*}
        \begin{equation*}
            \Rightarrow B = O\left( \frac{\kappa(p,d) M_2^2 G^2 R^2}{K N \varepsilon^4} \right)
        \end{equation*}
        количество работающих параллельно машин и
        \begin{equation*}
             T = N \cdot K \cdot B = \frac{2304 \kappa(p,d) M_2^2 G^2 R^2}{ \varepsilon^4} \,\,\, \Rightarrow
        \end{equation*}
        \begin{equation*}
            \Rightarrow \,\,\, T = \tilde O \left(  \frac{\kappa(p,d) M_2^2 G^2 R^2}{ \varepsilon^4}   \right) = \begin{cases}
            \tilde O \left(  \frac{d^2 M_2^2 G^2 R^2}{ \varepsilon^4}   \right), & p = 2 \;\;\; (q = 2),\\
            \tilde O \left(  \frac{d M_2^2 G^2 R^2}{ \varepsilon^4}   \right), & p = 1 \;\;\; (q = \infty),
            \end{cases}
        \end{equation*}
        общее количество вызовов одноточечного безградиентного оракула;

        \item Single-Machine Accelerated SGD
        
        Данный алгоритм после $NK$ итераций дает скорость сходимости для $f_\gamma (x)$ (см. \cite{Woodworth_2021, Lan_2012}) в соответствии со следствием \ref{corollary_lemma_3_l1_randomization}:
        \begin{equation*}
            \mathds{E}[f_{\gamma}(x_{ag}^{NK+1}) - f(x_{*})] \leq \frac{4L_{f_\gamma}R^2}{N^2K^2} + \frac{4 \sigma R}{\sqrt{NK}} + \frac{\sqrt{d} \Delta R}{\gamma},
        \end{equation*}
        где $x_{*}(\gamma) = \underset{x \in Q_\gamma}{\mathrm{argmin}}  f_{\gamma}(x)$.
        
        Если имеем  $\frac{\varepsilon}{2}$-точность для функции $f_{\gamma}(x)$ с $\gamma = \frac{ \sqrt{d} \varepsilon}{4 M_2}$ (из следствия \ref{gamma_parameter_clear}), то имеем $\varepsilon$-точность для функции $f(x)$:
        \begin{equation*}
            f_{\gamma}(x_{ag}^{N+1}) - f(x_{*}) \leq f_{\gamma}(x_{ag}^{N+1}) - f(x_{*}(\gamma)) \leq f_{\gamma}(x_{ag}^{N+1}) - f(x_{*}(\gamma)) + \frac{2}{\sqrt{d}} \gamma M_2  \leq \frac{\varepsilon}{2 } + \frac{\varepsilon}{2} =\varepsilon.
        \end{equation*}
        Чтобы была $\frac{\varepsilon}{2}$-точность для $f_{\gamma}(x) $ нужно
    	\begin{equation}
    	    \label{eq1_for_proof1:2_one_point_l1}
    	    \frac{\sqrt{d} \Delta R}{\gamma} \leq \frac{\varepsilon}{6},
    	\end{equation}
    	\begin{equation}
    	    \label{eq2_for_proof1:2_one_point_l1}
    	    \frac{4L_{f_\gamma}R^2}{K^2 N^2} \leq \frac{\varepsilon}{6},
    	\end{equation}
    	
    	\begin{equation}
    	    \label{eq3_for_proof1:2_one_point_l1}
    	    \frac{ 4\sigma R}{\sqrt{K N}} \leq \frac{\varepsilon}{6}.
    	\end{equation}
    
        Подставляя $L_{f_\gamma} =  \frac{d M}{\gamma}$ (из следствия \ref{Const_Lipshitz_grad}), где $\gamma = \frac{\sqrt{d} \varepsilon}{4 M_2}$ и $ \sigma^2 = \frac{4\kappa(p,d) M_2^2 G^2}{\varepsilon^2}$ (из следствия \ref{sigma_for_one_point}), в неравенства (\ref{eq1_for_proof1:2_one_point_l1})-(\ref{eq3_for_proof1:2_one_point_l1}), получим
        \begin{equation*}
            \Delta \leq \frac{\gamma \varepsilon }{6 \sqrt{d} R} \,\,\, \Rightarrow \,\,\, \Delta \leq \frac{ \varepsilon^2 }{12 \sqrt{d} M_2 R} \,\,\, \Rightarrow
        \end{equation*}
        \begin{equation*}
            \Delta = O\left( \frac{ \varepsilon^2 }{\sqrt{d} M_2 R} \right)
        \end{equation*}
        уровень неточности,
        \begin{equation*}
            NK  \geq \frac{576 \sigma^2 R^2}{\varepsilon^2} \Rightarrow N K  \geq \frac{2304 \kappa(p,d) M_2^2 G^2 R^2}{ \varepsilon^4}.
        \end{equation*}
        Так как $N$ напрямую зависит от $K$, то количество коммуникачионных раундов можно взять $N = 1$, тогда
        \begin{equation*}
            NK = K = O \left( \frac{ \kappa(p,d) M_2^2 G^2 R^2}{ \varepsilon^4} \right)
        \end{equation*}
        количество локальных вызовов безградиентного оракула и
        \begin{equation*}
             T = N \cdot K \cdot B = \frac{2304 \kappa(p,d) M_2^2 G^2 R^2}{ \varepsilon^4} \,\,\, \Rightarrow
        \end{equation*}
        \begin{equation*}
            \Rightarrow \,\,\, T = \tilde O \left(  \frac{\kappa(p,d) M_2^2 G^2 R^2}{ \varepsilon^4}   \right) = \begin{cases}
            \tilde O \left(  \frac{d^2 M_2^2 G^2 R^2}{ \varepsilon^4}   \right), & p = 2 \;\;\; (q = 2),\\
            \tilde O \left(  \frac{d M_2^2 G^2 R^2}{ \varepsilon^4}   \right), & p = 1 \;\;\; (q = \infty),
            \end{cases}
        \end{equation*}
        общее количество вызовов одноточечного безградиентного оракула;
    
        \item Local-AC-SA
        
        Данный алгоритм после $N$ раундов связи дает скорость сходимости для $f_\gamma (x)$ (см. \cite{Woodworth_2020, Lan_2012}) в соответствии со следствием \ref{corollary_lemma_3_l1_randomization}:
        \begin{equation*}
            \mathds{E}[f_{\gamma}(x_{ag}^{N+1}) - f(x_{*})] \leq \frac{4L_{f_\gamma} R^2}{K^2 N^2} + \frac{4 \sigma R}{\sqrt{B K N}} + \frac{\sqrt{d} \Delta R}{\gamma},
        \end{equation*}
        где $x_{*}(\gamma) = \underset{x \in Q_\gamma}{\mathrm{argmin}}  f_{\gamma}(x)$.
        
        Если имеем  $\frac{\varepsilon}{2}$-точность для функции $f_{\gamma}(x)$ с $\gamma = \frac{ \sqrt{d} \varepsilon}{4 M_2}$ (из следствия \ref{gamma_parameter_clear}), то имеем $\varepsilon$-точность для функции $f(x)$:
        \begin{equation*}
            f_{\gamma}(x_{ag}^{N+1}) - f(x_{*}) \leq f_{\gamma}(x_{ag}^{N+1}) - f(x_{*}(\gamma)) \leq f_{\gamma}(x_{ag}^{N+1}) - f(x_{*}(\gamma)) + \frac{2}{\sqrt{d}} \gamma M_2  \leq \frac{\varepsilon}{2 } + \frac{\varepsilon}{2} =\varepsilon.
        \end{equation*}
        Чтобы была $\frac{\varepsilon}{2}$-точность для $f_{\gamma}(x) $ нужно
    	\begin{equation}
    	    \label{eq1_for_proof1:3_one_point_l1}
    	    \frac{\sqrt{d} \Delta R}{\gamma} \leq \frac{\varepsilon}{6},
    	\end{equation}
    	\begin{equation}
    	    \label{eq2_for_proof1:3_one_point_l1}
    	    \frac{4L_{f_\gamma}R^2}{K^2 N^2} \leq \frac{\varepsilon}{6},
    	\end{equation}
    	
    	\begin{equation}
    	    \label{eq3_for_proof1:3_one_point_l1}
    	    \frac{ 4\sigma R}{\sqrt{B K N}} \leq \frac{\varepsilon}{6}.
    	\end{equation}
	
       Подставляя $L_{f_\gamma} =  \frac{d M}{\gamma}$ (из следствия \ref{Const_Lipshitz_grad}), где $\gamma = \frac{\sqrt{d} \varepsilon}{4 M_2}$ и $ \sigma^2 = \frac{4\kappa(p,d) M_2^2 G^2}{\varepsilon^2}$ (из следствия \ref{sigma_for_one_point}), в неравенства (\ref{eq1_for_proof1:3_one_point_l1})-(\ref{eq3_for_proof1:3_one_point_l1}), получим
        \begin{equation*}
            \Delta \leq \frac{\gamma \varepsilon }{6 \sqrt{d} R} \,\,\, \Rightarrow \,\,\, \Delta \leq \frac{ \varepsilon^2 }{12 \sqrt{d} M_2 R} \,\,\, \Rightarrow
        \end{equation*}
        \begin{equation*}
            \Delta = O\left( \frac{ \varepsilon^2 }{\sqrt{d} M_2 R} \right)
        \end{equation*}
        уровень неточности,
        \begin{equation*}
            N^2 K^2 \geq \frac{48 \sqrt{d} M M_2 R^2}{\varepsilon^2} \Rightarrow NK \geq \frac{4 \sqrt{3} d^{1/4} \sqrt{M M_2} R}{\varepsilon}.
        \end{equation*}
        Так как $N$ напрямую зависит от $K$, то количество коммуникачионных раундов можно взять $N = 1$, тогда
        \begin{equation*}
            NK = K = O \left( \frac{d^{1/4} \sqrt{M M_2} R}{\varepsilon} \right)
        \end{equation*}
        количество локальных вызовов безградиентного оракула,
        \begin{equation*}
            B \geq \frac{576 \sigma^2 R^2}{KN \varepsilon^2} \Rightarrow B \geq \frac{2304 \kappa(p,d) M_2^2 G^2 R^2}{K N \varepsilon^4} \Rightarrow
        \end{equation*}
        \begin{equation*}
            \Rightarrow B = O\left( \frac{\kappa(p,d) M_2^2 G^2 R^2}{K N \varepsilon^4} \right)
        \end{equation*}
        количество работающих параллельно машин и
        \begin{equation*}
             T = N \cdot K \cdot B = \frac{2304 \kappa(p,d) M_2^2 G^2 R^2}{ \varepsilon^4} \,\,\, \Rightarrow
        \end{equation*}
        \begin{equation*}
            \Rightarrow \,\,\, T = \tilde O \left(  \frac{\kappa(p,d) M_2^2 G^2 R^2}{ \varepsilon^4}   \right) = \begin{cases}
            \tilde O \left(  \frac{d^2 M_2^2 G^2 R^2}{ \varepsilon^4}   \right), & p = 2 \;\;\; (q = 2),\\
            \tilde O \left(  \frac{d M_2^2 G^2 R^2}{ \varepsilon^4}   \right), & p = 1 \;\;\; (q = \infty),
            \end{cases}
        \end{equation*}
        общее количество вызовов одноточечного безградиентного оракула;
    
        \item Federated Accelerated SGD (FedAc)
        
        Данный алгоритм после $N$ раундов связи дает скорость сходимости для $f_\gamma (x)$ (см. \cite{Yuan_Ma_2020}) в соответствии со следствием \ref{corollary_lemma_3_l1_randomization}:
        \begin{equation*}
            \mathds{E}[f_{\gamma}(x_{ag}^{N+1}) - f(x_{*})] \leq \frac{L_{f_\gamma} R^2}{K N^2} + \frac{ \sigma R}{\sqrt{B K N}} + \min \left\{ \frac{L_{f_\gamma}^{1/3} \sigma^{2/3} R^{4/3}}{K^{1/3} N}, \frac{L_{f_\gamma}^{1/2} \sigma^{1/2} R^{3/2}}{K^{1/4} N} \right\} + \frac{\sqrt{d} \Delta R}{\gamma},
        \end{equation*}
        где $x_{*}(\gamma) = \underset{x \in Q_\gamma}{\mathrm{argmin}}  f_{\gamma}(x)$.
        
        Если имеем  $\frac{\varepsilon}{2}$-точность для функции $f_{\gamma}(x)$ с $\gamma = \frac{ \sqrt{d} \varepsilon}{4 M_2}$ (из следствия \ref{gamma_parameter_clear}), то имеем $\varepsilon$-точность для функции $f(x)$:
        \begin{equation*}
            f_{\gamma}(x_{ag}^{N+1}) - f(x_{*}) \leq f_{\gamma}(x_{ag}^{N+1}) - f(x_{*}(\gamma)) \leq f_{\gamma}(x_{ag}^{N+1}) - f(x_{*}(\gamma)) + \frac{2}{\sqrt{d}} \gamma M_2  \leq \frac{\varepsilon}{2 } + \frac{\varepsilon}{2} =\varepsilon.
        \end{equation*}
        Чтобы была $\frac{\varepsilon}{2}$-точность для $f_{\gamma}(x) $ нужно
    	\begin{equation}
    	    \label{eq1_for_proof1:4_one_point_l1}
    	    \frac{\sqrt{d} \Delta R}{\gamma} \leq \frac{\varepsilon}{8},
    	\end{equation}
    	\begin{equation}
    	    \label{eq2_for_proof1:4_one_point_l1}
    	    \frac{L_{f_\gamma}R^2}{K N^2} \leq \frac{\varepsilon}{8},
    	\end{equation}
    	
    	\begin{equation}
    	    \label{eq3_for_proof1:4_one_point_l1}
    	    \frac{ \sigma R}{\sqrt{B K N}} \leq \frac{\varepsilon}{8},
    	\end{equation}
    	\begin{equation}
    	    \label{eq4_for_proof1:4_one_point_l1}
    	    \min \left\{ \frac{L_{f_\gamma}^{1/3} \sigma^{2/3} R^{4/3}}{K^{1/3} N}, \frac{L_{f_\gamma}^{1/2} \sigma^{1/2} R^{3/2}}{K^{1/4} N} \right\} \leq \frac{\varepsilon}{8}.
    	\end{equation}
	
       Подставляя $L_{f_\gamma} =  \frac{d M}{\gamma}$ (из следствия \ref{Const_Lipshitz_grad}), где $\gamma = \frac{\sqrt{d} \varepsilon}{4 M_2}$ и $ \sigma^2 = \frac{4\kappa(p,d) M_2^2 G^2}{\varepsilon^2}$ (из следствия \ref{sigma_for_one_point}), в неравенства (\ref{eq1_for_proof1:4_one_point_l1})-(\ref{eq4_for_proof1:4_one_point_l1}), получим
        \begin{equation*}
            \Delta \leq \frac{\gamma \varepsilon }{8 \sqrt{d} R} \,\,\, \Rightarrow \,\,\, \Delta \leq \frac{ \varepsilon^2 }{16 \sqrt{d} M_2 R} \,\,\, \Rightarrow
        \end{equation*}
        \begin{equation*}
            \Delta = O\left( \frac{ \varepsilon^2 }{\sqrt{d} M_2 R} \right)
        \end{equation*}
        уровень неточности,
        \begin{equation*}
            NK \geq \frac{8^{1/2} K^{1/2} L^{1/2}_{f_\gamma}R}{ \varepsilon^{1/2}} \,\,\,  \Rightarrow \,\,\,  NK \geq \frac{16^{1/2} K^{1/2} d^{1/4} M^{1/2} M_2^{1/2} R}{ \varepsilon} \,\,\,  \Rightarrow \,\,\, NK \gtrsim \frac{ K^{1/2} }{ \varepsilon},
        \end{equation*}
        \begin{equation*}
            NK \geq \frac{64 \sigma^2 R^2}{B \varepsilon^2}  \,\,\, \Rightarrow \,\,\, NK \geq \frac{256 \kappa(p,d) M_2^2  G^2 R^2}{B \varepsilon^4} \,\,\,  \Rightarrow \,\,\, NK \gtrsim \frac{1}{B \varepsilon^4}
        \end{equation*}
        и
        \begin{equation*}
            NK \geq \min \left\{ \frac{K^{2/3} L_{f_\gamma}^{1/3} \sigma^{2/3} R^{4/3}}{\varepsilon}, \frac{K^{3/4} L_{f_\gamma}^{1/2} \sigma^{1/2} R^{3/2}}{ \varepsilon} \right\} \,\,\, \Rightarrow 
        \end{equation*}
        \begin{equation*}
            \Rightarrow \,\,\, NK \geq \min \left\{  \frac{16 K^{2/3} d^{1/6} (M M_2)^{1/3} \kappa(p,d)^{1/3} G^{2/3}  R^{4/3}}{ \varepsilon^{2}}, \frac{16 K^{3/4} d^{1/4} (M M_2)^{1/2} \kappa(p,d)^{1/4} G^{1/2} R^{3/2}}{ \varepsilon^{2}} \right\} \,\,\, \Rightarrow 
        \end{equation*}
        \begin{equation*}
            \Rightarrow \,\,\, NK \geq \frac{16 K^{2/3} d^{1/6} (M M_2)^{1/3} \kappa(p,d)^{1/3} G^{2/3}  R^{4/3}}{ \varepsilon^{2}} \,\,\,  \Rightarrow \,\,\, NK \gtrsim \frac{K^{2/3}}{ \varepsilon^{2}}.
        \end{equation*}
       Таким образом, наименьшее число коммуникационных раундов при условиии, что $NK \in [1, \varepsilon^{-4}]$ и $T \in [1, \varepsilon^{-4}]$, будет выглядить следующим образом:
         \begin{equation*}
            N \sim \frac{1}{\varepsilon}: \;\;\;\;\; N \geq \frac{8 L_{f_\gamma}^{1/3} \sigma^{2/3} R^{4/3}}{K^{1/3} \varepsilon} \Rightarrow \;\;\; N = O\left( \frac{d^{1/6} (\kappa(p,d) M M_2)^{1/3} G^{2/3} R^{4/3}}{K^{1/3} \varepsilon^2} \right)
        \end{equation*}
        тогда
        \begin{equation*}
            K \sim \frac{1}{\varepsilon^3}: \;\;\;\;\; K \geq \frac{ \sigma^2 R^2}{B N \varepsilon^2} \Rightarrow \;\;\; K = O\left( \frac{ \kappa(p,d) M_2^2 G^2 R^2}{B N \varepsilon^4} \right)
        \end{equation*}
        количество локальных вызовов безградиентного оракула, $B = 1$~--- количество работающих параллельно машин и
        \begin{equation*}
             T \sim \frac{1}{\varepsilon^4}: \;\;\;\;\; T = N \cdot K \cdot B = \tilde O \left( \frac{\kappa(p,d) M_2^2 G^2 R^2}{ \varepsilon^4}   \right) = \begin{cases}
            \tilde O \left(  \frac{d^2 M_2^2 G^2 R^2}{ \varepsilon^4}   \right), & p = 2 \;\;\; (q = 2),\\
            \tilde O \left(  \frac{d M_2^2 G^2 R^2}{ \varepsilon^4}   \right), & p = 1 \;\;\; (q = \infty),
            \end{cases}
        \end{equation*}
        общее количество вызовов одноточечного безградиентного оракула.
    \end{itemize}
\underline{Для $l_2$-рандомизации имеем}:
    \begin{itemize}
        \item Minibatch Accelerated SGD
        
        Данный алгоритм после $N$ раундов связи дает скорость сходимости для $f_\gamma (x)$ (см. \cite{Woodworth_2021, Lan_2012}) в соответствии со следствием \ref{corollary_lemma_3_l2_randomization}:
        \begin{equation*}
            \mathds{E}[f_{\gamma}(x_{ag}^{N+1}) - f(x_{*})] \leq \frac{4L_{f_\gamma}R^2}{N^2} + \frac{4 \sigma R}{\sqrt{BKN}} + \frac{\sqrt{d} \Delta R}{\gamma},
        \end{equation*}
        где $x_{*}(\gamma) = \underset{x \in Q_\gamma}{\mathrm{argmin}}  f_{\gamma}(x)$.
        
        Если имеем  $\frac{\varepsilon}{2}$-точность для функции $f_{\gamma}(x)$ с $\gamma = \frac{\varepsilon}{2 M_2}$ (из следствия \ref{gamma_parameter_clear}), то имеем $\varepsilon$-точность для функции $f(x)$:
        \begin{equation*}
            f_{\gamma}(x_{ag}^{N+1}) - f(x_{*}) \leq f_{\gamma}(x_{ag}^{N+1}) - f(x_{*}(\gamma)) \leq f_{\gamma}(x_{ag}^{N+1}) - f(x_{*}(\gamma)) + \gamma M_2  \leq \frac{\varepsilon}{2 } + \frac{\varepsilon}{2} =\varepsilon.
        \end{equation*}
        Чтобы была $\frac{\varepsilon}{2}$-точность для $f_{\gamma}(x) $ нужно
    	\begin{equation}
    	    \label{eq1_for_proof1:1_one_point}
    	    \frac{\sqrt{d} \Delta R}{\gamma} \leq \frac{\varepsilon}{6},
    	\end{equation}
    	\begin{equation}
    	    \label{eq2_for_proof1:1_one_point}
    	    \frac{4L_{f_\gamma}R^2}{ N^2} \leq \frac{\varepsilon}{6},
    	\end{equation}
    	
    	\begin{equation}
    	    \label{eq3_for_proof1:1_one_point}
    	    \frac{ 4\sigma R}{\sqrt{B K N}} \leq \frac{\varepsilon}{6}.
    	\end{equation}

        Подставляя $L_{f_\gamma} =  \frac{\sqrt{d} M}{\gamma}$ (из следствия \ref{Const_Lipshitz_grad}), где $\gamma = \frac{\varepsilon}{2 M_2}$ и $ \sigma^2 = \frac{4\kappa(p,d) M_2^2 G^2}{\varepsilon^2}$ (из следствия \ref{sigma_for_one_point}), в неравенства (\ref{eq1_for_proof1:1_one_point})-(\ref{eq3_for_proof1:1_one_point}), получим
        \begin{equation*}
            \Delta \leq \frac{\gamma \varepsilon }{6 \sqrt{d} R} \,\,\, \Rightarrow \,\,\, \Delta \leq \frac{ \varepsilon^2 }{12 \sqrt{d} M_2 R} \,\,\, \Rightarrow
        \end{equation*}
        \begin{equation*}
            \Delta = O\left( \frac{ \varepsilon^2 }{\sqrt{d} M_2 R} \right)
        \end{equation*}
        уровень неточности,
        \begin{equation*}
            N^2 \geq \frac{48 \sqrt{d} M M_2 R^2}{\varepsilon^2} \Rightarrow N \geq \frac{4 \sqrt{3} d^{1/4} \sqrt{M M_2} R}{\varepsilon} \Rightarrow
        \end{equation*}
        \begin{equation*}
            \Rightarrow N = O\left( \frac{d^{1/4} \sqrt{M M_2} R}{\varepsilon} \right)
        \end{equation*}
        количество коммуникационных раундов,
        \begin{equation*}
            B \geq \frac{576 \sigma^2 R^2}{KN \varepsilon^2} \Rightarrow B \geq \frac{2304 \kappa(p,d) M_2^2 G^2 R^2}{K N \varepsilon^4} \Rightarrow
        \end{equation*}
        \begin{equation*}
            \Rightarrow B = O\left( \frac{\kappa(p,d) M_2^2 G^2 R^2}{K N \varepsilon^4} \right)
        \end{equation*}
        количество работающих параллельно машин и
        \begin{equation*}
             T = N \cdot K \cdot B = \frac{2304 \kappa(p,d) M_2^2 G^2 R^2}{ \varepsilon^4} \,\,\, \Rightarrow
        \end{equation*}
        \begin{equation*}
            \Rightarrow \,\,\, T = \tilde O \left(  \frac{\kappa(p,d) M_2^2 G^2 R^2}{ \varepsilon^4}   \right) = \begin{cases}
            \tilde O \left(  \frac{d^2 M_2^2 G^2 R^2}{ \varepsilon^4}   \right), & p = 2 \;\;\; (q = 2),\\
            \tilde O \left(  \frac{d (\ln d) M_2^2 G^2 R^2}{ \varepsilon^4}   \right), & p = 1 \;\;\; (q = \infty),
            \end{cases}
        \end{equation*}
        общее количество вызовов одноточечного безградиентного оракула;

        \item Single-Machine Accelerated SGD
        
        Данный алгоритм после $NK$ итераций дает скорость сходимости для $f_\gamma (x)$ (см. \cite{Woodworth_2021, Lan_2012}) в соответствии со следствием \ref{corollary_lemma_3_l2_randomization}:
        \begin{equation*}
            \mathds{E}[f_{\gamma}(x_{ag}^{NK+1}) - f(x_{*})] \leq \frac{4L_{f_\gamma}R^2}{N^2K^2} + \frac{4 \sigma R}{\sqrt{NK}} + \frac{\sqrt{d} \Delta R}{\gamma},
        \end{equation*}
        где $x_{*}(\gamma) = \underset{x \in Q_\gamma}{\mathrm{argmin}}  f_{\gamma}(x)$.
        
        Если имеем  $\frac{\varepsilon}{2}$-точность для функции $f_{\gamma}(x)$ с $\gamma = \frac{\varepsilon}{2 M_2}$ (из следствия \ref{gamma_parameter_clear}), то имеем $\varepsilon$-точность для функции $f(x)$:
        \begin{equation*}
            f_{\gamma}(x_{ag}^{N+1}) - f(x_{*}) \leq f_{\gamma}(x_{ag}^{N+1}) - f(x_{*}(\gamma)) \leq f_{\gamma}(x_{ag}^{N+1}) - f(x_{*}(\gamma)) + \gamma M_2  \leq \frac{\varepsilon}{2 } + \frac{\varepsilon}{2} =\varepsilon.
        \end{equation*}
        Чтобы была $\frac{\varepsilon}{2}$-точность для $f_{\gamma}(x) $ нужно
    	\begin{equation}
    	    \label{eq1_for_proof1:2_one_point}
    	    \frac{\sqrt{d} \Delta R}{\gamma} \leq \frac{\varepsilon}{6},
    	\end{equation}
    	\begin{equation}
    	    \label{eq2_for_proof1:2_one_point}
    	    \frac{4L_{f_\gamma}R^2}{K^2 N^2} \leq \frac{\varepsilon}{6},
    	\end{equation}
    	
    	\begin{equation}
    	    \label{eq3_for_proof1:2_one_point}
    	    \frac{ 4\sigma R}{\sqrt{K N}} \leq \frac{\varepsilon}{6}.
    	\end{equation}
    
        Подставляя $L_{f_\gamma} =  \frac{\sqrt{d} M}{\gamma}$ (из следствия \ref{Const_Lipshitz_grad}), где $\gamma = \frac{\varepsilon}{2 M_2}$ и $ \sigma^2 = \frac{4\kappa(p,d) M_2^2 G^2}{\varepsilon^2}$ (из следствия \ref{sigma_for_one_point}), в неравенства (\ref{eq1_for_proof1:2_one_point})-(\ref{eq3_for_proof1:2_one_point}), получим
        \begin{equation*}
            \Delta \leq \frac{\gamma \varepsilon }{6 \sqrt{d} R} \,\,\, \Rightarrow \,\,\, \Delta \leq \frac{ \varepsilon^2 }{12 \sqrt{d} M_2 R} \,\,\, \Rightarrow
        \end{equation*}
        \begin{equation*}
            \Delta = O\left( \frac{ \varepsilon^2 }{\sqrt{d} M_2 R} \right)
        \end{equation*}
        уровень неточности,
        \begin{equation*}
            NK  \geq \frac{576 \sigma^2 R^2}{\varepsilon^2} \Rightarrow N K  \geq \frac{2304 \kappa(p,d) M_2^2 G^2 R^2}{ \varepsilon^4}.
        \end{equation*}
        Так как $N$ напрямую зависит от $K$, то количество коммуникачионных раундов можно взять $N = 1$, тогда
        \begin{equation*}
            NK = K = O \left( \frac{ \kappa(p,d) M_2^2 G^2 R^2}{ \varepsilon^4} \right)
        \end{equation*}
        количество локальных вызовов безградиентного оракула и
        \begin{equation*}
             T = N \cdot K \cdot B = \frac{2304 \kappa(p,d) M_2^2 G^2 R^2}{ \varepsilon^4} \,\,\, \Rightarrow
        \end{equation*}
        \begin{equation*}
            \Rightarrow \,\,\, T = \tilde O \left(  \frac{\kappa(p,d) M_2^2 G^2 R^2}{ \varepsilon^4}   \right) = \begin{cases}
            \tilde O \left(  \frac{d^2 M_2^2 G^2 R^2}{ \varepsilon^4}   \right), & p = 2 \;\;\; (q = 2),\\
            \tilde O \left(  \frac{d (\ln d) M_2^2 G^2 R^2}{ \varepsilon^4}   \right), & p = 1 \;\;\; (q = \infty),
            \end{cases}
        \end{equation*}
        общее количество вызовов одноточечного безградиентного оракула;
    
        \item Local-AC-SA
        
        Данный алгоритм после $N$ раундов связи дает скорость сходимости для $f_\gamma (x)$ (см. \cite{Woodworth_2020, Lan_2012}) в соответствии со следствием \ref{corollary_lemma_3_l2_randomization}:
        \begin{equation*}
            \mathds{E}[f_{\gamma}(x_{ag}^{N+1}) - f(x_{*})] \leq \frac{4L_{f_\gamma} R^2}{K^2 N^2} + \frac{4 \sigma R}{\sqrt{B K N}} + \frac{\sqrt{d} \Delta R}{\gamma},
        \end{equation*}
        где $x_{*}(\gamma) = \underset{x \in Q_\gamma}{\mathrm{argmin}}  f_{\gamma}(x)$.
        
        Если имеем  $\frac{\varepsilon}{2}$-точность для функции $f_{\gamma}(x)$ с $\gamma = \frac{\varepsilon}{2 M_2}$ (из следствия \ref{gamma_parameter_clear}), то имеем $\varepsilon$-точность для функции $f(x)$:
        \begin{equation*}
            f_{\gamma}(x_{ag}^{N+1}) - f(x_{*}) \leq f_{\gamma}(x_{ag}^{N+1}) - f(x_{*}(\gamma)) \leq f_{\gamma}(x_{ag}^{N+1}) - f(x_{*}(\gamma)) + \gamma M_2  \leq \frac{\varepsilon}{2 } + \frac{\varepsilon}{2} =\varepsilon.
        \end{equation*}
        Чтобы была $\frac{\varepsilon}{2}$-точность для $f_{\gamma}(x) $ нужно
    	\begin{equation}
    	    \label{eq1_for_proof1:3_one_point}
    	    \frac{\sqrt{d} \Delta R}{\gamma} \leq \frac{\varepsilon}{6},
    	\end{equation}
    	\begin{equation}
    	    \label{eq2_for_proof1:3_one_point}
    	    \frac{4L_{f_\gamma}R^2}{K^2 N^2} \leq \frac{\varepsilon}{6},
    	\end{equation}
    	
    	\begin{equation}
    	    \label{eq3_for_proof1:3_one_point}
    	    \frac{ 4\sigma R}{\sqrt{B K N}} \leq \frac{\varepsilon}{6}.
    	\end{equation}
	
       Подставляя $L_{f_\gamma} =  \frac{\sqrt{d} M}{\gamma}$ (из следствия \ref{Const_Lipshitz_grad}), где $\gamma = \frac{\varepsilon}{2 M_2}$ и $ \sigma^2 = \frac{4\kappa(p,d) M_2^2 G^2}{\varepsilon^2}$ (из следствия \ref{sigma_for_one_point}), в неравенства (\ref{eq1_for_proof1:3_one_point})-(\ref{eq3_for_proof1:3_one_point}), получим:
        \begin{equation*}
            \Delta \leq \frac{\gamma \varepsilon }{6 \sqrt{d} R} \,\,\, \Rightarrow \,\,\, \Delta \leq \frac{ \varepsilon^2 }{12 \sqrt{d} M_2 R} \,\,\, \Rightarrow
        \end{equation*}
        \begin{equation*}
            \Delta = O\left( \frac{ \varepsilon^2 }{\sqrt{d} M_2 R} \right)
        \end{equation*}
        уровень неточности,
        \begin{equation*}
            N^2 K^2 \geq \frac{48 \sqrt{d} M M_2 R^2}{\varepsilon^2} \Rightarrow NK \geq \frac{4 \sqrt{3} d^{1/4} \sqrt{M M_2} R}{\varepsilon}.
        \end{equation*}
        Так как $N$ напрямую зависит от $K$, то количество коммуникачионных раундов можно взять $N = 1$, тогда
        \begin{equation*}
            NK = K = O \left( \frac{d^{1/4} \sqrt{M M_2} R}{\varepsilon} \right)
        \end{equation*}
        количество локальных вызовов безградиентного оракула,
        \begin{equation*}
            B \geq \frac{576 \sigma^2 R^2}{KN \varepsilon^2} \Rightarrow B \geq \frac{2304 \kappa(p,d) M_2^2 G^2 R^2}{K N \varepsilon^4} \Rightarrow
        \end{equation*}
        \begin{equation*}
            \Rightarrow B = O\left( \frac{\kappa(p,d) M_2^2 G^2 R^2}{K N \varepsilon^4} \right)
        \end{equation*}
        количество работающих параллельно машин и
        \begin{equation*}
             T = N \cdot K \cdot B = \frac{2304 \kappa(p,d) M_2^2 G^2 R^2}{ \varepsilon^4} \,\,\, \Rightarrow
        \end{equation*}
        \begin{equation*}
            \Rightarrow \,\,\, T = \tilde O \left(  \frac{\kappa(p,d) M_2^2 G^2 R^2}{ \varepsilon^4}   \right) = \begin{cases}
            \tilde O \left(  \frac{d^2 M_2^2 G^2 R^2}{ \varepsilon^4}   \right), & p = 2 \;\;\; (q = 2),\\
            \tilde O \left(  \frac{d (\ln d) M_2^2 G^2 R^2}{ \varepsilon^4}   \right), & p = 1 \;\;\; (q = \infty),
            \end{cases}
        \end{equation*}
        общее количество вызовов одноточечного безградиентного оракула;
    
        \item Federated Accelerated SGD (FedAc)
        
        Данный алгоритм после $N$ раундов связи дает скорость сходимости для $f_\gamma (x)$ (см. \cite{Yuan_Ma_2020}) в соответствии со следствием \ref{corollary_lemma_3_l2_randomization}:
        \begin{equation*}
            \mathds{E}[f_{\gamma}(x_{ag}^{N+1}) - f(x_{*})] \leq \frac{L_{f_\gamma} R^2}{K N^2} + \frac{ \sigma R}{\sqrt{B K N}} + \min \left\{ \frac{L_{f_\gamma}^{1/3} \sigma^{2/3} R^{4/3}}{K^{1/3} N}, \frac{L_{f_\gamma}^{1/2} \sigma^{1/2} R^{3/2}}{K^{1/4} N} \right\} + \frac{\sqrt{d} \Delta R}{\gamma},
        \end{equation*}
        где $x_{*}(\gamma) = \underset{x \in Q_\gamma}{\mathrm{argmin}}  f_{\gamma}(x)$.
        
        Если имеем  $\frac{\varepsilon}{2}$-точность для функции $f_{\gamma}(x)$ с $\gamma = \frac{\varepsilon}{2 M_2}$ (из следствия \ref{gamma_parameter_clear}), то имеем $\varepsilon$-точность для функции $f(x)$:
        \begin{equation*}
            f_{\gamma}(x_{ag}^{N+1}) - f(x_{*}) \leq f_{\gamma}(x_{ag}^{N+1}) - f(x_{*}(\gamma)) \leq f_{\gamma}(x_{ag}^{N+1}) - f(x_{*}(\gamma)) + \gamma M_2  \leq \frac{\varepsilon}{2 } + \frac{\varepsilon}{2} =\varepsilon.
        \end{equation*}
        Чтобы была $\frac{\varepsilon}{2}$-точность для $f_{\gamma}(x) $ нужно
    	\begin{equation}
    	    \label{eq1_for_proof1:4_one_point}
    	    \frac{\sqrt{d} \Delta R}{\gamma} \leq \frac{\varepsilon}{8},
    	\end{equation}
    	\begin{equation}
    	    \label{eq2_for_proof1:4_one_point}
    	    \frac{L_{f_\gamma}R^2}{K N^2} \leq \frac{\varepsilon}{8},
    	\end{equation}
    	
    	\begin{equation}
    	    \label{eq3_for_proof1:4_one_point}
    	    \frac{ \sigma R}{\sqrt{B K N}} \leq \frac{\varepsilon}{8},
    	\end{equation}
    	\begin{equation}
    	    \label{eq4_for_proof1:4_one_point}
    	    \min \left\{ \frac{L_{f_\gamma}^{1/3} \sigma^{2/3} R^{4/3}}{K^{1/3} N}, \frac{L_{f_\gamma}^{1/2} \sigma^{1/2} R^{3/2}}{K^{1/4} N} \right\} \leq \frac{\varepsilon}{8}.
    	\end{equation}
	
       Подставляя $L_{f_\gamma} =  \frac{\sqrt{d} M}{\gamma}$ (из следствия \ref{Const_Lipshitz_grad}), где $\gamma = \frac{\varepsilon}{2 M_2}$ и $ \sigma^2 = \frac{4\kappa(p,d) M_2^2 G^2}{\varepsilon^2}$ (из следствия \ref{sigma_for_one_point}), в неравенства (\ref{eq1_for_proof1:4_one_point})-(\ref{eq4_for_proof1:4_one_point}), получим
        \begin{equation*}
            \Delta \leq \frac{\gamma \varepsilon }{8 \sqrt{d} R} \,\,\, \Rightarrow \,\,\, \Delta \leq \frac{ \varepsilon^2 }{16 \sqrt{d} M_2 R} \,\,\, \Rightarrow
        \end{equation*}
        \begin{equation*}
            \Delta = O\left( \frac{ \varepsilon^2 }{\sqrt{d} M_2 R} \right)
        \end{equation*}
        уровень неточности,
        \begin{equation*}
            NK \geq \frac{8^{1/2} K^{1/2} L^{1/2}_{f_\gamma}R}{ \varepsilon^{1/2}} \,\,\,  \Rightarrow \,\,\,  NK \geq \frac{16^{1/2} K^{1/2} d^{1/4} M^{1/2} M_2^{1/2} R}{ \varepsilon} \,\,\,  \Rightarrow \,\,\, NK \gtrsim \frac{ K^{1/2} }{ \varepsilon},
        \end{equation*}
        \begin{equation*}
            NK \geq \frac{64 \sigma^2 R^2}{B \varepsilon^2}  \,\,\, \Rightarrow \,\,\, NK \geq \frac{256 \kappa(p,d) M_2^2  G^2 R^2}{B \varepsilon^4} \,\,\,  \Rightarrow \,\,\, NK \gtrsim \frac{1}{B \varepsilon^4}
        \end{equation*}
        и
        \begin{equation*}
            NK \geq \min \left\{ \frac{K^{2/3} L_{f_\gamma}^{1/3} \sigma^{2/3} R^{4/3}}{\varepsilon}, \frac{K^{3/4} L_{f_\gamma}^{1/2} \sigma^{1/2} R^{3/2}}{ \varepsilon} \right\} \,\,\, \Rightarrow 
        \end{equation*}
        \begin{equation*}
            \Rightarrow \,\,\, NK \geq \min \left\{  \frac{16 K^{2/3} d^{1/6} (M M_2)^{1/3} \kappa(p,d)^{1/3} G^{2/3}  R^{4/3}}{ \varepsilon^{2}}, \frac{16 K^{3/4} d^{1/4} (M M_2)^{1/2} \kappa(p,d)^{1/4} G^{1/2} R^{3/2}}{ \varepsilon^{2}} \right\} \,\,\, \Rightarrow 
        \end{equation*}
        \begin{equation*}
            \Rightarrow \,\,\, NK \geq \frac{16 K^{2/3} d^{1/6} (M M_2)^{1/3} \kappa(p,d)^{1/3} G^{2/3}  R^{4/3}}{ \varepsilon^{2}} \,\,\,  \Rightarrow \,\,\, NK \gtrsim \frac{K^{2/3}}{ \varepsilon^{2}}.
        \end{equation*}
       Таким образом, наименьшее число коммуникационных раундов при условиии, что $NK \in [1, \varepsilon^{-4}]$ и $T \in [1, \varepsilon^{-4}]$, будет выглядить следующим образом:
         \begin{equation*}
            N \sim \frac{1}{\varepsilon}: \;\;\;\;\; N \geq \frac{8 L_{f_\gamma}^{1/3} \sigma^{2/3} R^{4/3}}{K^{1/3} \varepsilon} \Rightarrow \;\;\; N = O\left( \frac{d^{1/6} (\kappa(p,d) M M_2)^{1/3} G^{2/3} R^{4/3}}{K^{1/3} \varepsilon^2} \right)
        \end{equation*}
        тогда
        \begin{equation*}
            K \sim \frac{1}{\varepsilon^3}: \;\;\;\;\; K \geq \frac{ \sigma^2 R^2}{B N \varepsilon^2} \Rightarrow \;\;\; K = O\left( \frac{ \kappa(p,d) M_2^2 G^2 R^2}{B N \varepsilon^4} \right)
        \end{equation*}
        количество локальных вызовов безградиентного оракула, $B = 1$~--- количество работающих параллельно машин и
        \begin{equation*}
             T \sim \frac{1}{\varepsilon^4}: \;\;\;\;\; T = N \cdot K \cdot B = \tilde O \left( \frac{\kappa(p,d) M_2^2 G^2 R^2}{ \varepsilon^4}   \right) = \begin{cases}
            \tilde O \left(  \frac{d^2 M_2^2 G^2 R^2}{ \varepsilon^4}   \right), & p = 2 \;\;\; (q = 2),\\
            \tilde O \left(  \frac{d (\ln d) M_2^2 G^2 R^2}{ \varepsilon^4}   \right), & p = 1 \;\;\; (q = \infty),
            \end{cases}
        \end{equation*}
        общее количество вызовов одноточечного безградиентного оракула.
    \end{itemize}
\end{proof}

\begin{theorem}
\label{theorem_4}
    Схема сглаживания из разд. \ref{section:Smooth_schemes}, применяемая к седловой задаче (см. замечание \ref{remark_convex_concave}), обеспечивает сходимость следующих одноточечных безградиентных алгоритмов: Minibatch SMP и Single-Machine SMP из Приложения \ref{Upper_bound}. Другими словами, для достижения $\varepsilon$ точности решения седловой задачи (см. замечание \ref{remark_convex_concave}) необходимо проделать $NK$ итераций с максимально допустимым уровнем шума $\Delta$ и общим числом вызова безградиентного оракула $T$ в соответствии с выбранным методом и схемой сглаживания:
    \begin{itemize}
        \item Minibatch SMP
        \begin{itemize}
            \item[i)] для $l_1$-рандомизации (\ref{grad_l1_randomization_one_point}):
            \begin{equation*}
                \Delta = O\left( \frac{ \varepsilon^2 }{\sqrt{d} M_2 R} \right); 
            \end{equation*}
            \begin{equation*}
                N = 1; \;\;\;
                K = 1; \;\;\;
                B = O\left( \frac{\kappa(p,d) M_2^2 G^2 R^2}{ \varepsilon^4} \right);
            \end{equation*}
            \begin{equation*}
                T = \tilde O \left(  \frac{\kappa(p,d) M_2^2 G^2 R^2}{ \varepsilon^4}   \right) = \begin{cases}
                \tilde O \left(  \frac{d^2 M_2^2 G^2 R^2}{ \varepsilon^4}   \right), & p = 2 \;\;\; (q = 2),\\
                \tilde O \left(  \frac{d M_2^2 G^2 R^2}{ \varepsilon^4}   \right), & p = 1 \;\;\; (q = \infty),
                \end{cases}
            \end{equation*}
        
        \item[ii)] для $l_2$-рандомизации (\ref{grad_l2_randomization_one_point}):
        \begin{equation*}
                \Delta = O\left( \frac{ \varepsilon^2 }{\sqrt{d} M_2 R} \right); 
            \end{equation*}
            \begin{equation*}
                N = 1; \;\;\;
                K = 1; \;\;\;
                B = O\left( \frac{\kappa(p,d) M_2^2 G^2 R^2}{ \varepsilon^4} \right);
            \end{equation*}
            \begin{equation*}
                T = \tilde O \left(  \frac{\kappa(p,d) M_2^2 G^2 R^2}{ \varepsilon^4}   \right) = \begin{cases}
                \tilde O \left(  \frac{d^2 M_2^2 G^2 R^2}{ \varepsilon^4}   \right), & p = 2 \;\;\; (q = 2),\\
                \tilde O \left(  \frac{d (\ln d) M_2^2 G^2 R^2}{ \varepsilon^4}   \right), & p = 1 \;\;\; (q = \infty),
                \end{cases}
            \end{equation*}
        \end{itemize}
        
        \item Single-Machine SMP
        
        \item[i)] для $l_1$-рандомизации (\ref{grad_l1_randomization_one_point}):
            \begin{equation*}
                \Delta = O\left( \frac{ \varepsilon^2 }{\sqrt{d} M_2 R} \right); 
            \end{equation*}
            \begin{equation*}
                N = 1; \;\;\;
                K = O\left( \frac{\kappa(p,d) M_2^2 G^2 R^2}{ \varepsilon^4} \right); \;\;\;
                B = 1;
            \end{equation*}
            \begin{equation*}
                T = \tilde O \left(  \frac{\kappa(p,d) M_2^2 G^2 R^2}{ \varepsilon^4}   \right) = \begin{cases}
                \tilde O \left(  \frac{d^2 M_2^2 G^2 R^2}{ \varepsilon^4}   \right), & p = 2 \;\;\; (q = 2),\\
                \tilde O \left(  \frac{d M_2^2 G^2 R^2}{ \varepsilon^4}   \right), & p = 1 \;\;\; (q = \infty),
                \end{cases}
            \end{equation*}
            
        \item[ii)] для $l_2$-рандомизации (\ref{grad_l2_randomization_one_point}):
            \begin{equation*}
                \Delta = O\left( \frac{ \varepsilon^2 }{\sqrt{d} M_2 R} \right); 
            \end{equation*}
            \begin{equation*}
                N = 1; \;\;\;
                K = O\left( \frac{\kappa(p,d) M_2^2 G^2 R^2}{ \varepsilon^4} \right); \;\;\;
                B = 1;
            \end{equation*}
            \begin{equation*}
                T = \tilde O \left(  \frac{\kappa(p,d) M_2^2 G^2 R^2}{ \varepsilon^4}   \right) = \begin{cases}
                \tilde O \left(  \frac{d^2 M_2^2 G^2 R^2}{ \varepsilon^4}   \right), & p = 2 \;\;\; (q = 2),\\
                \tilde O \left(  \frac{d (\ln d) M_2^2 G^2 R^2}{ \varepsilon^4}   \right), & p = 1 \;\;\; (q = \infty),
                \end{cases}
            \end{equation*}
    \end{itemize}
\end{theorem}
\begin{proof}\renewcommand{\qedsymbol}{} Рассмотрим доказательство для каждой рандомизации и каждого метода отдельно

\underline{Для $l_1$-рандомизации имеем}:
\begin{itemize}
    \item Minibatch SMP
        
        Данный алгоритм после $N$ раундов связи дает скорость сходимости для $f_\gamma (x)$ (см. следствие \ref{Optima_alg_VI}) в соответствии с замечанием \ref{remark_1}:
        \begin{equation*}
            \mathds{E}[f_{\gamma}(z_{N}) ] \leq  \max \left\{ \frac{7}{4} \frac{L R^2}{N}, 7 \frac{ \sigma R}{\sqrt{BKN}}  \right\} + \frac{\sqrt{d} \Delta R}{\gamma}.
        \end{equation*}
        
        Если имеем  $\frac{\varepsilon}{2}$-точность для функции $f_{\gamma}(z)$ с $\gamma = \frac{\sqrt{d} \varepsilon}{4 M_2}$ (из следствия \ref{gamma_parameter_clear}), то имеем $\varepsilon$-точность для функции $f(z)$:
        \begin{equation*}
             f(z_{N}) \leq f_{\gamma}(z_{N})  + \frac{2}{\sqrt{d}} \gamma M_2  \leq \frac{\varepsilon}{2 } + \frac{\varepsilon}{2} =\varepsilon.
        \end{equation*}
        Чтобы была $\frac{\varepsilon}{2}$-точность для $f_{\gamma}(z) $ нужно
    	\begin{equation}
    	    \label{eq1_for_proof2:1_one_point_l1}
    	    \frac{\sqrt{d} \Delta R}{\gamma} \leq \frac{\varepsilon}{4},
    	\end{equation}
    	\begin{equation}
    	    \label{eq2_for_proof2:1_one_point_l1}
    	    \max \left\{ \frac{7}{4} \frac{L R^2}{N}, 7 \frac{ \sigma R}{\sqrt{BKN}}  \right\} \leq \frac{\varepsilon}{4}.
    	\end{equation}
    
       Подставляя $L_{f_\gamma} =  \frac{d M}{\gamma}$ (из следствия \ref{Const_Lipshitz_grad}), где $\gamma = \frac{\varepsilon}{4 M_2}$ и $ \sigma^2 = \frac{4\kappa(p,d) M_2^2 G^2}{\varepsilon^2}$ (из следствия \ref{sigma_for_one_point}), в неравенства (\ref{eq1_for_proof2:1_one_point_l1}) и (\ref{eq2_for_proof2:1_one_point_l1}), получим
        \begin{equation*}
            \Delta \leq \frac{\gamma \varepsilon }{4 \sqrt{d} R} \,\,\, \Rightarrow \,\,\, \Delta \leq \frac{ \varepsilon^2 }{8 \sqrt{d} M_2 R} \,\,\, \Rightarrow
        \end{equation*}
        \begin{equation*}
            \Delta = O\left( \frac{ \varepsilon^2 }{\sqrt{d} M_2 R} \right)
        \end{equation*}
        уровень неточности,
        \begin{equation*}
            N \geq \frac{784 \sigma^2 R^2}{BK \varepsilon^2} \Rightarrow N \geq \frac{3136 \kappa(p,d) M_2^2 G^2 R^2}{BK \varepsilon^4} \Rightarrow
        \end{equation*}
        Так как $K = 1$, а $N$ напрямую зависит от $B$, то количество коммуникационных раундов можно взять $N = 1$, тогда
        \begin{equation*}
            \Rightarrow B = O\left( \frac{\kappa(p,d) M_2^2 G^2 R^2}{K N \varepsilon^4} \right)
        \end{equation*}
        количество работающих параллельно машин и
        \begin{equation*}
             T = N \cdot K \cdot B = \frac{2304 \kappa(p,d) M_2^2 G^2 R^2}{ \varepsilon^4} \,\,\, \Rightarrow
        \end{equation*}
        \begin{equation*}
            \Rightarrow \,\,\, T = \tilde O \left(  \frac{\kappa(p,d) M_2^2 G^2 R^2}{ \varepsilon^4}   \right) = \begin{cases}
            \tilde O \left(  \frac{d^2 M_2^2 G^2 R^2}{ \varepsilon^4}   \right), & p = 2 \;\;\; (q = 2),\\
            \tilde O \left(  \frac{d M_2^2 G^2 R^2}{ \varepsilon^4}   \right), & p = 1 \;\;\; (q = \infty),
            \end{cases}
        \end{equation*}
        общее количество вызовов одноточечного безградиентного оракула;
    
        \item Single-Machine SMP
        
        Данный алгоритм после $NK$ итераций дает скорость сходимости для $f_\gamma (x)$ (см. следствие \ref{Optima_alg_VI}) в соответствии с замечанием \ref{remark_1}:
        \begin{equation*}
            \mathds{E}[f_{\gamma}(z_{NK}) ] \leq  \max \left\{  \frac{7}{4} \frac{L R^2}{KN}, 7 \frac{\sigma R}{\sqrt{KN}} \right\} + \frac{\sqrt{d} \Delta R}{\gamma}.
        \end{equation*}
        
        Если имеем  $\frac{\varepsilon}{2}$-точность для функции $f_{\gamma}(z)$ с $\gamma = \frac{\sqrt{d} \varepsilon}{4 M_2}$ (из следствия \ref{gamma_parameter_clear}), то имеем $\varepsilon$-точность для функции $f(z)$:
        \begin{equation*}
             f(z_{N}) \leq f_{\gamma}(z_{N})  + \frac{2}{\sqrt{d}} \gamma M_2  \leq \frac{\varepsilon}{2 } + \frac{\varepsilon}{2} =\varepsilon.
        \end{equation*}
        Чтобы была $\frac{\varepsilon}{2}$-точность для $f_{\gamma}(z) $ нужно
    	\begin{equation}
    	    \label{eq1_for_proof2:2_one_point_l1}
    	    \frac{\sqrt{d} \Delta R}{\gamma} \leq \frac{\varepsilon}{4},
    	\end{equation}
    	\begin{equation}
    	    \label{eq2_for_proof2:2_one_point_l1}
    	   \max \left\{  \frac{7}{4} \frac{L R^2}{KN}, 7 \frac{\sigma R}{\sqrt{KN}} \right\} \leq \frac{\varepsilon}{4}.
    	\end{equation}
	
        Подставляя $L_{f_\gamma} =  \frac{\sqrt{d} M}{\gamma}$ (из следствия \ref{Const_Lipshitz_grad}), где $\gamma = \frac{\varepsilon}{4 M_2}$ и $ \sigma^2 = \frac{4\kappa(p,d) M_2^2 G^2}{\varepsilon^2}$ (из следствия \ref{sigma_for_one_point}), в неравенства (\ref{eq1_for_proof2:2_one_point_l1}) и (\ref{eq2_for_proof2:2_one_point_l1}), получим
        \begin{equation*}
            \Delta \leq \frac{\gamma \varepsilon }{4 \sqrt{d} R} \,\,\, \Rightarrow \,\,\, \Delta \leq \frac{ \varepsilon^2 }{8 \sqrt{d} M_2 R} \,\,\, \Rightarrow
        \end{equation*}
        \begin{equation*}
            \Delta = O\left( \frac{ \varepsilon^2 }{\sqrt{d} M_2 R} \right)
        \end{equation*}
        уровень неточности,
        \begin{equation*}
            N \geq \frac{784 \sigma^2 R^2}{K \varepsilon^2} \Rightarrow N \geq \frac{3136 \kappa(p,d) M_2^2 G^2 R^2}{K \varepsilon^4} \Rightarrow
        \end{equation*}
        Так как $N$ напрямую зависит от $K$, то количество коммуникационных раундов можно взять $N = 1$, тогда
        \begin{equation*}
            \Rightarrow K = O\left( \frac{\kappa(p,d) M_2^2 G^2 R^2}{K N \varepsilon^4} \right)
        \end{equation*}
        количество локальных вызовов безградиентного оракула и
        \begin{equation*}
             T = N \cdot K \cdot B = \frac{2304 \kappa(p,d) M_2^2 G^2 R^2}{ \varepsilon^4} \,\,\, \Rightarrow
        \end{equation*}
        \begin{equation*}
            \Rightarrow \,\,\, T = \tilde O \left(  \frac{\kappa(p,d) M_2^2 G^2 R^2}{ \varepsilon^4}   \right) = \begin{cases}
            \tilde O \left(  \frac{d^2 M_2^2 G^2 R^2}{ \varepsilon^4}   \right), & p = 2 \;\;\; (q = 2),\\
            \tilde O \left(  \frac{d M_2^2 G^2 R^2}{ \varepsilon^4}   \right), & p = 1 \;\;\; (q = \infty),
            \end{cases}
        \end{equation*}
        общее количество вызовов одноточечного безградиентного оракула.
    \end{itemize}

\underline{Для $l_2$-рандомизации имеем}:
    \begin{itemize}
    \item Minibatch SMP
        
        Данный алгоритм после $N$ раундов связи дает скорость сходимости для $f_\gamma (x)$ (см. следствие \ref{Optima_alg_VI}) в соответствии с замечанием \ref{remark_1}:
        \begin{equation*}
            \mathds{E}[f_{\gamma}(z_{N}) ] \leq  \max \left\{ \frac{7}{4} \frac{L R^2}{N}, 7 \frac{ \sigma R}{\sqrt{BKN}}  \right\} + \frac{\sqrt{d} \Delta R}{\gamma}.
        \end{equation*}
        
        Если имеем  $\frac{\varepsilon}{2}$-точность для функции $f_{\gamma}(z)$ с $\gamma = \frac{\varepsilon}{2 M_2}$ (из следствия \ref{gamma_parameter_clear}), то имеем $\varepsilon$-точность для функции $f(z)$:
        \begin{equation*}
             f(z_{N}) \leq f_{\gamma}(z_{N})  + \gamma M_2  \leq \frac{\varepsilon}{2 } + \frac{\varepsilon}{2} =\varepsilon.
        \end{equation*}
        Чтобы была $\frac{\varepsilon}{2}$-точность для $f_{\gamma}(z) $ нужно
    	\begin{equation}
    	    \label{eq1_for_proof2:1_one_point}
    	    \frac{\sqrt{d} \Delta R}{\gamma} \leq \frac{\varepsilon}{4},
    	\end{equation}
    	\begin{equation}
    	    \label{eq2_for_proof2:1_one_point}
    	    \max \left\{ \frac{7}{4} \frac{L R^2}{N}, 7 \frac{ \sigma R}{\sqrt{BKN}}  \right\} \leq \frac{\varepsilon}{4}.
    	\end{equation}
    
       Подставляя $L_{f_\gamma} =  \frac{\sqrt{d} M}{\gamma}$ (из следствия \ref{Const_Lipshitz_grad}), где $\gamma = \frac{\varepsilon}{2 M_2}$ и $ \sigma^2 = \frac{4\kappa(p,d) M_2^2 G^2}{\varepsilon^2}$ (из следствия \ref{sigma_for_one_point}), в неравенства (\ref{eq1_for_proof2:1_one_point}) и (\ref{eq2_for_proof2:1_one_point}), получим
        \begin{equation*}
            \Delta \leq \frac{\gamma \varepsilon }{4 \sqrt{d} R} \,\,\, \Rightarrow \,\,\, \Delta \leq \frac{ \varepsilon^2 }{8 \sqrt{d} M_2 R} \,\,\, \Rightarrow
        \end{equation*}
        \begin{equation*}
            \Delta = O\left( \frac{ \varepsilon^2 }{\sqrt{d} M_2 R} \right)
        \end{equation*}
        уровень неточности,
        \begin{equation*}
            N \geq \frac{784 \sigma^2 R^2}{BK \varepsilon^2} \Rightarrow N \geq \frac{3136 \kappa(p,d) M_2^2 G^2 R^2}{BK \varepsilon^4} \Rightarrow
        \end{equation*}
        Так как $K = 1$, а $N$ напрямую зависит от $B$, то количество коммуникационных раундов можно взять $N = 1$, тогда
        \begin{equation*}
            \Rightarrow B = O\left( \frac{\kappa(p,d) M_2^2 G^2 R^2}{K N \varepsilon^4} \right)
        \end{equation*}
        количество работающих параллельно машин и
        \begin{equation*}
             T = N \cdot K \cdot B = \frac{2304 \kappa(p,d) M_2^2 G^2 R^2}{ \varepsilon^4} \,\,\, \Rightarrow
        \end{equation*}
        \begin{equation*}
            \Rightarrow \,\,\, T = \tilde O \left(  \frac{\kappa(p,d) M_2^2 G^2 R^2}{ \varepsilon^4}   \right) = \begin{cases}
            \tilde O \left(  \frac{d^2 M_2^2 G^2 R^2}{ \varepsilon^4}   \right), & p = 2 \;\;\; (q = 2),\\
            \tilde O \left(  \frac{d (\ln d) M_2^2 G^2 R^2}{ \varepsilon^4}   \right), & p = 1 \;\;\; (q = \infty),
            \end{cases}
        \end{equation*}
        общее количество вызовов одноточечного безградиентного оракула;
    
        \item Single-Machine SMP
        
        Данный алгоритм после $NK$ итераций дает скорость сходимости для $f_\gamma (x)$ (см. следствие \ref{Optima_alg_VI}) в соответствии с замечанием \ref{remark_1}:
        \begin{equation*}
            \mathds{E}[f_{\gamma}(z_{NK}) ] \leq  \max \left\{  \frac{7}{4} \frac{L R^2}{KN}, 7 \frac{\sigma R}{\sqrt{KN}} \right\} + \frac{\sqrt{d} \Delta R}{\gamma}.
        \end{equation*}
        
        Если имеем  $\frac{\varepsilon}{2}$-точность для функции $f_{\gamma}(z)$ с $\gamma = \frac{\varepsilon}{2 M_2}$ (из следствия \ref{gamma_parameter_clear}), то имеем $\varepsilon$-точность для функции $f(z)$:
        \begin{equation*}
             f(z_{N}) \leq f_{\gamma}(z_{N})  + \gamma M_2  \leq \frac{\varepsilon}{2 } + \frac{\varepsilon}{2} =\varepsilon.
        \end{equation*}
        Чтобы была $\frac{\varepsilon}{2}$-точность для $f_{\gamma}(z) $ нужно
    	\begin{equation}
    	    \label{eq1_for_proof2:2_one_point}
    	    \frac{\sqrt{d} \Delta R}{\gamma} \leq \frac{\varepsilon}{4},
    	\end{equation}
    	\begin{equation}
    	    \label{eq2_for_proof2:2_one_point}
    	   \max \left\{  \frac{7}{4} \frac{L R^2}{KN}, 7 \frac{\sigma R}{\sqrt{KN}} \right\} \leq \frac{\varepsilon}{4}.
    	\end{equation}
	
        Подставляя $L_{f_\gamma} =  \frac{\sqrt{d} M}{\gamma}$ (из следствия \ref{Const_Lipshitz_grad}), где $\gamma = \frac{\varepsilon}{2 M_2}$ и $ \sigma^2 = \frac{4\kappa(p,d) M_2^2 G^2}{\varepsilon^2}$ (из следствия \ref{sigma_for_one_point}), в неравенства (\ref{eq1_for_proof2:2_one_point}) и (\ref{eq2_for_proof2:2_one_point}), получим:
        \begin{equation*}
            \Delta \leq \frac{\gamma \varepsilon }{4 \sqrt{d} R} \,\,\, \Rightarrow \,\,\, \Delta \leq \frac{ \varepsilon^2 }{8 \sqrt{d} M_2 R} \,\,\, \Rightarrow
        \end{equation*}
        \begin{equation*}
            \Delta = O\left( \frac{ \varepsilon^2 }{\sqrt{d} M_2 R} \right)
        \end{equation*}
        уровень неточности,
        \begin{equation*}
            N \geq \frac{784 \sigma^2 R^2}{K \varepsilon^2} \Rightarrow N \geq \frac{3136 \kappa(p,d) M_2^2 G^2 R^2}{K \varepsilon^4} \Rightarrow
        \end{equation*}
        Так как $N$ напрямую зависит от $K$, то количество коммуникационных раундов можно взять $N = 1$, тогда
        \begin{equation*}
            \Rightarrow K = O\left( \frac{\kappa(p,d) M_2^2 G^2 R^2}{K N \varepsilon^4} \right)
        \end{equation*}
        количество локальных вызовов безградиентного оракула и
        \begin{equation*}
             T = N \cdot K \cdot B = \frac{2304 \kappa(p,d) M_2^2 G^2 R^2}{ \varepsilon^4} \,\,\, \Rightarrow
        \end{equation*}
        \begin{equation*}
            \Rightarrow \,\,\, T = \tilde O \left(  \frac{\kappa(p,d) M_2^2 G^2 R^2}{ \varepsilon^4}   \right) = \begin{cases}
            \tilde O \left(  \frac{d^2 M_2^2 G^2 R^2}{ \varepsilon^4}   \right), & p = 2 \;\;\; (q = 2),\\
            \tilde O \left(  \frac{d (\ln d) M_2^2 G^2 R^2}{ \varepsilon^4}   \right), & p = 1 \;\;\; (q = \infty),
            \end{cases}
        \end{equation*}
        общее количество вызовов одноточечного безградиентного оракула.
    \end{itemize}
\end{proof}



\end{document}